\documentclass[a4paper, oneside, 11pt]{amsart}

\usepackage[english]{babel}
\usepackage{mathtools}
\usepackage{mathrsfs}
\usepackage[utf8]{inputenc}
\mathtoolsset{showonlyrefs}
\usepackage{enumerate}
\usepackage[dvipsnames]{xcolor}
\usepackage{nicefrac}
\usepackage{graphicx}
\usepackage[colorlinks=true, allcolors=blue]{hyperref}
\usepackage[utf8]{inputenc}
\usepackage[T1]{fontenc}
\usepackage{dsfont}
\usepackage{latexsym}
\usepackage{marvosym}
\usepackage{color}
\usepackage{yhmath}

\usepackage{fullpage}
\usepackage{verbatim}
\usepackage{cprotect}
\usepackage[vlined]{algorithm2e}
\SetAlgoCaptionSeparator{:}

\usepackage{empheq}
\usepackage{amssymb}
\usepackage{amsmath}
\usepackage{amsfonts}
\usepackage{amsthm}
\usepackage{titletoc}
\usepackage{caption}
\usepackage{cancel}
\usepackage{svg}
\usepackage{subfig} 
\usepackage{listings}
\usepackage{tikz}
\usetikzlibrary{decorations.text}
\usetikzlibrary{patterns}
\usetikzlibrary{decorations.pathreplacing}
\usetikzlibrary{calc,decorations.markings}
\usepackage{stmaryrd}
\usepackage{scalerel}
\hypersetup{
	colorlinks=true,
	breaklinks=true, 
	urlcolor= blue,
	linkcolor= blue,
	citecolor=blue,
}

\newtheorem{definition}{Definition}
\newtheorem{prop}[definition]{Proposition}
\newtheorem{theorem}[definition]{Theorem}
\newtheorem{corr}[definition]{Corollary}
\newtheorem{lemma}[definition]{Lemma}

\numberwithin{definition}{section}
\newtheorem{conjecture}{Conjecture}

\makeatletter
\newcommand*\bigcdot{\mathpalette\bigcdot@{.5}}
\newcommand*\bigcdot@[2]{\mathbin{\vcenter{\hbox{\scalebox{#2}{$\m@th#1\bullet$}}}}}
\makeatother

\newcommand{\f}[2]{\text{\rm For}_{#2}(#1)}
\newcommand{\df}[2]{\text{\rm For}_{#2}(#1)}
\newcommand{\edf}[1]{\text{\rm ExFor}(#1)}
\newcommand{\kedf}[2]{\text{\rm ExFor}_{#2}(#1)}
\newcommand{\Za}{(\mathbb{Z}^{2})^\ast}
\newcommand{\Z}{\mathbb{Z}}
\newcommand{\V}{\mathbb{V}}
\newcommand{\PP}{\mathbb{P}}
\newcommand{\E}{\mathcal{E}}
\newcommand{\N}{\mathbb{N}}

\newcommand{\norm}[1]{\left \lVert  #1 \right \rVert}

\usetikzlibrary{shapes.misc}
\tikzset{cross/.style={cross out, draw=black, minimum size=10*(#1-\pgflinewidth), inner sep=0pt, outer sep=0pt},
	cross/.default={1pt}}

\newcommand{\dualempty}{%
		\ensuremath{%
			\mathchoice{%
				\dualemptybase{1.0}{1}%
			}{%
				\dualemptybase{0.6}{0.65}%
			}{%
				\dualemptybase{0.5}{-0.05}%
			}{%
				\dualemptybase{0.3}{-0.05}%
			}%
		}%
	}
	
	\newcommand{\dualemptybase}[2]{%
		\begin{tikzpicture}[baseline=#2em,scale=#1]
			\node[inner sep=0, outer sep=0] at (-0.2,0) {\hspace{2pt}};
			\node[inner sep=0, outer sep=0] at (1.2,0) {\hspace{2pt}};
			\draw (0,0) -- (0,1) -- (1,1) -- (1,0) -- (0,0);
			\draw[fill=red, color=red] (0.5,0.5) circle  (0.075);fill=red
		\end{tikzpicture}%
	}

	\newcommand{\dualopen}{%
		\ensuremath{%
			\mathchoice{%
				\dualopenbase{1.0}{1}%
			}{%
				\dualopenbase{0.6}{0.65}%
			}{%
				\dualopenbase{0.5}{-0.05}%
			}{%
				\dualopenbase{0.3}{-0.05}%
			}%
		}%
	}
	
	\newcommand{\dualopenbase}[2]{%
		\begin{tikzpicture}[baseline=#2em,scale=#1]
			\node[inner sep=0, outer sep=0] at (-0.2,0) {\hspace{2pt}};
			\draw (0,0) -- (0,1) -- (1,1) -- (1,0) -- (0,0);
			\draw[fill=red, color=red] (0.5,0.5) circle  (0.075);fill=red
			\draw[color=red] (0.5,0.5) -- (1,0.5);
			\draw[->,blue,very thick] (1,0) -- (1,1); 
			\path[pattern=north west lines, pattern color=blue] (0.9,0) rectangle (1,1);
			\node[inner sep=0, outer sep=0] at (1.2,0.5) {\hspace{2pt}};
		\end{tikzpicture}%
	}

	\newcommand{\dualopentwo}{%
		\ensuremath{%
			\mathchoice{%
				\dualopenbasetwo{1.0}{1}%
			}{%
				\dualopenbasetwo{0.6}{0.65}%
			}{%
				\dualopenbasetwo{0.5}{-0.05}%
			}{%
				\dualopenbasetwo{0.3}{-0.05}%
			}%
		}%
	}
	\newcommand{\dualopenbasetwo}[2]{%
			\begin{tikzpicture}[baseline=#2em,scale=#1]
				\node[inner sep=0, outer sep=0] at (-0.2,0) {\hspace{2pt}};
				\node[inner sep=0, outer sep=0] at (1.2,0) {\hspace{2pt}};
				\draw (0,0) -- (0,1) -- (1,1) -- (1,0) -- (0,0);
				\draw[fill=red, color=red] (0.5,0.5) circle  (0.075);fill=red
				\draw[color=red, very thick ,opacity= 0.4] (0.5,0.5) -- (0.5,0);
				\draw[color=red, very thick ,opacity= 0.4] (0.5,0.5) -- (0,0.5);
                \begin{scope}[rotate around={180:(0.5,0.5)}]
				\draw[->,blue,very thick] (1,0) -- (1,1); 
				\path[pattern=north west lines, pattern color=blue] (0.9,0) rectangle (1,1);
			\end{scope}
				\begin{scope}[rotate around={270:(0.5,0.5)}]
					\draw[->,blue,very thick] (1,0) -- (1,1); 
					\path[pattern=north west lines, pattern color=blue] (0.9,0) rectangle (1,1);
				\end{scope}
		\end{tikzpicture}
	}
	\newcommand{\dualopenthree}{%
		\ensuremath{%
			\mathchoice{%
				\dualopenbasethree{1.0}{1}%
			}{%
				\dualopenbasethree{0.6}{0.65}%
			}{%
				\dualopenbasethree{0.5}{-0.05}%
			}{%
				\dualopenbasethree{0.3}{-0.05}%
			}%
		}%
	}
	\newcommand{\dualopenbasethree}[2]{%
		\begin{tikzpicture}[baseline=#2em,scale=#1]
			\node[inner sep=0, outer sep=0] at (-0.2,0) {\hspace{2pt}};
			\node[inner sep=0, outer sep=0] at (1.2,0) {\hspace{2pt}};
			\draw (0,0) -- (0,1) -- (1,1) -- (1,0) -- (0,0);
			\draw[fill=red, color=red] (0.5,0.5) circle  (0.075);fill=red
			\draw[color=red, very thick ,opacity= 0.4] (0.5,0.5) -- (1,0.5);
			\draw[color=red, very thick ,opacity= 0.4] (0.5,0.5) -- (0,0.5);
			\draw[color=red, very thick ,opacity= 0.4] (0.5,0.5) -- (0.5,0);
			\draw[->,blue,very thick] (1,0) -- (1,1); 
			\path[pattern=north west lines, pattern color=blue] (0.9,0) rectangle (1,1);
			\begin{scope}[rotate around={180:(0.5,0.5)}]
				\draw[->,blue,very thick] (1,0) -- (1,1); 
				\path[pattern=north west lines, pattern color=blue] (0.9,0) rectangle (1,1);
			\end{scope}
			\begin{scope}[rotate around={270:(0.5,0.5)}]
				\draw[->,blue,very thick] (1,0) -- (1,1); 
				\path[pattern=north west lines, pattern color=blue] (0.9,0) rectangle (1,1);
			\end{scope}
		\end{tikzpicture}
	}

\newcommand{\pattern}{
\scalebox{0.4}{
\begin{tikzpicture}
    \draw[step=1cm,gray,very thin] (0,0) grid (2,1);
    \draw[blue, line width=1mm] (0,0) -- (0,1) -- (1,1) -- (1,0) -- (2,0) -- (2,1);
\end{tikzpicture}}\, }

\newcommand{\patterntwo}{
\scalebox{0.1}{
\begin{tikzpicture}
    \draw[step=1cm, gray,line width=0.05mm] (0,0) grid (2,1);
    \draw[blue, line width=2mm] (0,0) -- (0,1) -- (1,1) -- (1,0) -- (2,0) -- (2,1);
\end{tikzpicture}}\, }
\newcommand{\A}[1]{\mathcal{A}^{\patterntwo}_{#1}}

\title{The planar lattice two-neighbor graph percolates}
\date{\today}

\author{David Coupier} 
\address{IMT Nord Europe, Institut Mines-T\'el\'ecom, Univ.\ Lille, F-59000 Lille, France}
\email{david.coupier@imt-nord-europe.fr}
\author{Beno\^{i}t Henry}
\address{IMT Nord Europe, Institut Mines-T\'el\'ecom, Univ.\ Lille, F-59000 Lille, France}
\email{benoit.henry@imt-nord-europe.fr}
\author{Benedikt Jahnel}
\address{Institut f\"ur Mathematische Stochastik, Technische Universit\"at Braunschweig\\
and Weierstrass Institute for Applied Analysis and Stochastics,
Berlin, Germany 
}
\email{benedikt.jahnel@tu-braunschweig.de}
\author{Jonas K\"oppl}
\address{Weierstrass Institute for Applied Analysis and Stochastics, Berlin, Germany
}
\email{jonas.koeppl@wias-berlin.de}


\begin{document}

\begin{abstract}
The $k$-neighbor graph is a directed percolation model on the hypercubic lattice $\Z^d$ in which each vertex independently picks exactly $k$ of its $2d$ nearest neighbors at random, and we open directed edges towards those. We prove that the $2$-neighbor graph percolates on $\Z^2$, i.e., that the origin is connected to infinity with positive probability. The proof rests on duality, an exploration algorithm, a comparison to i.i.d.~bond percolation under constraints as well as enhancement arguments. As a byproduct, we show that i.i.d.~bond percolation with forbidden local patterns has a strictly larger percolation threshold than $1/2$. Additionally, our main result provides further evidence that, in low dimensions, less variability is beneficial for percolation. 
\end{abstract}

\maketitle

\section{Introduction}

Percolation models have seen a tremendous interest in the last decades, in part because they have wide-ranging applications, for example, in statistical physics, communication networks, or mathematical epidemiology. On the other hand, they are mathematically intriguing in part because the models and related questions are often relatively easy to state, but rigorous answers are hard to obtain. For the case of i.i.d.~bond and site percolation on the hypercubic lattice, many aspects of the model are now very well understood (see, for example,~\cite{grimmett1999percolation,bollobas2006percolation}) but many variants of the classical paradigmatic models, in particular directed-percolation models, still remain to be explored and provide a large number of unsolved problems.  

Let us mention a few variants of percolation models on the lattice from the recent and not so recent literature. For example, in {\em $AB$-percolation}, vertices are assigned one of two possible colors independently and at random and the edge is declared open if and only if the vertices at the end carry different colors~\cite{wierman1989ab}.
In the {\em constraint-degree percolation model} every edge independently tries to open at a random time in $[0,1]$, however, this is only successful if, at that time, both of its end vertices have degree at most $k-1$, see~\cite{de2020constrained}. A wide literature focuses on {\em inhomogeneous percolation} in which a (possibly random) subset of edges are opened still independently but with a different parameter than other edges, introducing dependence in the model (see~\cite{zhang1994note,iliev2015phase,10.1214/11-AOP720,de2022approximation,newman1997percolation}).

In the domain of directed percolation, a non-trivial model can be analyzed where, in the two-dimensional square lattice, every horizontal (respectively, vertical) edge independently chooses to be oriented towards the left or right (respectively towards up or down), see~\cite{grimmett2001infinite}. %~\cite[Section 12.8]{grimmett1999percolation}. 
Still considering directed lattice percolation, \cite{holmes2014degenerate} proposes a general framework for random fields driven by i.i.d.~random variables on the sites (not on the bonds), which they call a {\em degenerate random environment}. Every vertex draws its outgoing open edges from a distribution on the set of incident edges. Indeed, this framework includes, for example, the {\em compass model}, i.i.d.~site percolation, classical oriented site percolation, or the {\em (half-) orthant model}, see also~\cite{Holmes,beekenkamp2021sharpness,10.1214/20-AOP1476,beaton2024chemical}. However, even though many different degenerate random environments are closely related and many share the same connectivity features, so far, to the best of our knowledge, general percolation statements are very limited, see~\cite[Lemma 2.1 \& 2.2]{holmes2014degenerate} for some basic observations in that respect. Hence, percolation for degenerate random environments has to be studied on a case-by-case bases and this has been done, for example, for a number of models where only two distinct sets of outgoing edges are possible (such as the orthant model). 

\medskip
We contribute to this line of research by considering the {\em $k$-neighbor graph} on $\Z^d$, which can be viewed as a degenerate random environment according to~\cite{holmes2014degenerate}, in which every vertex chooses, independently of each other, precisely $k$ of its $2d$ incident bonds as open outgoing edges uniformly at random, see Figure~\ref{fig:config_and_forward_cluster} for an illustration of the case $k=d=2$. 

\begin{figure}[ht]
\centering
\includegraphics[width=.9\textwidth]{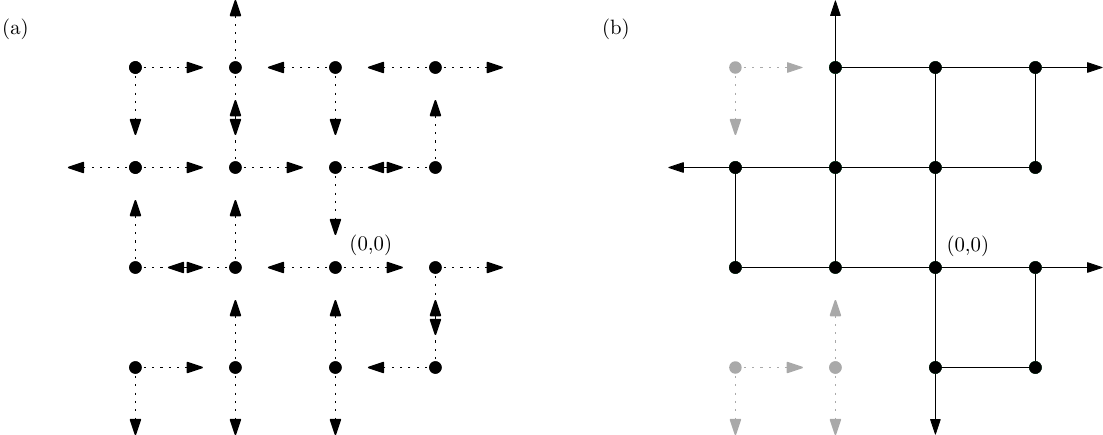}
\caption{\label{fig:config_and_forward_cluster}(a) A possible local configuration of the planar $2$-neighbor model and (b) the resulting forward cluster of the origin. Note that, since the model is directed, the relation $x \rightsquigarrow y$ is {\em not} reflexive.}
\end{figure}

The original motivation for studying the $k$-neighbor graph on $\Z^d$ comes from a continuum percolation model initially introduced in~\cite{HaggMee96}. In that work, the authors consider an homogeneous Poisson point process on $\mathbb{R}^d$ (say with intensity $1$) and draw undirected edges from each Poisson point to its $k\in \N$ nearest neighbors (w.r.t.~the Euclidean distance). They then show the existence of a non-trivial critical integer $k_c = k_c(d)$, depending on the dimension, such that for any $k \geq k_c$, the model percolates while it does not percolate for any $k<k_c$. The case of directed edges is also of interest and the same statement can be proved leading to a critical integer $k_c^{\to}(d)$. However, stating accurate theoretical bounds for the critical integers $k_c(d)$ or $k_c^\to(d)$ appears to be a very difficult task, especially for the most important dimensions $d \in \{2,3\}$, see~\cite{balister2013percolation}. This is the reason why the $k$-neighbor graph on $\Z^d$ was introduced in~\cite{KJLT} as a discrete counterpart to the continuum model from~\cite{HaggMee96}.

In~\cite{KJLT}, several criteria are derived for the existence and absence of infinite clusters, depending on $k$ and $d$, for both cases of directed and undirected edges. One of the intriguing questions that was left open concerned the directed case in $\Z^2$. Whereas the $k$-neighbor graph on $\Z^2$ clearly percolates when $k=3$ and does not percolate when $k=1$, the intermediate case with $k=2$ outgoing edges was conjectured to percolate, a claim strongly supported by simulations\footnote{See also \url{https://bennhenry.github.io/NeighPerc/}.}. Indeed, the set of vertices that can be reached from the origin in the $2$-neighbor graph seems to cover almost the whole lattice $\Z^2$. In this manuscript, we prove this conjecture in Theorem~\ref{thm:upperBound}.

\medskip
Let us finally mention that our model lies in the subclass of degenerate-random-environment models that feature a {\em degree constraint}. Degree-constraint models have been studied extensively also in the context of statistical mechanics where however the random field is not i.i.d.~but rather given by Gibbs measures of various types, see for example~\cite{holroyd2021constrained,grimmett20171,kenyon2006dimers,grimmett2010random}.
Our main result (Theorem~\ref{thm:upperBound}) can then be interpreted in the sense that a fixed degree is beneficial for percolation compared to the classical i.i.d.~(directed) bond percolation where an expected outdegree given by two is insufficient for percolation (see Lemma~\ref{lem:Critical-iid}). In a broader sense, we thus provide some further evidence for the hypothesis that less variability is good for percolation, at least in low dimensions. For comparison, let us mention here the associated discussion about the critical intensity for percolation in the Poisson--Boolean model with i.i.d.~radii and its dependence on the radius distribution. At least in dimensions two and three there is strong numerical evidence that, indeed, deterministic radii minimize the critical intensity; see~\cite{quintanilla2007asymmetry,gouere2014percolation}. In high dimensions this is however not true, see~\cite{gouere2016nonoptimality}. 
 
\medskip

The manuscript is organized as follows. In Section~\ref{sec_setting}, we define the {\em $2dp$-neighbor graph} on $\Z^d$, for $p \in [0,1]$, which is a continuous-parameter version of the $k$-neighbor graph previously mentioned. In particular, on $\Z^2$, the continuous $2$-neighbor graph coincides with the $2$-neighbor graph. The {\em percolation probability} $\theta_p(d)$ that the origin is connected to infinity in the $2dp$-neighbor graph on $\Z^d$ and the associated {\em critical parameter} $p_c(d)$ are introduced in Section~\ref{sec_results}, followed by a lower bound for $p_c(d)$ (Lemma~\ref{lem_low1st}) and an upper bound for $p_c(2)$ (Theorem~\ref{thm:upperBound}, our main result). In Section~\ref{sec_discussion}, the way that lack of variability for the number of outgoing edges promotes percolation is discussed through two extra directed percolation models, namely the {\em independent} and {\em all-or-none directed percolation models}, whose critical parameters are compared to $p_c(2)$. The main ingredients of the proof of Theorem~\ref{thm:upperBound} are explained in Section~\ref{sec_ingredients} while its complete proof is provided in  Sections~\ref{sec:block-construction}--\ref{sect:enhancement}. In Section~\ref{sec_upperbound_cor} we give the proof for a statement similar to our main Theorem~\ref{thm:upperBound} stating that percolation occurs also for a related model, namely the {\em directed-corner model} (Theorem~\ref{thm_upperbound_cor}). Finally, in the Appendix, we provide the proofs of all remaining lemmas.

\section{Setting, main results and discussion}\label{sec_setting_results}

\subsection{Setting}
\label{sec_setting}
Let us define the model in any dimension and with a continuous parameter before focusing on the two-neighbor graph in two dimensions. We consider the $d$-dimensional  integer lattice $\mathbb Z^d$ and a parameter $p\in [0,1]$ to which we associate two further parameters
\begin{equation}
\label{param}
k:=\lfloor 2dp\rfloor \; \text{ and } \; \varepsilon:=2dp-k\in [0,1).
\end{equation}
Based on this, we define a directed edge-percolation model as follows. Each vertex in $\mathbb Z^d$ chooses, independently and uniformly at random, $k$ out of its $2d$ direct neighbors and we draw directed arrows towards these neighbors. Additionally, independently each vertex also throws a coin with success probability $\varepsilon$ and in case of success, one additional neighbor is chosen uniformly from the previously not chosen neighbors and we draw an additional directed edge towards that neighbor. This construction provides a probability distribution denoted by $\mathbb{P}_p$ on the configuration set $\Omega := \{0,1\}^\E$, where $\E := \{(x,y) \in (\Z^d)^2\colon \|x-y\|_{\ell^{1}} = 1\}$ stands for the set of directed edges of $\mathbb Z^d$. As usual, a directed edge $e$ is called {\em open} (respectively {\em closed}) in the configuration $\omega \in \Omega$ if $\omega(e) = 1$.

This is a continuous-parameter version of the $k$-neighbor graph studied previously in~\cite{KJLT}, which we call the {\em $2dp$-neighbor graph}. On the one hand, it generalizes the $k$-neighbor graph since, for example for $d=2$, the $2dp$-neighbor graph with $p=1/4$, $p=1/2$ and $p=3/4$, respectively, correspond to the $k$-neighbor graph with $k=1$, $k=2$ and $k=3$.  On the other hand, when for instance $d=2$ and the parameter $p$ goes from $1/4$ to $1/2$, the $2dp$-neighbor graph interpolates the $1$ and $2$-neighbor graphs (with $k=1$ and $\varepsilon$ going from $0$ to $1$), which will turn out to be useful to solve percolation questions.

Note that, in this model, each vertex has minimal and maximal degrees respectively given by $k$ and $k+1$. Moreover, the expected degree is equal to 
\begin{align}
\mathbb E_p [\text{deg}(o)] = k(1-\varepsilon)+(k+1)\varepsilon = k + \varepsilon = 2dp, 
\end{align}
where $o$ denotes the origin in $\mathbb Z^d$ and $\mathbb{E}_p$ the expectation w.r.t.~the probability measure $\mathbb{P}_p$. As a result, each directed edge has probability $p$ to be open, which justifies the choice of the parametrization~\eqref{param}. It is worth pointing out that, in the $2dp$-neighbor model, the states of directed edges starting from a given vertex are dependent and negatively correlated.

\subsection{Results}
\label{sec_results}

Given two vertices $x,y \in \Z^d$, we write $x\rightsquigarrow y$ for the event that there exists a directed open path from $x$ to $y$ in the lattice $\Z^d$, i.e., if there exists $x_0,x_1,\dots,x_n$ in $\Z^d$ with $x_0 = x$ and $x_n = y$ such that all the $(x_i,x_{i+1})$'s are open directed edges. We write $o \rightsquigarrow \infty$ for the event that there exists an infinite self-avoiding open path starting at the origin $o$. We are interested in the percolation behavior of the resulting directed graph represented by the {\em percolation probability}
\begin{align}
\theta_p = \theta_p(d) := \mathbb P_p(o\rightsquigarrow\infty).
\end{align}
When $\theta_p > 0$, we will say that {\em percolation occurs} in the $2dp$-neighbor model.

A standard coupling argument (given in the appendix for completeness) guarantees that the percolation probability can not decrease as $p$ increases.
\begin{lemma}[Monotonicity]
\label{lem:monotoneP_kappa}
For all $d\ge 1$, the function $p \mapsto \mathbb \theta_p(d)$ is non-decreasing.
\end{lemma}

As usual, let us define the {\em critical parameter} as
\begin{align}
p_c = p_c(d) := \sup\{p\ge 0\colon\theta_p=0\}.
\end{align}
The monotonicity property asserts that $\theta_p$ is zero when $p < p_c$ and positive when $p > p_c$. As a reference, in Table~\ref{table-1}, we give estimates on the critical parameter based on simulations.

\begin{table}[hb]
\begin{center}
    \begin{tabular}{|c|c|c|c|} \hline
       &  $d=2$ & $d=3$ & $d=4$ \\ \hline
       $2dp_c $   & 1.84  & 1.38  & 1.20 \\ \hline
        $p_c $   & 0.45  & 0.23  & 0.15  \\ \hline
    \end{tabular}
\end{center}
\caption{Estimates for the critical parameters based on simulations\protect\footnotemark.}\label{table-1}
\end{table}
\footnotetext{See \url{https://bennhenry.github.io/NeighPerc/}.}
It is clear that $p_c(1)=1$ as the degree-one graph does not percolate since backtracking is possible with positive probability, see~\cite[Proposition~2.1]{KJLT}. Additionally, a general lower bound can be derived from a simple union bound (given again in the appendix). For this, recall the {\em connective constant} of $\mathbb Z^d$
\begin{align*}
c(d) := \lim_{n \to \infty} c_n(d)^{1/n},
\end{align*}
see, e.g., \cite{grimmett2006random}, where $c_n(d)$ represents the number of self-avoiding paths of length $n$ in the $d$-dimensional hypercubic lattice that start at the origin.

\begin{lemma}[Lower bound]
\label{lem_low1st}
For all $d\ge 2$, it holds that $p_c(d) \geq 1/c(d)$.
\end{lemma}

For $d=2$, it is known that $c(2)\leq 2.679192495$, see~\cite{PonTitt00}, and hence $p_c(2) \ge 0.37324679054$. However, the main focus of this article is to show that $p_c(2)$ is strictly less than $1/2$ in the planar case.

\begin{theorem}[Upper bound]
\label{thm:upperBound}
It holds that $p_c(2) < 1/2$. In particular, the $2$-neighbor graph percolates.  
\end{theorem}

Combining these lower and upper bounds leads to $0.373 < p_{c}(2) < 0.5$ and thus proves the conjecture of~\cite{KJLT}.

Let us finally mention a straightforward consequence of Theorem~\ref{thm:upperBound}. In~\cite[Theorem~2.3]{KJLT}, it is proved that the percolation of the $k$-neighbor graph in $\Z^d$ implies the one of the $(k+1)$-neighbor graph in $\Z^{d+1}$. 
This implies the following result.
%that $p_{\rm c}(d)< 1/2$ for all $d\ge 2$ and the following stronger statement.
\begin{corr}[Upper bound in all dimensions]
For $d\ge 2$ it holds that $p_c(d) < 1/2$. In particular, the $d$-neighbor graph percolates in $\Z^d$.
\end{corr}
In the following section, we discuss the relation of our model to a variety of other degenerate random environments.  

\subsection{Discussion}
\label{sec_discussion}

As mentioned in the introduction, our result (and even more so our simulations) suggests that the degree constraint is beneficial for percolation. In order to make this more precise, we consider several related directed percolation models that are also isotropic and i.i.d.~over the set of sites.  

\subsubsection{Independent directed percolation}

First, consider the product probability distribution $\mathbb{P}_p^{\text{iid}}$ on $\Omega = \{0,1\}^\E$, where $\E$ still denotes the set of directed edges of $\mathbb Z^2$, in which directed edges are open independently from each other with probability $p\in [0,1]$. To be consistent with the previous notations, we set
\[
\theta_p^{\text{iid}} := \mathbb P_p^{\text{iid}}(o\rightsquigarrow\infty) \quad \mbox{ and } \quad p_c^{\text{iid}} := \sup\{p\ge 0\colon\theta_p^{\text{iid}}=0\}.
\]
A coupling with the undirected independent bond percolation model on $\Z^2$ allows to prove that planar independent directed percolation also features a critical parameter given by $1/2$. 

\begin{lemma}[Critical independent directed percolation]
\label{lem:Critical-iid}
In dimension $d=2$, the percolation probability satisfies $\theta_p^{\text{iid}} = 0$ if and only if $p\leq 1/2$. In particular, $p_c^{\text{iid}} = 1/2$.
\end{lemma}

The complete proof is again presented in the Appendix.

\subsubsection{All-or-none directed percolation}

Second, consider a percolation model where, for each vertex of $\mathbb Z^d$ and independently from each other, either its $2d$ outgoing edges are open with probability $p\in [0,1]$ or they are all closed with probability $1-p$. We denote by $\mathbb{P}_p^{\text{aon}}$ the corresponding probability distribution on $\Omega = \{0,1\}^\E$. According to $\mathbb{P}_p^{\text{aon}}$, the probability for a given directed edge to be open is $p$. As before we also define the percolation probability $\theta_p^{\text{aon}}$ and the critical parameter $p_c^{\text{aon}}$. In this model only the vertices having four outgoing edges may help the origin to reach infinity. This basic remark provides the next result where $p_c^{\text{site}} = p_c^{\text{site}}(\Z^d)$ denotes the critical value for the i.i.d.~site percolation.

\begin{lemma}[Critical all-or-none directed percolation]
\label{lem:Critical-0-4}
In all dimensions $d\ge 1$, $p_c^{\text{aon}} = p_c^{\text{site}}$.
\end{lemma}

The complete proof is again presented in the appendix. Let us recall that, in dimension $2$, $p_c^{\text{site}}$ is estimated at around $0.59$, see~\cite{jacobsen2015critical}, and the best rigorous lower bound is $0.55$, see~\cite{van1996new}.

\subsubsection{Geometrically constraint directed percolation}

Let us mention two more isotropic degree constraint models, called the {\em northsouth-eastwest model} and the {\em directed-corner model}, that will be used for comparison. We focus only on the planar case $\Z^2$. Consider again $p\in [0,1]$ and the associated parameters $k = \lfloor 4p\rfloor$ and $\varepsilon = 4p-k \in [0,1)$. Both models coincide with the $4p$-neighbor graph when $k=0$ (or $p < 1/4$). Now, if $k=1$, we pick uniformly one of the four neighbors and draw a directed edge towards it. Then, we throw a coin with success parameter $\varepsilon$ and, in case of success, we open the edge to the neighbor that is located in the opposite direction of the already connected neighbor for the northsouth-eastwest model. In the limiting case $\varepsilon=1$ (corresponding to $p=1/2$), we thus get a model in which either the edges towards north and south or the edges towards east and west are open, see Figure~\ref{fig:straight_config_and_forward_cluster} for an illustration. For the directed-corner model, in case of success, we open one of the two neighboring edges of the already open edge uniformly at random. Again, in the limiting case $\varepsilon=1$ (corresponding to $p=1/2$), we get a model where each of the following combinations of edges is open with probability $1/4$: northeast, northwest, southeast and southwest. See Figure~\ref{fig:corner_config_and_forward_cluster} for an illustration. When $k=2$ and in case of success (with probability $\varepsilon$), for both models, a third open edge is added, chosen uniformly from the remaining unopened edges. Both models present no interest w.r.t.~percolation when $k\geq 3$. We write $\theta_p^{\text{ns--ew}}$, $p_c^{\text{ns--ew}}$ and $\theta_p^{\text{corn}}$, $p_c^{\text{corn}}$ for the associated percolation functions and the critical parameters. Let us mention the following simple observation without proof.

\begin{lemma}[Percolation for the northsouth-eastwest model]
We have that $\theta_{1/2}^{\text{ns--ew}}=1$. 
\end{lemma}

\begin{figure}[ht]
\centering
\includegraphics[width=.9\textwidth]{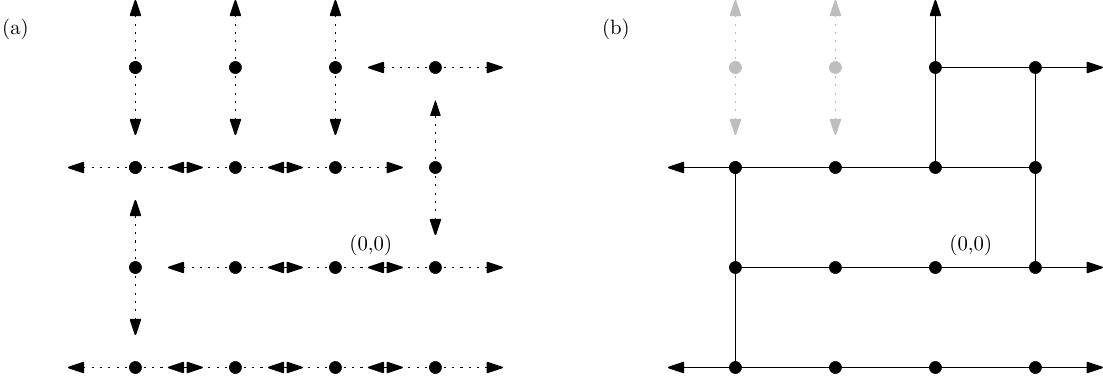}
\caption{\label{fig:straight_config_and_forward_cluster}(a) A possible local configuration of the northsouth-eastwest model with $p=1/2$ and (b) the resulting forward cluster of the origin.}
\end{figure}

Our method of proof for Theorem~\ref{thm:upperBound} is sufficiently robust to also provide an upper bound for the critical value in the directed-corner model, leading to a corresponding statement in this case, which will be proved in Section~\ref{sec_upperbound_cor}. 

\begin{theorem}[Upper bound for the directed-corner model]
\label{thm_upperbound_cor}
It holds that $p_c^{\text{corn}} < 1/2$.
\end{theorem}

\begin{figure}[ht]
\centering
\includegraphics[width=.9\textwidth]{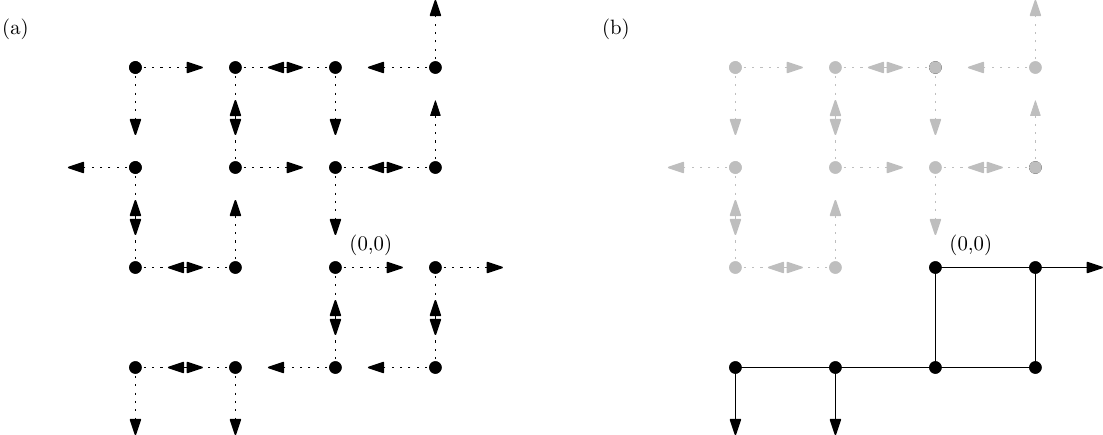}
\caption{\label{fig:corner_config_and_forward_cluster}(a) A possible local configuration of the directed corner model with $p=1/2$ and (b) the resulting forward cluster of the origin.}
\end{figure}

\subsubsection{Comparison}

From our rigorous analysis we can formulate the following comparison of critical values.

\begin{corr}[Comparison]
For $d=2$, we have that $
p_c < p_c^{\text{iid}} = 1/2 < p_c^{\text{aon}}$.
\end{corr}

Additionally, our simulations suggest the following strict order of percolation thresholds in two spatial dimensions.

\begin{conjecture}
For $d=2$, we have that 
$
p_c^{\text{ns--ew}}<p_c<p_c^{\text{corn}}< p_c^{\text{iid}}< p_c^{\text{aon}}$. 
\end{conjecture}

This only partially supports the hypothesis that less variability benefits percolation. Let us add that simulations indicate that $p_c^{\text{ns--ew}}$ and $p_c^{\text{corn}}$ are respectively estimated at around  $0.42$ and $0.48$, meaning that the strict upper bound $1/2$ given by Theorem~\ref{thm_upperbound_cor} is quite accurate.

Let us elaborate on the relation between variability and percolation based on the observation that also the percolation functions at $p=1/2$ are ordered,
\begin{align}\label{Comp}
1 = \theta_{1/2}^{\text{ns--ew}} > \theta_{1/2} > \theta_{1/2}^{\text{corn}} > \theta_{1/2}^{\text{iid}} = \theta_{1/2}^{\text{aon}} = 0,
\end{align}
where only the inequality $\theta_{1/2} > \theta_{1/2}^{\text{corn}}$ is based on simulations.  
First, we note that, for $p=1/2$, the expected degree is equal to two in all the models. However, the {\em variance in the number of outgoing edges} differs and hence may serve as an indicator. Indeed, $\V_{1/2}[\deg(o)]=0$ whereas $\V^{\text{iid}}_{1/2}[\deg(o)]=1$ and $\V^{\text{aon}}_{1/2}[\deg(o)]=4$. However, the number variance is unable to explain the differences in the geometrically constraint models since also there it is zero.

Another notion that quantifies variability is rigidity, inspired from statistical mechanics. That is, a point process is called {\em rigid} (or more precisely {\em number-rigid}) if the number of points in a bounded Borel set $\Delta$ is a function of the point configuration outside $\Delta$. This property has been introduced in~\cite{HolroydSoo} and studied for a large class of point processes as for instance the Ginibre process~\cite{GhoshPeres} or perturbed lattices~\cite{PeresSly}. It is believed that, in a very rough sense, also {\em rigidity should help to percolate} but, to our knowledge and until now, there is no evidence sustaining this claim. 
Now, our main model (as well as the two other geometrically constraint models) is rigid in the following sense. Given a bounded subset $\Delta$ of directed edges and a configuration $\omega \in \Omega$ distributed according to $\mathbb{P}_{1/2}$, then the number of directed edges of $\omega$ inside $\Delta$ is determined by the configuration outside $\Delta$. Of course, this is not the case for the models $\mathbb{P}_{1/2}^{\text{iid}}$ and $\mathbb{P}_{1/2}^{\text{aon}}$ and hence also rigidity is a criterion that is able to describe the second right-most inequality sign in~\eqref{Comp}, but again not the others. 

Focusing on the three degree-constraint models, we can say the following. Within the set of isotropic, degree-two  degenerate random environments the northsouth-eastwest model and the directed-corner model are extremal in the following sense. In order to maintain isotropy the four corner-edge configurations need to have the same probability, say $\rho\in [0,1/4]$, and the same is true for the two configurations northsouth and eastwest to which we associate the  probability $\psi\in [0,1/2]$. Note that this imposes the constraint $4\rho+2\psi=1$ and that the choice $\rho=1/6$ recovers our main model whereas $\rho=0$ represents the northsouth-eastwest model respectively $\rho=1/4$ the directed corner model. Hence, we can parametrized any isotropic degree-two model via a single parameter $\rho\in[0,1/4]$. Our simulations for the three values $\rho\in\{0,1/6,1/4\}$ now suggest that $\rho\mapsto\theta_{1/2}^{\rho}$ is decreasing. In summary, it remains a challenge for the future to find a good notion of variability that is able to explain the differences in percolation behavior. This may include a notion of geometrical variance that seems to favor edge pairs facing in opposite directions, or rigidity beyond number rigidity, or even a geometrical version of convex ordering that has been successfully used in the context of first-passage percolation.

Let us close this section with a more specific open question. Let $\mathcal{C}_p$ denote the class of probability distributions on $\Omega$ which are i.i.d.~over the vertices, isotropic, and with expected degree equal to $p$. Then, let $\mathbf{P}$ be any element of the class $\mathcal{C}_p$ and denote by $\theta(\bf P)$ its percolation probability. Based on our numerical observations and variance intuition we pose the following conjecture. 

\begin{conjecture}
For all $\mathbf{P}\in \mathcal{C}_p$, we have  $\theta^{\text{ns--ew}}_p \geq \theta({\bf P}) \geq \theta^{\text{aon}}_p$. 
\end{conjecture}

In the next section we provide an overview for the proof of our main result. 

\subsection{Strategy of proof of Theorem~\ref{thm:upperBound}}
\label{sec_ingredients}

Let us mainly focus on the $4p$-neighbor graph with $4p = 2+\varepsilon$ outgoing edges in mean per vertex, i.e., we allow a third outgoing edge to occur with probability $\varepsilon \geq 0$. Hence, the probability for a given directed edge to be open is given by
\begin{equation}
\label{Setting:2+eps}
p = \frac{1}{4} \mathbb E[ \text{deg}(o) ] = \frac{1}{2} + \frac{\varepsilon}{4} \; \mbox{ with $\varepsilon \geq 0$.}
\end{equation}
This is why we write in the sequel $\mathbb{P}_{(2,\varepsilon)}$ instead of $\mathbb{P}_p$ with $p = 1/2 + \varepsilon/4$. For short, we will sometimes refer to the probability measure $\mathbb{P}_{(2,\varepsilon)}$ as the $(2,\varepsilon)$-model. Although the monotonicity property stated in Lemma~\ref{lem:monotoneP_kappa} allows to reduce our attention to the case $\varepsilon = 0$, we will work in the general setting \eqref{Setting:2+eps} to emphasize the fact that, when $\varepsilon = 0$, more work will be required.

The main work consists in proving Theorem~\ref{thm:clusterSize} below,  which roughly says that, in the $(2,\varepsilon)$-model with $\varepsilon \geq 0$, the size of dual forward sets decay exponentially fast. This strong decay in the case of $\varepsilon = 0$ (i.e., $p = 1/2)$, combined with a classical block-renormalization approach, allows us to state in Section~\ref{sec:block-construction} that the $4p$-neighbor graph percolates for some $p < 1/2$, leading to $p_c < 1/2$ (i.e., Theorem~\ref{thm:upperBound}). Let us now give some more details on the proof of Theorem~\ref{thm:clusterSize}.

\medskip
A dual (directed) edge $e^\ast$ is said to be {\em open} if and only if its associated primal (directed) edge $e$ is closed, which occurs with probability
\[
\mathbb{P}_{(2,\varepsilon)} \left( \scalebox{1.2}{\dualopen} \right) = 1 - \Big(\frac{1}{2}+\frac{\varepsilon}{4}\Big) = \frac{1}{2} - \frac{\varepsilon}{4}
\]
in the $(2,\varepsilon)$-model. The fact that this probability is smaller than $1/2$ is our basic ingredient to get Theorem~\ref{thm:clusterSize}, i.e., that the probability for the dual forward set $\df{o^\ast}{} = \{x^\ast \in (\Z^\ast)^2 \colon o^\ast \rightsquigarrow x^\ast \}$ to be large decreases exponentially fast. Now, to carry out this strategy, we face two major obstacles:
\begin{itemize}
\item In our model, the primal edges are dependent and so do the dual edges, which prevents a direct comparison to the i.i.d.~(undirected) bond percolation model with parameter $1/2 - \varepsilon/4$.
\item Assuming such domination would hold, when $\varepsilon = 0$, the tail distribution of the size of the cluster of the origin in the (critical) i.i.d.~bond percolation model does not decrease fast enough to get Theorem~\ref{thm:clusterSize}.
\end{itemize}

In Section~\ref{sec_Exploration}, we define an exploration process allowing to explore the forward set $\f{o^\ast}{}$. During the exploration process, each revealed dual edge of $\f{o^\ast}{}$ is open with probability less than $1/2 - \varepsilon/4$, except some special edges, called {\em pivotal}, whose probability to be open tends to $2/3$ as $\varepsilon \to 0$. The existence of such pivotal edges, due to the dependency of our model, is not good for a comparison to a subcritical (or critical) i.i.d.~bond percolation model. But the planarity helps: each revealed pivotal edge appears to be an opportunity to stop the exploration process (Lemma~\ref{lem:pivot}). Hence, roughly speaking, the dual forward set $\f{o^\ast}{}$ can be viewed as a subgeometric number of disjoint clusters, called {\em visited clusters} and denoted by $\kedf{o^\ast}{k}$ for $k\geq 1$, separated by pivotal edges. Also, inside a given visited cluster $\kedf{o^\ast}{k}$, each revealed dual edge has a probability to be open smaller than $1/2 - \varepsilon/4$, and thus a comparison with the i.i.d.~bond percolation model with parameter $1/2 - \varepsilon/4$ is doable for each $\kedf{o^\ast}{k}$. The first major obstacle is overcome.

At this stage, a subexponential decay for the size of each of these visited clusters (Theorem~\ref{theo:subGeomClusterSize}) will be enough to obtain Theorem~\ref{thm:clusterSize}. This is when we take advantage of the rigidity of the $(2,\varepsilon)$-model: since at least two primal outgoing edges start from each vertex, the exploration process of $\kedf{o^\ast}{k}$ cannot create the pattern \scalebox{0.3}{\begin{tikzpicture}
    \draw[step=1cm, gray,very thin] (-1,0) grid (2,1);
    \draw[blue, line width=1mm, -latex] (0,1) -- (0,0);
    \draw[blue, line width=1mm, -latex] (0,0) -- (1,0);
    \draw[blue, line width=1mm, -latex] (1,0) -- (1,1);
\end{tikzpicture}}, called a {\em left winding}, because it involves three dual directed edges which all refer to the same primal vertex. Along this line, we prove in Proposition~\ref{prop:stochasticDomination} that each visited cluster is dominated by the cluster of the origin in the i.i.d.~(undirected) bond percolation model with parameter $1/2 - \varepsilon/4$ but under constraints. The constraints are mainly that the undirected pattern \scalebox{0.7}{\pattern}, which necessarily contains a left winding, is forbidden to be used for percolating paths.

In the case where $\varepsilon > 0$, the parameter $1/2-\varepsilon/4$ is strictly smaller than $1/2$ and then corresponds to the subcritical regime for the i.i.d.~bond percolation model in $\Z^2$. A sharp transition being well-known in that case (see for instance~\cite[Chapter~6]{grimmett1999percolation}), we immediately obtain an exponential decay for the size of visited clusters, without using the constraint of forbidden patterns. When $\varepsilon = 0$, the i.i.d.~bond percolation model with parameter $1/2 - \varepsilon/4$ becomes critical and we have no choice but to take advantage of the constraint of forbidden patterns. This last step is done in Section~\ref{sect:enhancement} in which we use an enhancement approach to prove that taking into account the constraint of forbidden patterns makes percolation harder and actually strictly increases the critical parameter of the corresponding percolation model. However, let us note that while the proof is inspired by the works~\cite{aizenman1991strict} and~\cite{balister2014essential}, we cannot reuse their results for our purposes, because the enhancements (or diminishments) we consider are {\em not essential}.
To the best of our knowledge, this also provides the first successful application of the enhancement technique to models in which percolating paths are forbidden to use certain patterns and is of independent interest. Therefore, we made sure to keep Section~\ref{sect:enhancement} self-contained so that it can be read independently from the rest of the manuscript.

\section{Strong decay of dual forward sets}
\label{sec:block-construction}

Recall that we write in the sequel $\mathbb{P}_{(2,\varepsilon)}$ instead of $\mathbb{P}_p$ with $p = 1/2 + \varepsilon/4$ and we call it the $(2,\varepsilon)$-model.

\subsection{Duality}
\label{sec:duality}

Let $\Za = (1/2,1/2)+\mathbb{Z}^{2}$ be the {\em dual lattice} of $\Z^2$ whose set of directed dual edges is denoted by $\E^\ast$. Let $o^\ast = (1/2,1/2)$ be the origin of $\Za$. To each {\em primal edge} $e =(x,y) \in \E$, we define its {\em dual edge} $e^\ast \in \E^\ast$ as follows: $e^\ast$ is the dual directed edge obtained from $e$ by the counter-clockwise rotation with center $(x+y)/2$ and angle $\pi/2$, see Figure~\ref{fig:dualEdge} for an illustration.

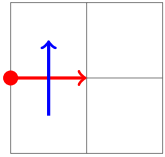
\begin{figure}[!ht]
\centering
\begin{tikzpicture}
\tikzset{arrow/.style={->,>=stealth,thick}}
\coordinate (start) at (2,1.5);
\coordinate (end2) at (2,1.5);
\coordinate (control1) at (10,5); 
\coordinate (control3) at (-5,5);
\draw[step=1cm,gray,very thin] (0,0) grid (2,2);
\draw[->,red, very thick] (0,1) -- (1,1);
\draw[->,blue, very thick] (0.5,0.5) -- (0.5,1.5);
\node at (0,1) [circle,fill=red,inner sep=2pt] {};
\end{tikzpicture}
\caption{Primal edge in red and its dual edge in blue. This color code will be conserved for the whole paper. As primal edges, dual edges are also directed.}
\label{fig:dualEdge}
\end{figure}

Given a configuration $\omega \in \Omega$, we define its {\em dual configuration} $\omega^{\ast}$ as the element of $\Omega^\ast = \{0,1\}^{\E^{\ast}}$ such that
\[
\forall e \in \E, \, \omega^{\ast}(e^{\ast}) = 1 \; \text{ if and only if } \; \omega(e) = 0.
\]
As usual, the dual edge $e^\ast$ is open if and only if its primal edge $e$ is closed. Hence, the probability in the $(2,\varepsilon)$-model for a given dual edge to be open is
\begin{equation}
\label{ProbaOpenDualEdge}
\mathbb{P}_{(2,\varepsilon)} \left( \scalebox{1.2}{\dualopen} \right) = 1 - \Big(\frac{1}{2}+\frac{\varepsilon}{4}\Big) = \frac{1}{2} - \frac{\varepsilon}{4},
\end{equation}
which is smaller than $1/2$ since $\varepsilon \geq 0$.

% Forward sets:
The {\em forward set} of $x \in \Z^2$ is made up with vertices that can be reached by a directed open path starting at $x$, i.e., 
\[
\df{x}{} = \{y \in \Z^2 \colon x \rightsquigarrow y\}.
\]
We use the convention that $x \rightsquigarrow x$ is always true meaning that any vertex belongs to its own forward set. Next, we extend the relation '$\cdot\rightsquigarrow\cdot$' to the dual lattice. That is, given $x^\ast,y^\ast \in \Za$, we set $x^\ast \rightsquigarrow y^\ast$ if there exists a directed open dual path from $x^\ast$ to $y^\ast$. The (dual) forward set of $x^\ast$ is then defined as
\[
\df{x^\ast}{} = \{y^\ast \in \Z^2 \colon x^\ast \rightsquigarrow y^\ast \}.
\]

Now that we have introduced the dual model and set up the whole notation, let us state the main ingredient for the proof of Theorem~\ref{thm:upperBound}: the tail distribution of the size of dual forward sets admits a subexponential decay.

\begin{theorem}
\label{thm:clusterSize}
Let $\varepsilon\geq 0$. Then, there exist constants $c,C > 0$ such that, for any integer $n$,
\begin{align}
%\label{DecayFor0*}
\PP_{(2,\varepsilon)} \big( | \df{o^\ast}{} | \geq n \big) \leq C e^{-c n^{1/4}}.
\end{align}
\end{theorem}

The proof of this result is split into several steps and the rest of the manuscript is devoted to it. From Theorem~\ref{thm:clusterSize}, it will be easy to get our main result Theorem~\ref{thm:upperBound} and this is done in the next section.

\subsection{Proof of Theorem~\ref{thm:upperBound}}
\label{sect:proofTheoUB}

For this section only, it will be more convenient to use the notation $\mathbb{P}_p$ for the $4p$-neighbor graph (recall that $d=2$). We will employ a block-renormalization argument to reduce the percolation problem for the $4p$-neighbor graph, for some $p<1/2$, to that of a dependent site percolation model which will be proved to be supercritical (or percolating) thanks to the classical stochastic domination result of~\cite{liggett1997domination}. The key input here is that, at the level $p=1/2$, the dual forward sets decay subexponentially fast by Theorem~\ref{thm:clusterSize}.

\medskip
Let us define, for any $z \in \Z^2$ and $L \in \N$, the (good) event
\[
G(Lz,3L) := \big\{ \text{$\exists$ an open (primal) cycle surrounding $Lz$ in $B( Lz,3L/2 )\!\setminus\! B(Lz,L/2)$} \big\}, 
\]
where $B(x,R)$ denotes the $\norm{\cdot}_\infty$-ball with center $x$ and radius $R>0$. See Figure~\ref{fig:block-argument} for an illustration. Thus, we introduce the (dependent) site percolation model $\{\xi^{(L)}_z\}_{z \in \Z^2}$ on $\{0,1\}^{\Z^2}$ where
\[
\xi^{(L)}_z := \mathds{1}_{G(Lz,3L)}
\]
for any $z \in \Z^2$ and $L \in \N$. Since the $4p$-neighbor graph is translation invariant in law, the $\xi^{(L)}_z$'s are identically distributed. The random field $\{\xi^{(L)}_z\}_{z \in \Z^2}$ is also $6$-dependent since the indicators $\xi^{(L)}_z$ and $\xi^{(L)}_{z'}$ are independent whenever $\|z-z'\|_\infty > 6$. Finally, percolation for $\{\xi^{(L)}_z\}_{z \in \Z^2}$ implies percolation for the $4p$-neighbor graph, as depicted in Figure~\ref{fig:block-argument}. Henceforth, it is sufficient to prove that the random field $\{\xi^{(L)}_z\}_{z \in \Z^2}$ percolates under the probability measure $\mathbb{P}_p$ for some $p<1/2$.

\begin{figure}[!ht]
\centering
\includegraphics[width=.4\textwidth]{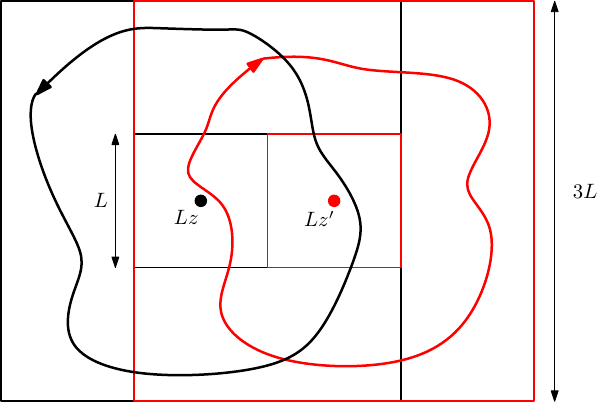}
\caption{\label{fig:block-argument} Two neighboring blocks, centered at $Lz$ and $Lz'$ with $z'=z+(1,0)$, and the corresponding open cycles realizing the events $G(Lz,3L)$ and $G(Lz',3L)$.}
\end{figure}

Theorem~\ref{thm:clusterSize} allows us to establish the next result whose proof is postponed to the end of the section.

\begin{lemma}
\label{lem:G(z)TendsTo1}
The following limit holds:
\[
\lim_{L\to\infty} \mathbb{P}_{1/2} (G(o,3L)) = 1.
\]
\end{lemma}

Let us now conclude by applying~\cite[Theorem~0.0]{liggett1997domination} to the $6$-dependent random field $\{\xi^{(L)}_z\}_{z \in \Z^2}$. Let us first pick $\rho > p_c^{\text{site}}(\Z^2)$ such that the product $\{0,1\}$-valued random field with density $\rho$ percolates. \cite[Theorem~0.0]{liggett1997domination} asserts that there exists $\rho'<1$ (only depending on $\rho$, the lattice $\Z^2$ and the degree of dependence which is $6$ here) such that the lower bound on the marginals
\begin{equation}
\label{LiggettMarginals}
\forall z \in \Z^2 , \; \mathbb{P}_p(\xi^{(L)}_z = 1) \geq \rho'
\end{equation}
implies that the random field $\{\xi^{(L)}_z\}_{z \in \Z^2}$, under the probability measure $\mathbb{P}_p$, stochastically dominates the product $\{0,1\}$-valued random field with density $\rho$, and then it almost-surely percolates.

It then remains to prove that \eqref{LiggettMarginals} holds for some $p<1/2$. It is time to invoke Lemma \ref{lem:G(z)TendsTo1}: for any $z \in \Z^2$,
\[
\mathbb{P}_{1/2}(\xi^{(L)}_z = 1) = \mathbb{P}_{1/2}(\xi^{(L)}_o = 1) = \mathbb{P}_{1/2}(G(o,3L)) > \rho',
\]
for $L$ large enough and by translation invariance. The parameter $L$ being fixed, the event $G(o,3L)$ only involves a finite number of random variables leading to the continuity of the function $p \mapsto \mathbb{P}_p\left(G(o,3L)\right)$, especially at $p=1/2$. Hence, we can choose $p<1/2$ but close enough to $1/2$ such that $\mathbb{P}_{p}(G(o,3L))$ is still larger than $\rho'$. The lower bound~\eqref{LiggettMarginals} is then satisfied for the probability measure $\mathbb{P}_p$ with $p<1/2$.

In conclusion, we have proved that for some $p < 1/2$ the $4p$-neighbor graph percolates. This means that $p_c < 1/2$ and achieves the proof of Theorem~\ref{thm:upperBound}.

\begin{proof}[Proof of Lemma~\ref{lem:G(z)TendsTo1}]
Let us fix $p = 1/2$ here. Given $L \in \N$, we consider the event $\text{Cross}(L)$ defined by the existence of an open primal path crossing the horizontal rectangle $[0,3L]\times[0,L]$ from the left side to the right side while remaining inside the rectangle. As illustrated in Figure~\ref{fig:glueing}, it is sufficient to prove that
\begin{equation}
\label{glueing}
\lim_{L\to\infty} \mathbb{P}_{1/2} ( \text{Cross}(L) ) = 1
\end{equation}
to get Lemma~\ref{lem:G(z)TendsTo1}. We use here that the $4p$-neighbor graph is invariant in distribution w.r.t.~rotations with angles $\pi/2$, $\pi$ and $3\pi/2$.

\begin{figure}[!ht]
\centering
\begin{tabular}{cp{1cm}c}
\includegraphics[width=.4\textwidth]{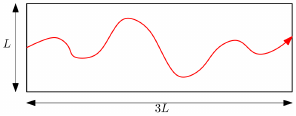} & & \includegraphics[width=.25
\textwidth]{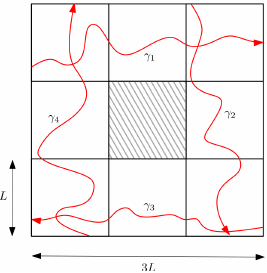} 
\end{tabular}
\caption{\label{fig:glueing}\textit{Left:} A representation of the event $\text{Cross}(L)$. \textit{Right:} Considering four rotated and shifted copies of the event $\text{Cross}(L)$ and glueing the corresponding open paths $\gamma_i$, $i=1,2,3,4$, we obtain a cycle surrounding the origin $o$ in $B(o,3L/2)\!\setminus\! B(o,L/2)$. The event $G(o,3L)$ then occurs.}
\end{figure}

The absence of an open primal path horizontally crossing the rectangle $[0,3L]\times[0,L]$ forces the existence of an open dual one crossing vertically that rectangle. Precisely, there exists an open dual path starting at $z_i^\ast := (i+1/2,-1/2)$, for some integer $0\leq i\leq 3L-1$, and hitting the horizontal line $y=L+1/2$. Since such dual path admits more than $L$ edges, we can assert that:
\[
1 - \mathbb{P}_{1/2} ( \text{Cross}(L) ) \leq \mathbb{P}_{1/2} \Big( \bigcup_{i=0}^{3L-1} \big\{ |\df{z_i^\ast}{}| \geq L \big\} \Big) \leq 3L \times C e^{-c L^{1/4}} 
\]
by Theorem~\ref{thm:clusterSize} applied with $\varepsilon = 0$. This implies the claimed convergence as $L$ tends to infinity.
\end{proof}

\section{Exploration of the dual forward set}
\label{sec_Exploration}

To prove Theorem~\ref{thm:clusterSize} we aim to compare the dual forward set of $o^\ast$ to a standard Bernoulli percolation. Recall that the probability for a dual edge to be open is smaller than $1/2$, see~\eqref{ProbaOpenDualEdge}. This cannot be done directly for reasons to be made explicit later. As a consequence, to attain this purpose, we need to introduce an exploration process allowing to explore $\f{o^\ast}{}$ part by part, each part being independently dominated. To do this, we introduce an exploration algorithm in Section~\ref{sect:ExploAlgo}. During the exploration process of a dual forward set, some explored edges will play a crucial role: they are called {\em pivotal edges} and are studied in Section~\ref{sect:PivotalAlgo}. Roughly speaking, we will prove that a dual forward set can be viewed as a subgeometric number of disjoint clusters separated by pivotal edges (see Figure~\ref{fig:cluster_chain}). Henceforth, a subexponential decay for the size of these clusters (Theorem~\ref{theo:subGeomClusterSize}) will be enough to get Theorem~\ref{thm:clusterSize}, see Section~\ref{sect:VisitedCluster}.

\subsection{The exploration algorithm}
\label{sect:ExploAlgo}

Given a dual vertex $x^\ast$, the exploration algorithm performs a (partial) exploration of the forward set $\f{x^\ast}{}$. It deals with an ordered list of directed dual edges that are to be explored. When this list becomes empty, the algorithm stops. During the whole process, we will keep a record of the set of dual vertices and edges already explored. Here, exploring the edge $e^\ast$ means its state $\omega^\ast(e^\ast)$ is revealed. Note also that the algorithm is processed trajectorially, i.e., for a given $\omega\in\Omega$, but, as before, we omit the dependence in $\omega$ to lighten notations.

Let us denote by $\mathcal{U}$ the set of (finite) ordered lists of directed dual edges of arbitrary length, that is
\[
\mathcal{U} = \bigcup_{n\geq 1} (\E^{\ast})^{n}.
\] 
For $u \in \mathcal{U}$, denote by $|u|$ the length of the list $u$, i.e., the unique integer $n$ such that $u \in (\E^\ast)^{n}$. The concatenation of two lists $u,v \in \mathcal{U}$ is defined by
\[
u:v = (u_{1},\ldots u_{|u|},v_{1},\ldots,v_{|v|}).
\]

Denote by $L^{(n)} \in \mathcal{U}$ the list of edges to be explored at the beginning of the $n$-th step of the algorithm. The current state during the $n$-th step is
\[
X_n = (x^\ast_n , e^\ast_n , \omega^\ast(e^\ast_n))
\]
where $e^\ast_n = (x^\ast_n,y^\ast_n)$ is the first element of the list $L^{(n)}$. The algorithm explores the forward set $\f{x^\ast}{}$ in the {\em depth-first fashion} and in the {\em counter-clockwise sense}, see Figure~\ref{fig:explorationOrder}. In other words, if at the end of the $n$-th step there are three new directed dual edges to be explored and to be added to the list $L^{(n)}$, say $(y^\ast_n,a^\ast_1)$, $(y^\ast_n,a^\ast_2)$ and $(y^\ast_n,a^\ast_3)$, then they are placed at the beginning of the new list $L^{(n+1)}$ and such that $\det(e^\ast_n , (y^\ast_n,a^\ast_{i}))$ is increasing with $i \in \{1,2,3\}$.

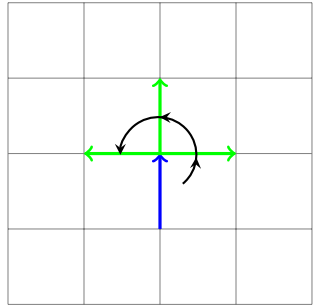
\begin{figure}[!ht]
\centering
\begin{tikzpicture}
\tikzset{arrow/.style={->,>=stealth,thick}}
\coordinate (start) at (2,1.5);
\coordinate (end2) at (2,1.5);
\coordinate (control1) at (10,5); 
\coordinate (control3) at (-5,5);
\draw[step=1cm,gray,very thin] (0,0) grid (4,4);
\draw[->,blue, very thick] (2,1) -- (2,2);
\draw[->,green, very thick] (2,2) -- (2,3);
\draw[->,green, very thick] (2,2) -- (3,2);
\draw[->,green, very thick] (2,2) -- (1,2);
\draw[arrow, postaction={decorate, decoration={markings,
	mark=at position 0.2 with {\arrow{>}},
	mark=at position 0.6 with {\arrow{>}}}}](2.3,1.6) arc (-50:180:0.5cm);
\end{tikzpicture}
\caption{In the exploration algorithm, edges are explored in the counter-clockwise sense.}
\label{fig:explorationOrder}
\end{figure}

In addition, our exploration algorithm will respect two limiting rules detailed below. Let $\mathcal{V}_{n} \subset \Za$ be the set of vertices already explored (or visited) by the exploration process until the $n$-th step. We do not explore an edge leading to an element of $\mathcal{V}_{n}$:
\begin{enumerate}[\bfseries (Rule 1)]
    \item A dual directed edge $(a^\ast,b^\ast)$ whose endpoint $b^\ast$ has already been visited (i.e., belongs to $\mathcal{V}_{n}$) will not be explored.
\end{enumerate}
For a finite and connected (w.r.t.~the $\ell^1$-norm) subset $A$ of $\Z^2$, denote by  $A^{\text{out}}_\infty$ the unique unbounded connected component of its complement and set
\[
\text{Fill}(A) := (A^{\text{out}}_\infty)^c ~.
\]
Then by construction, $A \subset \text{Fill}(A)$, the set $\text{Fill}(A)$ is still connected and finite, but this time without holes. We also apply this filling operation to subsets of $\Za$.
\begin{enumerate}[\bfseries (Rule 2)]
    \item A dual directed edge $(a^\ast,b^\ast)$ whose endpoint $b^\ast$ satisfies $b^\ast \in \text{Fill}(\mathcal{V}_{n}) \setminus \mathcal{V}_{n}$ will not be explored.
\end{enumerate}
Rule 2 has to be understood as follows. Such an edge $(a^\ast,b^\ast)$ with $b^\ast \in \text{Fill}(\mathcal{V}_{n}) \setminus \mathcal{V}_{n}$ targets a subpart of $\Za$ which is completely trapped and surrounded by the set of edges and vertices already explored. The exploration of $\f{b^\ast}{}$, which is a subset of $\f{x^\ast}{}$, actually is irrelevant to determine if $\f{x^\ast}{}$ reaches the outside of some large balls.

\medskip
Let us now describe precisely the exploration process\footnotemark of the forward set $\f{x^\ast}{}$ \footnotetext{See  \url{https://bennhenry.github.io/NeighPerc/} for simulations of the exploration process.}. It starts with initializing $\mathcal{V}_{0} = \{x^\ast\}$, $\E_{0} = \emptyset$ and $L^{(0)} = \big( (x^\ast,x^\ast+(1,0)),(x^\ast,x^\ast+(0,1)),(x^\ast,x^\ast-(1,0)),(x^\ast,x^\ast-(0,1))\big)$. Thus, for any $n$, while the list $L^{(n)}$ is not empty, we reveal the state of its first element and we update the sets $\mathcal{V}_{n}$, $\E_{n}$ and $L^{(n)}$ accordingly:
\begin{enumerate}
\item The current state of the $n$-th step is $X_n = (x^\ast_n,e^\ast_n,\omega^\ast(e^\ast_n))$ where $e^\ast_n = (x^\ast_n,y^\ast_n)$ is the first element of $L^{(n)}$.
\item We update the set of revealed edges, i.e., $\E_{n+1} = \E_{n} \cup \{e^\ast_n\}$.
\item If $\omega^{\ast}(e^\ast_n)=1$ ($e^\ast_n$ is open), then $\mathcal{V}_{n+1}=\mathcal{V}_{n} \cup \{y^\ast_n\}$ and $U^{(n)}$ contains the three new (directed, dual) edges starting at $y^\ast_n$, ordered in the counter-clockwise sense from $e^\ast_n$ (as in Figure~\ref{fig:explorationOrder}). Else ($\omega^{\ast}(e^\ast_n)=0$, i.e., $e^\ast_n$ is closed), $\mathcal{V}_{n+1}=\mathcal{V}_{n}$ and we set $U^{(n)}=\emptyset$.
\item We update the ordered list of edges to be explored, i.e., 
\[
L^{(n+1)} = U^{(n)}\colon L^{(n)}_{>1} ~,
\]
where $L^{(n)}_{>1}$ denotes the list $L^{(n)}$ deprived of its first element.
\item We clean $L^{(n+1)}$ according to Rules 1 and 2. Indeed, the set of vertices already visited being possibly augmented by one element (from $\mathcal{V}_{n}$ to $\mathcal{V}_{n+1}$), some edges of $L^{(n+1)}$ might not satisfy Rules~1 and 2 and should therefore be deleted.
\end{enumerate}

We denote by $\edf{x^\ast}$, and call the {\em explored cluster} of $x^\ast$, the set of vertices visited during the exploration process of $\df{x^\ast}{}$. When the exploration algorithm stops (and we will prove later that it stops $\PP_{(2,\varepsilon)}$-almost surely), the exploration of $\df{x^\ast}{}$ is almost exhaustive in the following sense:

\begin{lemma}
\label{lem:AlgoStop}
Let $x^\ast \in \Za$ be such that the exploration algorithm of $\df{x^\ast}{}$ stops. Then, the set $\df{x^\ast}{}$ is finite and
\[
\edf{x^\ast} \subset \df{x^\ast}{} \subset \text{{\upshape Fill}}(\edf{x^\ast}).
\]
\end{lemma}

\begin{proof}
By construction, the exploration algorithm adds to the $\mathcal{V}_n$'s only vertices that one can reach from $x^\ast$ through dual open edges. This justifies the inclusion $\edf{x^\ast} \subset \df{x^\ast}{}$.

The second inclusion requires more work. We first claim:
\begin{equation}
\label{ExForFini}
\text{the exploration algorithm stops $\,\Rightarrow\,$ $\df{x^\ast}{}$ is finite.}
\end{equation}
Let us proceed by contradiction and assume that the exploration algorithm stops, which means that $\edf{x^\ast}$ is finite, while $\df{x^\ast}{}$ is infinite. Note that the set $W=\text{Fill}(\edf{x^\ast})$ is finite too. Since $\df{x^\ast}{}$ is infinite, there exists a directed open edge $(y^\ast,z^\ast)$ such that $y^\ast,z^\ast \in \df{x^\ast}{}$, $y^\ast \in W$ and  $z^\ast \notin W$. Since $z^\ast$ belongs to the unbounded connected component of $\edf{x^\ast}$ then $y^\ast$ actually is in $\edf{x^\ast}$. So the vertex $y^\ast$ has been visited during the exploration process, and hence the edge $(y^\ast,z^\ast)$ has been added to the list of edges to be explored, but not really explored by the algorithm since $z^\ast \notin W$ (whereas $(y^\ast,z^\ast)$ is open). Necessarily the edge $(y^\ast,z^\ast)$ has been removed from the list before exploration due to Rules~1 or 2. This cannot be due to Rule~1 because $z^\ast \notin W$. This cannot be due to Rule~2 too since otherwise $z^\ast$ would be in $\text{Fill}(\edf{x^\ast}) = W$, which provides the contradiction.

Let us call the {\em free exploration process} the exploration process previously defined but {\em without Rule~2}. Henceforth, because the set $\df{x^\ast}{}$ is finite, it is not difficult to be convinced that the set of vertices $V^{\text{free}}(x^\ast)$ visited by the free exploration process is equal to $\df{x^\ast}{}$. Indeed, the set of directed edges visited by the free exploration process performs a tree rooted at $x^\ast$ and spanning the whole set $\df{x^\ast}{}$. Now, taking into account Rule 2 reduces the set $V^{\text{free}}(x^\ast)$ to $\edf{x^\ast}$ by removing subsets of vertices which are included in $\text{Fill}(\edf{x^\ast})$ by the definition of Rule 2. This forces
\[
\df{x^\ast}{} = V^{\text{free}}(x^\ast) \subset \text{Fill}(\edf{x^\ast}),
\]
as desired. 
\end{proof}

\subsection{Pivotal edges}
\label{sect:PivotalAlgo}

During the exploration process, we will sometimes find edges whose probability to be open is strictly larger than $1/2$ due to the presence of another already explored and closed dual edge, see Figure~\ref{fig:PivotalEdge}. These edges pose a problem for dominating $\text{For}(x^\ast)$ by a subcritical bond percolation cluster and need to be dealt with separately. Let us first make the notion of pivotality precise. 

For this, consider the edge $e^\ast_n = (x^\ast_n,y^\ast_n)$ revealed at the $n$-th step of the exploration algorithm of $\df{x^\ast}{}$. Let $z_n \in \Z^2$ be the starting point of the primal directed edge $e_n$ whose $e^\ast_n$ is the dual one. For the current section we assume by rotation invariance that $e^\ast_n$ is the east side of the unit square centered at $z_n$, i.e., $x^\ast_n = z_n + (1/2,-1/2)$ and $y^\ast_n = z_n + (1/2,1/2)$. Thus, let us respectively denote by $e^\ast_N$, $e^\ast_W$ and $e^\ast_S$ the elements of $\E^\ast$ obtained from $e^\ast_n$ by rotation (in the counter-clockwise sense) with center $z_n$ and angle $\pi/2$, $\pi$ and $3\pi/2$, which respectively correspond to the north, west and south sides of the unit square centered at $z_n$. The states of dual edges $e^\ast_n$, $e^\ast_N$, $e^\ast_W$ and $e^\ast_S$ depend of the states of primal edges starting at $z_n$ and then are dependent.

We call the dual edge $e^\ast_n$ {\em pivotal} for the exploration process of $\df{x^\ast}{}$ if at least one of the edges $e^\ast_N$, $e^\ast_W$ and $e^\ast_S$ has already been explored by the exploration algorithm before step $n$ and is closed, see Figure~\ref{fig:PivotalEdge}. The fact that $e^\ast_n$ is pivotal or not does not depend on its own state.

\begin{figure}[!ht]
\begin{center}
\begin{tikzpicture}
\draw (0,0) -- (0,1) -- (1,1) -- (1,0) -- (0,0);
\draw[fill=red, color=red] (0.5,0.5) circle  (0.075);
\draw[->,green,very thick] (1,0) -- (1,1); 
\begin{scope}[rotate around={180:(0.5,0.5)}]
\draw[fill=red, color=red] (0.5,0.5) circle  (0.075);
\draw[->,very thick,color=red,opacity=1] (0.5,0.5) -- (1.5,0.5);
%\draw[color=red, very thick ,opacity= 0.4] (0,0.5) -- (0.5,0.5);
\draw[->,blue,very thick, opacity=0.3] (1,0) -- (1,1); 
\draw (1,0.5) node[color=blue,cross,rotate=10] {};
\end{scope}
\end{tikzpicture}
\caption{\label{fig:PivotalEdge}The directed dual edge $e^\ast_n$ is the east edge of the unit square and is represented in green. It is pivotal for the exploration process since the west (dual) edge $e^\ast_W$ has been already visited by the exploration process and is closed (in light blue). This is why its primal associated edge (in red) is open.}
\end{center}
\end{figure}
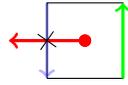

The pivotal edge $e^\ast_n$ constitutes an obstacle in order to stochastically dominate $\df{x^\ast}{}$ by a subcritical bond percolation cluster since its probability to be open can be strictly larger than $1/2$ due to the presence of an already explored and closed dual edge among $e^\ast_N$, $e^\ast_W$ and $e^\ast_S$. Indeed, following the example of Figure~\ref{fig:PivotalEdge}, the probability for the edge $e^\ast_n$ to be open, knowing that $e^\ast_W$ is already explored and closed, is equal to
\[
\mathbb{P}_{(2,\varepsilon)} \left(
\begin{tikzpicture}[baseline=0.6em,scale=0.65]
\draw (0,0) -- (0,1) -- (1,1) -- (1,0) -- (0,0);
\draw[fill=red, color=red] (0.5,0.5) circle  (0.075);
\draw[->,blue,very thick] (1,0) -- (1,1); 
\begin{scope}[rotate around={180:(0.5,0.5)}]
\draw[fill=red, color=red] (0.5,0.5) circle  (0.075);
\draw[->,very thick,color=red,opacity=1] (0.5,0.5) -- (1.5,0.5);
\draw[color=red, very thick ,opacity= 0.4] (0,0.5) -- (0.5,0.5);
\draw[->,blue,very thick, opacity=0.3] (1,0) -- (1,1); 
\draw (1,0.5) node[color=blue,cross,rotate=10] {};
\end{scope}
\end{tikzpicture} \; \Big| \; \begin{tikzpicture}[baseline=0.6em,scale=0.65]
\draw (0,0) -- (0,1) -- (1,1) -- (1,0) -- (0,0);
\draw[fill=red, color=red] (0.5,0.5) circle  (0.075); 
\begin{scope}[rotate around={180:(0.5,0.5)}]
\draw[fill=red, color=red] (0.5,0.5) circle  (0.075);
\draw[->,very thick,color=red,opacity=1] (0.5,0.5) -- (1.5,0.5);
\draw[->,blue,very thick, opacity=0.3] (1,0) -- (1,1); 
\draw (1,0.5) node[color=blue,cross,rotate=10] {};
\end{scope}
\end{tikzpicture} \; \right) = \frac{\mathbb{P}_{(2,\varepsilon)} \left(
\begin{tikzpicture}[baseline=0.6em,scale=0.65]
\draw (0,0) -- (0,1) -- (1,1) -- (1,0) -- (0,0);
\draw[fill=red, color=red] (0.5,0.5) circle  (0.075);
\draw[->,blue,very thick] (1,0) -- (1,1); 
\begin{scope}[rotate around={180:(0.5,0.5)}]
\draw[fill=red, color=red] (0.5,0.5) circle  (0.075);
\draw[->,very thick,color=red,opacity=1] (0.5,0.5) -- (1.5,0.5);
\draw[color=red, very thick ,opacity= 0.4] (0,0.5) -- (0.5,0.5);
\draw[->,blue,very thick, opacity=0.3] (1,0) -- (1,1); 
\draw (1,0.5) node[color=blue,cross,rotate=10] {};
\end{scope}
\end{tikzpicture} \; \right)}{\mathbb{P}_{(2,\varepsilon)} \left(
\begin{tikzpicture}[baseline=0.6em,scale=0.65]
\draw (0,0) -- (0,1) -- (1,1) -- (1,0) -- (0,0);
\draw[fill=red, color=red] (0.5,0.5) circle  (0.075);
%\draw[->,blue,very thick] (1,0) -- (1,1); 
\begin{scope}[rotate around={180:(0.5,0.5)}]
\draw[fill=red, color=red] (0.5,0.5) circle  (0.075);
\draw[->,very thick,color=red,opacity=1] (0.5,0.5) -- (1.5,0.5);
%\draw[color=red, very thick ,opacity= 0.4] (0,0.5) -- (0.5,0.5);
\draw[->,blue,very thick, opacity=0.3] (1,0) -- (1,1); 
\draw (1,0.5) node[color=blue,cross,rotate=10] {};
\end{scope}
\end{tikzpicture} \; \right)} = \frac{\varepsilon/4 + (1-\varepsilon)/3}{1/2+\varepsilon/4} = \frac{2}{3} \times \frac{1-\varepsilon/4}{1+\varepsilon/2},
\]
which tends to $2/3$ as $\varepsilon \downarrow 0$. This is why pivotal edges have to be treated separately.

However, encountering a pivotal edge during the exploration process is also a good news; its revealment constitutes an opportunity to stop the algorithm, as stated in the following result.

\begin{lemma}
\label{lem:pivot}
Using notations of Section~\ref{sect:ExploAlgo}, let $L^{(n)}$ be the list of edges remaining to be explored at the $n$-th step of the exploration algorithm and $e^\ast_n$ its first element. If $e^\ast_n$ is pivotal for the exploration algorithm, then $L^{(n)}$ is actually reduced to $e^\ast_n$, i.e., $|L^{(n)}| = 1$. Consequently, if $e^\ast_n$ is closed, then the exploration process stops.
\end{lemma}

\begin{proof}
Assume we are at the $n$-th step of the exploration algorithm of the forward set $\df{x^\ast}{}$ and we use the notations introduced before: $e^\ast_n$ is going to be revealed during this $n$-th step and is pivotal, i.e., among $e^\ast_N$, $e^\ast_W$ and $e^\ast_S$, at least one has been already explored and is closed. Hence, several cases have to be distinguished. Some of them will turn out to be forbidden by Rules 1 and 2. For all the remaining cases, we will prove that $L^{(n)}$ is actually reduced to $e^\ast_n$.

First, let us remark that the north edge $e^\ast_N$ could not be explored before step $n$. Otherwise its starting point $y^\ast_n$ would have already been visited and would be stocked in $\mathcal{V}_n$. Then Rule~1 prevents the algorithm to explore the edge $e^\ast_n$. So, only $e^\ast_W$ and $e^\ast_S$ could be explored before step $n$ and by hypothesis ($e^\ast_n$ is pivotal) at least one of them is closed.

\medskip

\textbf{Case 1.} {\em Only the edge $e^\ast_W$ has been explored before step $n$.} The case where the only edge already explored is $e^\ast_S$ is completely similar and will not be treated. For this case, we refer to Figure~\ref{fig:Case1} to help the reader. In a first time the exploration process reaches the edge $e^\ast_W$ which is revealed and closed. So there is a path $\pi$ from $x^\ast$ to $e^\ast_W$ of dual open edges explored by the algorithm. In a second time, the exploration process reaches the edge $e^\ast_n$ and its starting point $x^\ast_n$. So there is a path $\pi'$ from $x^\ast$ to $e^\ast_n$ of dual open edges explored by the algorithm. The directed paths $\pi$ and $\pi'$ coincide from the root $x^\ast$ to some bifurcating vertex from which they are disjoint by Rule~1 (see Figure~\ref{fig:Case1}). We denote by $\pi_0$ and $\pi'_0$ their respective pieces beyond their bifurcating vertex. Then, by planarity, two situations may occur. Either $\pi_0$ and $\pi'_0$ surround and trap the edge $e^\ast_n$ (see the left part of Figure~\ref{fig:Case1}) but in this case the exploration of $e^\ast_n$ contradicts Rule~2. This situation is forbidden. Or $\pi_0$ and $\pi'_0$ surround and trap $e^\ast_W$ (see the right part of Figure~\ref{fig:Case1}). Since the exploration algorithm proceeds in the depth-first fashion and in the counter-clockwise sense, all the edges starting from $\pi$ and from the right side of $\pi'_0$ (recall that $\pi'_0$ is oriented which identifies its right and left sides) have already been explored. The remaining edges to be explored are among $e^\ast_n$ and the edges starting from the left side of $\pi'_0$. However these later edges are irrelevant thanks to Rule~2 so that $e^\ast_n$ actually is the last edge to be explored at this stage of the algorithm.

\begin{figure}[!ht]
\begin{center}
\includegraphics[width=11cm,height=6cm]{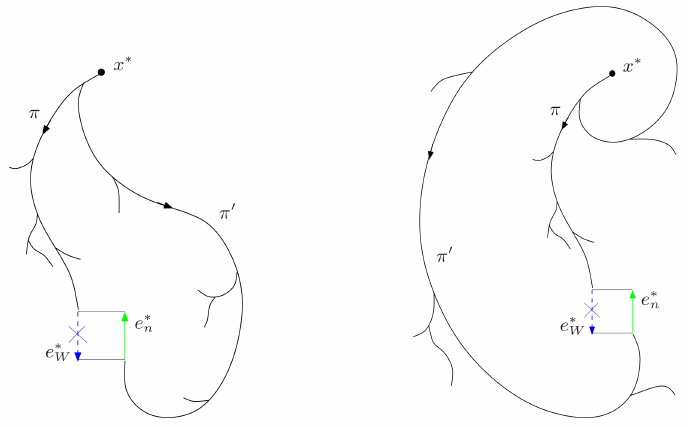}
\caption{\label{fig:Case1} Both pictures represent two paths $\pi$ and $\pi'$ performed during the exploration process of the forward set of $x^\ast$ until the $n$-th step of the algorithm. $\pi$ and $\pi'$ respectively reach the directed dual edges $e^\ast_W$ and $e^\ast_n$: $e^\ast_W$ has been previously revealed and is closed (in blue) while $e^\ast_n$ is about to be revealed (in green). The paths $\pi$ and $\pi'$ surround the edge $e^\ast_n$ to the left and the edge $e^\ast_W$ to the right.}
\end{center}
\end{figure}

\medskip
\textbf{Case 2.} {\em Both edges $e^\ast_W$ and $e^\ast_S$ have been explored before step $n$ and $e^\ast_W$ is closed.} By Rule~1, we can assert that $e^\ast_W$ is visited before $e^\ast_S$. We can exhibit two open paths $\pi$ and $\pi'$ from $x^\ast$ to $e^\ast_W$ and $e^\ast_S$ respectively. We also denote by $\pi_0$ and $\pi'_0$ their respective pieces beyond their bifurcating vertex. As before, using planarity, we can distinguish two cases. In the first one (see the left part of Figure~\ref{fig:Case2}), $\pi$ and $\pi'$ surround and trap the edges $e^\ast_S$ and $e^\ast_n$. This case is forbidden by Rule~2 since $e^\ast_S$ and $e^\ast_n$ target toward a region already surrounded by the exploration algorithm. It then remains the case where $\pi$ and $\pi'$ surround the edge $e^\ast_W$ (see the right part of Figure~\ref{fig:Case2}). Note that all the edges starting from $\pi$ and from the right side of $\pi'_0$ have already been explored. So, when the exploration algorithm reaches $e^\ast_S$, this edge is the last opportunity for the algorithm to continue its exploration and $e^\ast_S$ is a pivotal edge. If $e^\ast_S$ was closed, then the algorithm would stop and $e^\ast_n$ would never be explored, which is impossible. So $e^\ast_S$ was necessarily open and the exploration of $\df{x^\ast}{}$ continues after that until step $n$ at which $e^\ast_n$ has to be revealed (see the right-hand side of Figure~\ref{fig:Case2}). The edge $e^\ast_n$ is then the last edge to be explored at this stage of the algorithm.

Let us point out here that, when one explores the edge $e^\ast_n$, it is impossible, by planarity, that both edges $e^\ast_W$ and $e^\ast_S$ have already been explored and declared closed. This remark will be used in Section~\ref{sec_upperbound_cor}.

\begin{figure}[!ht]
\begin{center}
\includegraphics[width=11cm,height=6cm]{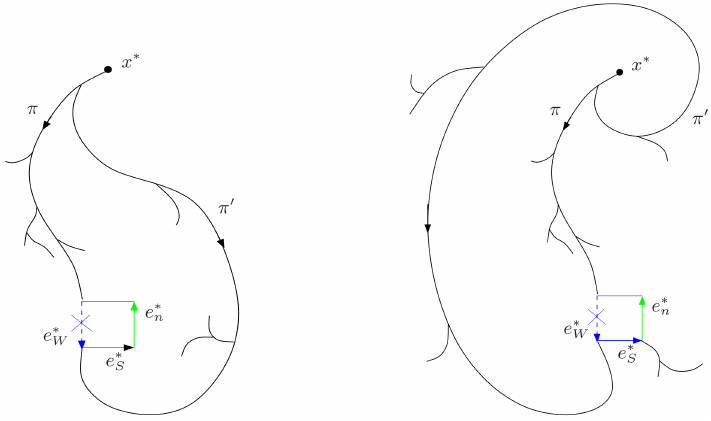}
\caption{\label{fig:Case2} Unlike Figure~\ref{fig:Case1}, the paths $\pi$ and $\pi'$ performed during the exploration process of $\df{x^\ast}{}$ until step $n$, respectively reach $e^\ast_W$ and $e^\ast_S$. On the left part, is depicted the case where $\pi$ and $\pi'$ surround $e^\ast_S$ and $e^\ast_n$ which is forbidden by Rule~2. On the right part, $\pi$ and $\pi'$ surround $e^\ast_W$ which is closed, while $e^\ast_S$ is open.}
\end{center}
\end{figure}

\medskip
\textbf{Case 3.} {\em Both edges $e^\ast_W$ and $e^\ast_S$ have been explored before step $n$ and $e^\ast_W$ is open}, meaning that $e^\ast_S$ is closed (since $e^\ast_n$ is pivotal). The same arguments as in the first two cases will allow us to conclude. As before, $e^\ast_W$ has been visited before $e^\ast_S$ (by Rule~1). Let $\pi$ and $\pi'$ be two open paths from $x^\ast$ to $e^\ast_W$ and $e^\ast_n$ respectively, explored by the algorithm. We also denote by $\pi_0$ and $\pi'_0$ their respective pieces beyond their bifurcating vertex. Once again, the case where $\pi_0$ and $\pi'_0$ surround and trap the edge $e^\ast_n$ is forbidden by Rule~2. We focus on the case where $\pi_0$ and $\pi'_0$ surround the edges $e^\ast_W$ and $e^\ast_S$. As previously, using planarity and the exploration in the depth first fashion and in the counter-clockwise sense, $e^\ast_n$ is then the last dual edge to be explored.
\end{proof}

\subsection{The geometric structure of the explored cluster}

Note that the result stated in Lemma~\ref{lem:pivot} is not only a termination property of the algorithm but also informs us about the geometry of $\edf{x^\ast}$. Roughly speaking, $\edf{x^\ast}$ can be seen as a disjoint union of clusters of vertices visited during the algorithm, linked by pivotal edges, see Figure~\ref{fig:cluster_chain}.

More precisely, let $\mathcal{T}_{x^\ast}$ be the number of steps of the exploration algorithm starting at $x^\ast$, i.e., 
\[
\mathcal{T} = \mathcal{T}_{x^\ast} = \inf \big\{ n \geq 1 \colon L^{(n+1)} = \emptyset \big\} \in \mathbb{N} \cup \{\infty\}.
\]
The exploration algorithm stops if and only if $\mathcal{T} < \infty$. In the sequel, we will prove that this algorithm stops with $\PP_{(2,\varepsilon)}$-probability $1$.
To give a rigorous description of the decomposition obtained from Lemma~\ref{lem:pivot} set $T_{0} = 0$ and let
\[
T_1 = \inf \big\{ 1 \leq n \leq \mathcal{T} \colon \text{the edge explored at the $n$-th step is pivotal} \big\} \in \mathbb{N} \cup \{\infty\}
\]
be the {\em first pivotal step} (if it occurs, using the convention $\inf \emptyset = \infty$). Now we can define the {\em first visited cluster} as the set of visited vertices until the first pivotal edge (if it occurs):
\[
\kedf{x^\ast}{1} =  \big\{ X_n(1) \colon 1 \leq n \leq T_{1} \wedge \mathcal{T} \big\}.
\]
Henceforth, several cases can be distinguished. When $T_1 = \infty$, the algorithm never meets pivotal edges, no matter if the exploration algorithm stops ($\mathcal{T} < \infty$) or not ($\mathcal{T} = \infty$). The sets $\kedf{x^\ast}{1}$ and $\edf{x^\ast}$ coincide in this case. When $T_1 < \infty$, the exploration algorithm discovers a first pivotal edge, say $e^\ast(1)$. So the set $\kedf{x^\ast}{1}$ is finite and contains by construction the starting vertex of $e^\ast(1)$, but not its ending one denoted by $x^\ast(1)$. If the pivotal edge $e^\ast(1)$ is closed, then the exploration process stops by Lemma~\ref{lem:pivot} since, at the $T_1$-th step, the list of edges to be explored is reduced to $e^\ast(1)$. In this case $T_1 = \mathcal{T} < \infty$ and $\kedf{x^\ast}{1} = \edf{x^\ast}$. Otherwise, this first pivotal edge is open, and the process continues from the vertex $x^\ast(1)$.

By induction, we can then define for any integer $k \geq 2$ such that the pivotal times $T_1,\ldots,T_{k-1}$ are all finite and whose corresponding pivotal edges are all open, the {\em $k$-th pivotal step} or {\em pivotal time} $T_k$ by
\[
T_{k} = \inf \big\{ T_{k-1} < n \leq \mathcal{T} \colon \text{the edge explored at the $n$-th step is pivotal} \big\} \in \mathbb{N} \cup \{\infty\}
\]
and the {\em $k$-th visited cluster} by 
\[
\kedf{x^\ast}{k} =  \big\{ X_n(1) \colon T_{k-1}+1 \leq n \leq T_{k} \wedge \mathcal{T} \big\}.
\]
As above, a trichotomy appears:
\begin{itemize}
\item Either $T_k$ is infinite, then the exploration process never stops and reveals exactly $k-1$ pivotal edges which are all open. While the first $(k-1)$ visited clusters are finite, the $k$-th counts infinitely many elements.
\item Or $T_k$ is finite and the algorithm discovers a $k$-th pivotal step, which turns out to be closed. Then the exploration process stops by Lemma~\ref{lem:pivot}. Exactly $k$ pivotal edges have been discovered by the algorithm and only the first $k-1$ are open. The $k$ visited clusters are all finite.
\item Or $T_k$ is finite and the algorithm discovers a $k$-th pivotal step, say $e^\ast(k)$, which turns out to be open. In this case, the exploration process continues from the ending vertex of $e^\ast(k)$.
\end{itemize}

Let us finally denote by $\mathcal{T}_{\text{piv}}$ the number of open pivotal edges revealed by the exploration process. The previous analysis leads to the following decomposition of the explored cluster $\edf{x^\ast}$:
\begin{equation}
\label{decomposition}
\edf{x^\ast} = \bigcup_{1 \leq k \leq \mathcal{T}_{\text{piv}}+1} \kedf{x^\ast}{k}.
\end{equation}

\begin{figure}[!ht]
\centering
\includegraphics[width=.75\textwidth]{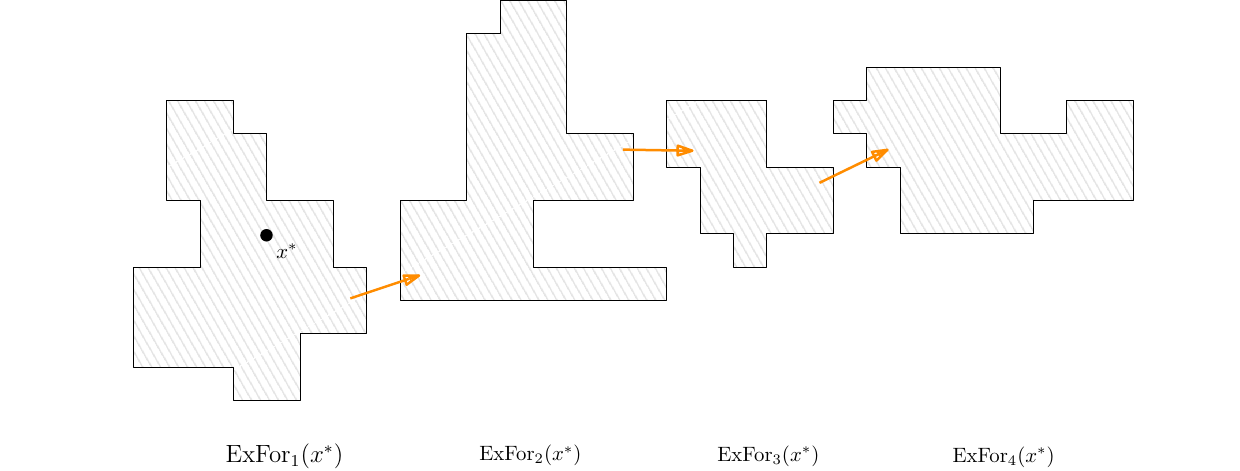}
\caption{\label{fig:cluster_chain} Schematic representation of the structure of the explored forward cluster $\text{ExFor}(x^*)$ as the union of the $k$-th explored clusters (shaded) linked by pivotal edges (orange).}
\end{figure}

\subsection{Proof of Theorem~\ref{thm:clusterSize}}
\label{sect:VisitedCluster}

In order to conclude the proof of Theorem~\ref{thm:clusterSize}, we need two more ingredients in addition to the decomposition~\eqref{decomposition}. The first one asserts that each of the visited clusters is actually quite small with high probability. 

\begin{theorem}
\label{theo:subGeomClusterSize}
For any $\varepsilon \geq 0$, there exist constants $C,c>0$ such that, for any integer $k \ge 1$,
\[
\forall n \in \mathbb{N} , \; \PP_{(2,\varepsilon)} \big( | \kedf{o^\ast}{k} | \geq n \big) \leq C e^{-c n}.
\]
\end{theorem}

The proof of Theorem~\ref{theo:subGeomClusterSize} requires a considerable amount of work and the rest of this and the next section is devoted to it. It consists essentially of two main steps. First, we compare the $k$-th visited cluster in the $(2,\varepsilon)$-model to the cluster of the origin in a standard i.i.d.~Bernoulli bond percolation model, with parameter $1/2-\varepsilon/4$ but {\em with some constraints}, see Proposition~\ref{prop:stochasticDomination}. These constraints have to be understood as a family of patterns which are forbidden to be used by percolating paths. When $\varepsilon>0$ we can simply drop the constraint and use standard results for Bernoulli bond percolation to obtain Theorem~\ref{thm:clusterSize}.
However, when $\varepsilon = 0$, we are exactly at the critical value and the decay in the model without constraints is too slow for our purposes. Therefore, a more careful analysis based on enhancement techniques is needed. This will be carried out in Section~\ref{sect:enhancement}.

But before we do this, let us for now assume that Theorem~\ref{theo:subGeomClusterSize} holds and complete the proof of  Theorem~\ref{thm:clusterSize} by providing the last ingredient which tells us that the number of visited clusters in the decomposition~\eqref{decomposition} is subgeometric. Its proof is straightforward and given at the end of this section.

\begin{lemma}
\label{lem:SubGeom}
In the $(2,\varepsilon)$-model, with $\varepsilon \geq 0$, the random variable $\mathcal{T}_{\text{piv}}$ is subgeometric, i.e., 
\[
\forall n \in \mathbb{N} , \; \PP_{(2,\varepsilon)}(\mathcal{T}_{\text{piv}} \geq n) \leq (2/3)^n.
\]
\end{lemma}

With these two results, we are ready to prove Theorem~\ref{thm:clusterSize}.

\begin{proof}[Proof of Theorem~\ref{thm:clusterSize}]
The decomposition \eqref{decomposition} gives
\[
| \edf{x^\ast} | = \sum_{1\leq k\leq \mathcal{T}_{\text{piv}}+1} | \kedf{x^\ast}{k} |
\]
which implies, thanks to Theorem~\ref{theo:subGeomClusterSize} and Lemma~\ref{lem:SubGeom}, the upper bound
\begin{equation*}
\begin{split}
\PP_{(2,\varepsilon)} \big( | \edf{x^\ast} | \geq n \big) & \leq \PP_{(2,\varepsilon)} \Big( \mathcal{T}_{\text{piv}} < \lfloor \sqrt{n} \rfloor , \, \sum_{1\leq k \leq \lfloor \sqrt{n} \rfloor} | \kedf{x^\ast}{k} | \geq n \Big) \, + \, \PP_{(2,\varepsilon)} \big( \mathcal{T}_{\text{piv}} \geq \lfloor \sqrt{n} \rfloor \big) \\
& \leq  \sum_{1\leq k\leq \lfloor \sqrt{n} \rfloor} \PP_{(2,\varepsilon)} \big( | \kedf{x^\ast}{k} | \geq \sqrt{n} \big) \, + \, \Big( \frac{2}{3} \Big)^{\lceil\sqrt{n}\rceil} \\
& \leq  C_1 e^{- c_1 \sqrt{n}},
\end{split}
\end{equation*}
where $C_1,c_1$ are positive constants. As a conclusion, the set of vertices visited by the exploration algorithm is almost-surely finite, meaning that the exploration process of $\df{x^\ast}{}$ almost-surely stops after a finite number of steps. Lemma~\ref{lem:AlgoStop} then applies and the forward set $\df{x^\ast}{}$ is almost-surely finite too, and is included in $\text{Fill}(\edf{x^\ast})$. 

To conclude, note that $|\edf{x^\ast}| \leq n$ implies $|\df{x^\ast}{}| \leq 4 n^2$. Indeed, $|\edf{x^\ast}| \leq n$ means that $\edf{x^\ast}$ is included in the square $x^\ast + \llbracket -n,n \rrbracket^2$ (this is a connected set containing $x^\ast$). By definition, the same inclusion also holds for $\text{Fill}(\edf{x^\ast})$ and hence, by Lemma~\ref{lem:AlgoStop} for $\df{x^\ast}{}$, this leads to $|\df{x^\ast}{}| \leq 4 n^2$. Therefore,
\[
\PP_{(2,\varepsilon)} \big( | \df{x^\ast}{} | \geq n \big) \leq \PP_{(2,\varepsilon)} \Big( | \edf{x^\ast} | \geq \sqrt{n}/2 \Big) \leq  C_1 e^{- c_2 n^{1/4}},
\]
where $C_1,c_2$ are positive constants.
\end{proof}

\begin{proof}[Proof of Lemma~\ref{lem:SubGeom}]
Let us denote by $\mathcal{F}_{n}$ the $\sigma$-algebra generated by $X_k = (x^\ast_k , e^\ast_k ,\omega^\ast(e^\ast_k))$, the states of the exploration process until step $n-1$ (included). That is, conditionally on $\mathcal{F}_{n}$, the $n$-th dual edge which will be explored is known but its state is yet unknown. Conditionally on $\mathcal{F}_{T_{k-1}}$ the event $T_k < \infty$ means that $T_{k-1} < \infty$ and especially that the algorithm explores a $k$-th visited cluster, which is possible only if the $(k-1)$-th pivotal edge was open. So,
\[
\PP_{(2,\varepsilon)} \big( T_k < \infty \,|\, \mathcal{F}_{T_{k-1}} \big) \leq \mathds{1}_{T_{k-1} < \infty} \PP_{(2,\varepsilon)} \big( \text{the $(k-1)$-th pivotal edge is open} \,|\, \mathcal{F}_{T_{k-1}} \big).
\]
Reusing the notations of Section~\ref{sect:PivotalAlgo}, let the pivotal edge be denoted by $e^\ast$ and assumed to be the east side of a unit square while the other edges are denoted by $e^\ast_N$, $e^\ast_W$ and $e^\ast_S$. To compute the above conditional probability, the worst case occurs when only one edge among $e^\ast_N$, $e^\ast_W$ and $e^\ast_S$ is closed (the case where both $e^\ast_W$ and $e^\ast_S$ are closed is forbidden, see the proof of Lemma~\ref{lem:pivot} and Figure~\ref{fig:Case2}). Then,
\[
\PP_{(2,\varepsilon)} \big( \text{the $(k-1)$-th pivotal edge is open} \,|\, \mathcal{F}_{T_{k-1}} \big) \leq \frac{2}{3} \times \frac{1-\varepsilon/4}{1+\varepsilon/2}\leq \frac{2}{3}.
\]
We can now conclude,
\begin{align*}
\PP_{(2,\varepsilon)} \big( T_k < \infty \big)&= \mathbb{E} \big[ \PP_{(2,\varepsilon)} \big( T_k < \infty \,|\, \mathcal{F}_{T_{k-1}} \big) \big] \\
&=\mathbb{E} \big[ \mathds{1}_{T_{k-1}<\infty} \PP_{(2,\varepsilon)} \big( \text{the $(k-1)$-th pivotal edge is open} \,|\, \mathcal{F}_{T_{k-1}} \big) \big]\\
&\leq \frac{2}{3} \PP_{(2,\varepsilon)} \big( T_{k-1} < \infty \big)
\end{align*}
and the prove is finished by induction.
\end{proof}

\section{Domination by a percolation model under constraints}
\label{sec:UnderConstraints}
This section is dedicated to the first ingredient of the proof of the exponential tails of the size of the visited clusters $\kedf{x^\ast}{k}$, $k\geq 1$. Roughly speaking, we compare the $k$-th visited cluster in the $(2,\varepsilon)$-model to the cluster of the origin in an i.i.d.~Bernoulli bond percolation model with parameter $1/2-\varepsilon/4$ {\em under some constraints}, see Proposition~\ref{prop:stochasticDomination}. These constraints have to be understood as a family of patterns which are forbidden to be used by percolating paths. 

In Section~\ref{sect:forbidden}, we introduce the forbidden patterns and state Proposition~\ref{prop:stochasticDomination}, which will then be proved in Section~\ref{sect:ProofForbiddenPatterns}.

\subsection{Independent bond percolation and forbidden patterns.}
\label{sect:forbidden}

Let us now introduce the independent and undirected bond percolation model on the dual lattice $\Za$ by which we want to dominate each cluster $\kedf{x^\ast}{k}$. Let $\Omega' = \{0,1\}^{\E'}$ be the set of undirected bond configurations where $\E'$ denotes the set of undirected edges of $\Za$. The configuration set $\Omega'$ is endowed with the product measure $\mathcal{P}_{q}$ where each coordinate is a Bernoulli random variable of parameter $q \in [0,1]$.

We are interested in the event $\{ o^\ast \rightsquigarrow o^\ast + \partial \Lambda_{n} \}$ in which the origin $o^\ast$ is linked by an (undirected) open path to the boundary $o^\ast + \partial \Lambda_{n}$, where $\partial \Lambda_{n} = \{ z \in \Z^2 \colon \|z\|_{\ell^\infty} = n\}$, but under constraints. For this purpose, let us introduce two {\em forbidden patterns} to which we will refer by using the symbol \scalebox{0.7}{\pattern}. The first one is defined by its open and closed edges:
\begin{itemize}
\item $\{o^\ast,o^\ast+(0,1)\}$, $\{o^\ast+(0,1),o^\ast+(1,1)\}$, $\{o^\ast+(1,0),o^\ast+(1,1)\}$, $\{o^\ast+(1,0),o^\ast+(2,0)\}$ and $\{o^\ast+(2,0),o^\ast+(2,1)\}$ are open;
\item $\{o^\ast,o^\ast+(1,0)\}$ and $\{o^\ast+(1,1),o^\ast+(2,1)\}$ are closed.
\end{itemize}
The second forbidden pattern is obtained by rotating the first pattern (in the counter-clockwise sense) by the angle $\pi/2$ around the center $o^\ast$, see Figure~\ref{fig:ForbiddenPatterns} for an illustration. Hence, the set of forbidden patterns is given by the first and second forbidden patterns and all their translations in $\Za$. We now define the event denoted by
\[
\big\{ o^\ast \rightsquigarrow o^\ast + \partial \Lambda_{n} \; \mbox{without \scalebox{0.7}{\pattern}} \big\}
\]
in which $o^\ast$ is linked to the boundary $o^\ast + \partial \Lambda_{n}$ by an{\em  open path using none of the forbidden patterns}, see Figure~\ref{fig:ForbiddenPatterns} for examples. Note that the forbidden pattern is allowed to appear in a configuration and its individual edges may be used for an open path to the boundary, we only exclude open paths which use the {\em whole pattern}. 
\begin{figure}[ht!]
\begin{center}
\begin{tabular}{cp{1cm}c}
\includegraphics[scale=1]{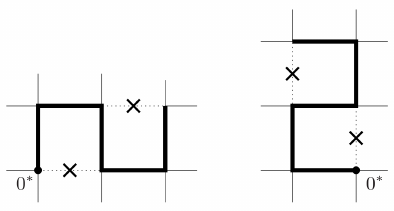} & & \includegraphics[scale=1]{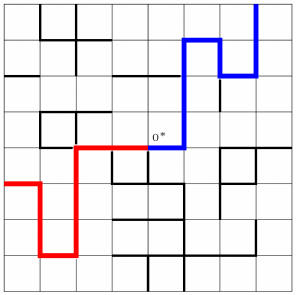}
\end{tabular}
\caption{Left: the first and second forbidden patterns. The small black crosses emphasize the fact that the corresponding edges are closed. 
Right: here is a configuration of the box $o^\ast+\Lambda_4$ satisfying the event $\{o^\ast \rightsquigarrow o^\ast + \partial \Lambda_{4} \text{ without }\scalebox{0.7}{\protect\pattern}\}$ thanks to the red path (for instance), but not via the blue one which uses a forbidden pattern.
}
\label{fig:ForbiddenPatterns}
\end{center}
\end{figure}

We are ready to state the main result of this section. Roughly speaking, it says that each visited cluster $\kedf{x^\ast}{k}$ is dominated by the cluster of the origin $o^\ast$ in a Bernoulli percolation model under constraints. 

\begin{prop}
\label{prop:stochasticDomination}
For any $\varepsilon \geq 0$, $x^\ast \in \Za$ and for any integer $k \ge 1$, the following inequality holds for any integer $n$,
\[
\PP_{(2,\varepsilon)} \big( | \kedf{x^\ast}{k} | \geq n \big) \leq \mathcal{P}_{1/2-\varepsilon/4} \big( o^\ast \rightsquigarrow o^\ast + \partial \Lambda_{n} \; \mbox{without \scalebox{0.7}{\pattern}} \big).
\]
\end{prop}

Proposition~\ref{prop:stochasticDomination} turns out to be sufficient to prove Theorem~\ref{thm:clusterSize} in the case where $\varepsilon > 0$.

\begin{proof}[Proof of Theorem~\ref{thm:clusterSize} when $\varepsilon > 0$] In this case, the parameter $1/2-\varepsilon/4$ is strictly smaller than $1/2$ and then corresponds to the subcritical regime for the independent bond percolation model in $\Z^2$ for which a sharp transition is well-known (see for instance~\cite[Chapter 6]{grimmett1999percolation}). Therefore, we can simply neglect the forbidden patterns and just write
\begin{eqnarray}
\label{BoundEps>0}
\PP_{(2,\varepsilon)} \big( | \kedf{x^\ast}{k} | \geq n \big) & \leq &  \mathcal{P}_{1/2-\varepsilon/4} \big( o^\ast \rightsquigarrow o^\ast + \partial \Lambda_{n} \; \mbox{without \scalebox{0.7}{\pattern}} \big) \nonumber \\
& \leq &  \mathcal{P}_{1/2-\varepsilon/4} \big( o^\ast \rightsquigarrow o^\ast + \partial \Lambda_{n} \big) \\
& \leq & C e^{-c n} \nonumber
\end{eqnarray}
for any $n\geq 1$, where $C,c$ are positive constants (depending on $\varepsilon$).
\end{proof}

Remark that, when $\varepsilon = 0$, the upper bound~\eqref{BoundEps>0} is merely $\mathcal{P}_{1/2}(o^\ast \rightsquigarrow o^\ast + \partial \Lambda_{n})$ and this decays too slowly as $n$ tends to infinity to be useful for our purposes. In this situation, we need to benefit from the forbidden patterns through a sharper analysis that we achieve in Section~\ref{sect:enhancement} using enhancement techniques.

\subsection{Proof of Proposition~\ref{prop:stochasticDomination}}
\label{sect:ProofForbiddenPatterns}
Let us start by introducing some more notation. Let $k$ be some positive integer. If the $(k-1)$-th pivotal time $T_{k-1}$ is infinite, or finite but the corresponding pivotal edge is closed, then the $k$-th visited cluster $\kedf{x^\ast}{k}$ is empty by definition; there is nothing to do in this case. So, in the sequel, we focus on the case where $T_{k-1} < \infty$ and the corresponding pivotal edge is open, so that $\kedf{x^\ast}{k}$ is non-empty. The exploration of the visited cluster $\kedf{x^\ast}{k}$ ends at step $T_k \wedge \mathcal{T} \in \mathbb{N}\cup\{\infty\}$, i.e., according to the exploration or not of a $k$-th pivotal edge. Let us introduce the random variable
\[
T := \big( T_k \wedge \mathcal{T} - 1 \big) \mathds{1}_{T_k \wedge \mathcal{T} < \infty} + \infty \mathds{1}_{T_k \wedge \mathcal{T} = \infty}
\]
which indicates, when at least $T_k$ or $\mathcal{T}$ is finite, the index of the last explored edge before the exploration of the $k$-th visited cluster $\kedf{x^\ast}{k}$ terminates. Now, we go through all the edges explored during the $k$-th visited cluster, indexed by $m \in \{1,\ldots,T-T_{k-1}\}$. Recall that the $n$-th edge explored by the exploration algorithm is denoted by $e^\ast_n$ and its state by $\omega^\ast(e^\ast_n)$. To be shorter, we set in this proof for $1\leq m\leq T-T_{k-1}$,
\[
Y_{m} := e^\ast_{T_{k-1}+m} \; \mbox{ and } \; Z_m := \omega^\ast(e^\ast_{T_{k-1}+m})
\]
the $m$-th explored edge during the exploration of $\kedf{x^\ast}{k}$, and its state. Also the $\sigma$-algebra $\mathcal{F}_m$, for $1\leq m\leq T-T_{k-1}$, encodes the information collected by the exploration process until step $m$, including the edge $Y_m$ but not its state $Z_m$:
\[
\mathcal{F}_{m} = \sigma \big( e^\ast_1 , \ldots , e^\ast_{T_{k-1}+m} ; \omega^\ast(e^\ast_1) , \ldots , \omega^\ast(e^\ast_{T_{k-1}+m-1}) \big).
\]

Given a configuration $\omega \in \Omega$, recall that $\omega^\ast \in \Omega^\ast$ denotes its dual configuration, that is to say a configuration of dual directed edges. Let us also consider a configuration $\gamma^\ast \in \Omega'$ of dual undirected edges of $\E'$ (the edge set of the dual lattice $\Za$). From the tuple $(\omega,\gamma^\ast)$, let us build a new configuration $\mathcal{R}(\omega,\gamma^\ast) \in \Omega'$ as follows. First recall that $\E_{T+1}\!\setminus\!\mathcal{E}_{T_{k-1}+1}$ is the set (depending on $\omega^\ast$) of directed edges revealed during the exploration of the $k$-th visited cluster $\kedf{x^\ast}{k}$. For any (undirected) edge $\{u^\ast,v^\ast\} \in \E'$, we set
\[
\mathcal{R}(\omega,\gamma^\ast)(\{u^\ast,v^\ast\}) = \gamma^\ast(\{u^\ast,v^\ast\})
\]
when both edges $(u^\ast,v^\ast)$ and $(v^\ast,u^\ast)$ do not belong to $\E_{T+1}\!\setminus\!\mathcal{E}_{T_{k-1}+1}$. Otherwise, only one of them has been revealed by Rule~1, and then we set
\[
\mathcal{R}(\omega,\gamma^\ast)(\{u^\ast,v^\ast\}) = \begin{cases}
\omega^\ast((u^\ast,v^\ast)) \mbox{ if $(u^\ast,v^\ast) \in \E_{T+1}\!\setminus\!\E_{T_{k-1}+1}$}\\
\omega^\ast((v^\ast,u^\ast)) \mbox{ if $(v^\ast,u^\ast) \in \E_{T+1}\!\setminus\!\E_{T_{k-1}+1}$}
\end{cases}.
\]

\begin{figure}[ht!]
\begin{center}
\includegraphics[scale=1]{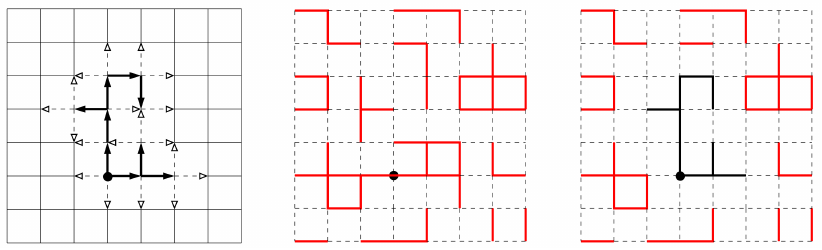}
\caption{{\em Left}: The set of revealed directed edges during the exploration process of $\df{x^\ast}{}$ (associated to the dual configuration $\omega^\ast$). The open and closed revealed edges are respectively represented by bold and dashed arrows. In this basic example, there are no pivotal edges and the first visited cluster equals to the whole forward set of $x^\ast$: $\kedf{x^\ast}{1} = \edf{x^\ast}{} = \df{x^\ast}{}$. {\em Center}: The configuration $\gamma^\ast \in \Omega'$ whose open (undirected) edges are in red. {\em Right}: The new configuration $\mathcal{R}(\omega,\gamma^\ast)$ built from $\omega$ and $\gamma^\ast$. Its (undirected) open edges are in red and black and its closed ones are given with dashed lines. Its black open edges come from $\kedf{x^\ast}{1}$, i.e., from $\omega^\ast$, while the red ones are given by $\gamma^\ast$.}
\label{fig:ConfigR}
\end{center}
\end{figure}

To go from $\mathcal{R}_{0}(\omega,\gamma^\ast) := \gamma^\ast$ to the (final) configuration $\mathcal{R}(\omega,\gamma^\ast)$, we introduce a collection of configurations $\{\mathcal{R}_{m}(\omega,\gamma^\ast)\colon 1\leq m\leq T-T_{k-1}\}$. Precisely, for any $1\leq m\leq T-T_{k-1}$ and for any edge $\{u^\ast,v^\ast\} \in \E'$, we set
\[
\mathcal{R}_m(\omega,\gamma^\ast)(\{u^\ast,v^\ast\}) = \begin{cases}
\gamma^\ast(\{u^\ast,v^\ast\}) \mbox{ if $(u^\ast,v^\ast)$ and $(u^\ast,v^\ast)$ do not belong to $\E_{T_{k-1}+m+1}\!\setminus\!\E_{T_{k-1}+1}$}\\
\omega^\ast((u^\ast,v^\ast)) \mbox{ if $(u^\ast,v^\ast) \in \E_{T_{k-1}+m+1}\!\setminus\!\E_{T_{k-1}+1}$}\\
\omega^\ast((v^\ast,u^\ast)) \mbox{ if $(v^\ast,u^\ast) \in \E_{T_{k-1}+m+1}\!\setminus\!\E_{T_{k-1}+1}$},\\
\end{cases}
\]
where $\E_{T_{k-1}+m+1}\!\setminus\!\E_{T_{k-1}+1}$ is the set of edges revealed during the exploration of $\kedf{x^\ast}{k}$ until step $m$ (included). Let us point out that, as $\mathcal{R}$, the $\mathcal{R}_m$'s are configurations of undirected edges.

To be consistent with the previous notations, we denote by $X_1$ the starting vertex of the edge $Y_{1} := e^\ast_{T_{k-1}+1}$ which is the first revealed edge during the exploration of $\kedf{x^\ast}{k}$. The random variable $X_1$ is known during the exploration of $\kedf{x^\ast}{k}$ and this is why we will work in the sequel conditionally to the event $X_1 = z^\ast$ for some $z^\ast \in \Za$. So we set $\mathbb{P}^{z^\ast}_{(2,\varepsilon)} := \mathbb{P}_{(2,\varepsilon)}(\cdot \mid X_1=z^\ast)$. Finally, let us set respectively $\widetilde{\mathbb{P}}$ and $\widetilde{\mathbb{P}}_{m}$ the push-forward measures of $\mathbb{P}^{z^\ast}_{(2,\varepsilon)}\otimes \mathcal{P}_{q_\varepsilon}$ through the maps $\mathcal{R}$ and $\mathcal{R}_{m}$, with $q_\varepsilon = 1/2-\varepsilon/4$.\\

We split the proof of Proposition~\ref{prop:stochasticDomination} into three steps.\\

\noindent
\textbf{Step 1: Heredity.} The goal in the first step consists in proving that, for any non-negative, local and increasing function $f\colon \Omega' \to \mathbb{R}$ and for any $1\leq m\leq T-T_{k-1}$,
\begin{equation}
\label{Induction}
\int f(\gamma^\ast) \, \widetilde{\mathbb{P}}_{m}(d\gamma^\ast) \, \leq \, \int f(\gamma^\ast)\ \widetilde{\mathbb{P}}_{m-1}(d\gamma^\ast)
\end{equation}
which we will apply to the indicator function $f$ of the event (defined in the previous section)
\[
\A{z^\ast} := \big\{ \gamma^\ast \in \Omega'\colon z^\ast \rightsquigarrow z^\ast + \partial \Lambda_{\sqrt{n}} \; \mbox{without \scalebox{0.7}{\pattern}} \big\}.
\]

Let us start by writing, for a fixed $\gamma^\ast$,
\begin{equation*}
\begin{split}
\int f(\mathcal{R}_{m}(\omega,\gamma^\ast)) \, &\mathbb{P}^{z^\ast}_{(2,\varepsilon)}(d\omega)  =  \int \mathds{1}_{Z_m=0} f(\mathcal{R}_{m}(\omega,\gamma^\ast)) \, \mathbb{P}^{z^\ast}_{(2,\varepsilon)}(d\omega) + \int \mathds{1}_{Z_m=1} f(\mathcal{R}_{m}(\omega,\gamma^\ast)) \, \mathbb{P}^{z^\ast}_{(2,\varepsilon)}(d\omega) \\
& =  \int \mathds{1}_{Z_m=0} f(\mathcal{R}^{Y_m,0}_{m}(\omega,\gamma^\ast)) \, \mathbb{P}^{z^\ast}_{(2,\varepsilon)}(d\omega) + \int \mathds{1}_{Z_m=1} f(\mathcal{R}^{Y_m,1}_{m}(\omega,\gamma^\ast)) \, \mathbb{P}^{z^\ast}_{(2,\varepsilon)}(d\omega),
\end{split}
\end{equation*}
where $\mathcal{R}_m^{Y_m,i}(\omega,\gamma^\ast)(\{u^\ast,v^\ast\})$ equals $i$ if $Y_m \in \{(u^\ast,v^\ast),(v^\ast,u^\ast\}$ and $\mathcal{R}_m(\omega,\gamma^\ast)(\{u^\ast,v^\ast\})$ otherwise. Note that, $\gamma^\ast$ being fixed, $\mathcal{R}_m^{Y_m,i}(\omega,\gamma^\ast)(\{u^\ast,v^\ast\})$ is measurable with respect to $\mathcal{F}_m$ (but not $Z_m$). Hence, conditioning with respect to $\mathcal{F}_m$, we obtain
\begin{eqnarray}
\label{step1-Fm}
\int f(\mathcal{R}_{m}(\omega,\gamma^\ast)) \, \mathbb{P}^{z^\ast}_{(2,\varepsilon)}(d\omega) & = & \int \mathbb{P}^{z^\ast}_{(2,\varepsilon)} \big( Z_m=0 \mid \mathcal{F}_{m} \big) f(\mathcal{R}^{Y_m,0}_{m}(\omega,\gamma^\ast)) \, \mathbb{P}^{z^\ast}_{(2,\varepsilon)}(d\omega) \nonumber \\
& & + \int \mathbb{P}^{z^\ast}_{(2,\varepsilon)} \big( Z_m=1 \mid \mathcal{F}_{m} \big) f(\mathcal{R}^{Y_m,1}_{m}(\omega,\gamma^\ast)) \, \mathbb{P}^{z^\ast}_{(2,\varepsilon)}(d\omega).
\end{eqnarray}
Let us assume for the moment that
\begin{equation}
    \label{Zm<q_eps}
    \mathbb{P}^{z^\ast}_{(2,\varepsilon)} \big( Z_m=1 \mid \mathcal{F}_{m} \big) \leq q_{\varepsilon}.
\end{equation}
Using the monotonicity of $f$ and \eqref{Zm<q_eps}, the following inequality holds
\[
\mathbb{P} \big( Z_m=1 \mid \mathcal{F}_{m} \big) \underbrace{\big( f(\mathcal{R}^{Y_m,1}_{m}(\omega,\gamma^\ast)) - f(\mathcal{R}^{Y_m,0}_{m}(\omega,\gamma^\ast)) \big)}_{ \geq 0} \leq q_\varepsilon \big( f(\mathcal{R}^{Y_m,1}_{m}(\omega,\gamma^\ast)) - f(\mathcal{R}^{Y_m,0}_{m}(\omega,\gamma^\ast)) \big)
\]
and, combining with \eqref{step1-Fm}, this leads to
\begin{align}\label{step1-Major}
    \int f&(\mathcal{R}_{m}(\omega,\gamma^\ast)) \, \mathbb{P}^{z^\ast}_{(2,\varepsilon)}(d\omega) 
    \\\
    \leq 
    &(1-q_\varepsilon) \int f(\mathcal{R}^{Y_m,0}_{m}(\omega,\gamma^\ast)) \, \mathbb{P}^{z^\ast}_{(2,\varepsilon)}(d\omega) \nonumber + q_\varepsilon \int f(\mathcal{R}^{Y_m,1}_{m}(\omega,\gamma^\ast)) \, \mathbb{P}^{z^\ast}_{(2,\varepsilon)}(d\omega).
    \nonumber
\end{align}

On the other hand, the random variables $\mathcal{R}_{m-1}(\omega,\gamma^\ast)$ and $\mathcal{R}^{Y_m,1}_{m}(\omega,\gamma^\ast)$ differ only on the edge $Y_m$, where, with a slight abuse of notation, we also consider $Y_m$ as an undirected edge when we write $\gamma^\ast(Y_m)$. The state of $Y_m$ is given by $\gamma^\ast(Y_m)$ for the configuration $\mathcal{R}_{m-1}(\omega,\gamma^\ast)$ and equals $1$ for $\mathcal{R}^{Y_m,1}_{m}(\omega,\gamma^\ast)$ by definition. In particular, $\mathcal{R}_{m-1}(\omega,\gamma^\ast)$ and $\mathcal{R}^{Y_m,1}_{m}(\omega,\gamma^\ast)$ coincide when $\gamma^\ast(Y_m) = 1$. Moreover, $\mathcal{R}^{Y_m,1}_{m-1}(\omega,\gamma^\ast)$ and $\gamma^\ast(Y_m)$ are independent. We then can write
\begin{equation*}
\begin{split}
\int \mathds{1}_{\gamma^\ast(Y_m)=1} f(\mathcal{R}_{m-1}(\omega,\gamma^\ast)) \, \mathbb{P}^{z^\ast}_{(2,\varepsilon)}(d\omega) \mathcal{P}_{q_\varepsilon}(d\gamma^\ast) & = \int \mathds{1}_{\gamma^\ast(Y_m)=1} f(\mathcal{R}^{Y_m,1}_{m}(\omega,\gamma^\ast)) \, \mathbb{P}^{z^\ast}_{(2,\varepsilon)}(d\omega) \mathcal{P}_{q_\varepsilon}(d\gamma^\ast) \\
& = q_\varepsilon \int f(\mathcal{R}^{Y_m,1}_{m}(\omega,\gamma^\ast)) \, \mathbb{P}^{z^\ast}_{(2,\varepsilon)}(d\omega) \mathcal{P}_{q_\varepsilon}(d\gamma^\ast),
\end{split}
\end{equation*}
and hence,
\begin{eqnarray*}
\int f(\mathcal{R}_{m-1}(\omega,\gamma^\ast)) \, \mathbb{P}^{z^\ast}_{(2,\varepsilon)}(d\omega) \mathcal{P}_{q_\varepsilon}(d\gamma^\ast) & = & \int \mathds{1}_{\gamma^\ast(Y_m)=0} f(\mathcal{R}_{m-1}(\omega,\gamma^\ast)) \, \mathbb{P}^{z^\ast}_{(2,\varepsilon)}(d\omega) \mathcal{P}_{q_\varepsilon}(d\gamma^\ast) \\
& & + \int \mathds{1}_{\gamma^\ast(Y_m)=1} f(\mathcal{R}_{m-1}(\omega,\gamma^\ast)) \, \mathbb{P}^{z^\ast}_{(2,\varepsilon)}(d\omega) \mathcal{P}_{q_\varepsilon}(d\gamma^\ast) \\
& = & (1-q_\varepsilon) \int f(\mathcal{R}^{Y_m,0}_{m}(\omega,\gamma^\ast)) \, \mathbb{P}^{z^\ast}_{(2,\varepsilon)}(d\omega) \mathcal{P}_{q_\varepsilon}(d\gamma^\ast) \\
& & +\,  q_\varepsilon \int f(\mathcal{R}^{Y_m,1}_{m}(\omega,\gamma^\ast)) \, \mathbb{P}^{z^\ast}_{(2,\varepsilon)}(d\omega) \mathcal{P}_{q_\varepsilon}(d\gamma^\ast) \\
& \geq & \int f(\mathcal{R}_{m}(\omega,\gamma^\ast)) \, \mathbb{P}^{z^\ast}_{(2,\varepsilon)}(d\omega) \mathcal{P}_{q_\varepsilon}(d\gamma^\ast),
\end{eqnarray*}
thanks to \eqref{step1-Major}. Inequality \eqref{Induction} is now proved.\\

It then remains to prove~\eqref{Zm<q_eps}. This is the only place where the knowledge of the previous explored edges plays a role. Recall that $Y_m$ is the $m$-th revealed edge during the exploration of $\kedf{x^\ast}{k}$. Without loss of generality we can assume that $Y_m$ is the east (directed) side of a unit square (centered at some primal vertex, say $z$). Let us again denote by $e^\ast_N$, $e^\ast_W$ and $e^\ast_S$ the edges obtained from $Y_m$ by rotation (in the counter-clockwise sense) with center $z$ and angle $\pi/2$, $\pi$ and $3\pi/2$, which are the north, west and south sides of the unit square centered at $z$. By Rule~1, only $e^\ast_W$ and $e^\ast_S$ may have already been explored. In this case, they are necessarily open. Otherwise $Y_m$ would be by definition a pivotal edge and this cannot be true (by construction, each edge visited at step $T_{k-1}+m$, with $1\leq m\leq T-T_{k-1}$, is not pivotal). Hence, the only possibilities are given by
\begin{equation}
\label{eq:casesForDomination}
\begin{cases}
\mathbb{P}_{(2,\varepsilon)} \Big( \dualopen \Big| \dualempty \Big) = \mathbb{P}_{(2,\varepsilon)}\Big(\dualopen\Big) = \frac{1}{2} - \frac{\varepsilon}{4},\\
\\
\mathbb{P}_{(2,\varepsilon)}\Big(\begin{tikzpicture}[baseline=0.6em,scale=0.65]
			\node[inner sep=0, outer sep=0] at (-0.2,0) {\hspace{2pt}};
			\node[inner sep=0, outer sep=0] at (1.2,0) {\hspace{2pt}};
			\draw (0,0) -- (0,1) -- (1,1) -- (1,0) -- (0,0);
			\draw[fill=red, color=red] (0.5,0.5) circle  (0.075);fill=red
			\draw[color=red, very thick ,opacity= 0.4] (0.5,0.5) -- (1,0.5);
			\draw[color=red, very thick ,opacity= 0.4] (0.5,0.5) -- (0.5,0);
			\draw[->,blue,very thick] (1,0) -- (1,1); 
			\path[pattern=north west lines, pattern color=blue] (0.9,0) rectangle (1,1);
			\begin{scope}[rotate around={270:(0.5,0.5)}]
				\draw[->,blue,very thick] (1,0) -- (1,1); 
				\path[pattern=north west lines, pattern color=blue] (0.9,0) rectangle (1,1);
			\end{scope}
		\end{tikzpicture}
  \Big| 
  \begin{tikzpicture}[baseline=0.6em,scale=0.65]
			\node[inner sep=0, outer sep=0] at (-0.2,0) {\hspace{2pt}};
			\node[inner sep=0, outer sep=0] at (1.2,0) {\hspace{2pt}};
			\draw (0,0) -- (0,1) -- (1,1) -- (1,0) -- (0,0);
			\draw[fill=red, color=red] (0.5,0.5) circle  (0.075);fill=red
			\draw[color=red, very thick ,opacity= 0.4] (0.5,0.5) -- (0.5,0);
			\begin{scope}[rotate around={270:(0.5,0.5)}]
				\draw[->,blue,very thick] (1,0) -- (1,1); 
				\path[pattern=north west lines, pattern color=blue] (0.9,0) rectangle (1,1);
			\end{scope}
		\end{tikzpicture}
  \Big) = \mathbb{P}_{(2,\varepsilon)}\Big(\begin{tikzpicture}[baseline=0.6em,scale=0.65]
			\node[inner sep=0, outer sep=0] at (-0.2,0) {\hspace{2pt}};
			\node[inner sep=0, outer sep=0] at (1.2,0) {\hspace{2pt}};
			\draw (0,0) -- (0,1) -- (1,1) -- (1,0) -- (0,0);
			\draw[fill=red, color=red] (0.5,0.5) circle  (0.075);
			\draw[color=red, very thick ,opacity= 0.4] (0.5,0.5) -- (1,0.5);
			\draw[color=red, very thick ,opacity= 0.4] (0.5,0.5) -- (0,0.5);
			\draw[->,blue,very thick] (1,0) -- (1,1); 
			\path[pattern=north west lines, pattern color=blue] (0.9,0) rectangle (1,1);
			\begin{scope}[rotate around={180:(0.5,0.5)}]
				\draw[->,blue,very thick] (1,0) -- (1,1); 
				\path[pattern=north west lines, pattern color=blue] (0.9,0) rectangle (1,1);
			\end{scope}
		\end{tikzpicture}
  \Big| 
  \begin{tikzpicture}[baseline=0.6em,scale=0.65]
			\node[inner sep=0, outer sep=0] at (-0.2,0) {\hspace{2pt}};
			\node[inner sep=0, outer sep=0] at (1.2,0) {\hspace{2pt}};
			\draw (0,0) -- (0,1) -- (1,1) -- (1,0) -- (0,0);
			\draw[fill=red, color=red] (0.5,0.5) circle  (0.075);
			\draw[color=red, very thick ,opacity= 0.4] (0.5,0.5) -- (0,0.5);
			\begin{scope}[rotate around={180:(0.5,0.5)}]
				\draw[->,blue,very thick] (1,0) -- (1,1); 
				\path[pattern=north west lines, pattern color=blue] (0.9,0) rectangle (1,1);
			\end{scope}
		\end{tikzpicture}
  \Big) = \frac{(1 - \varepsilon)/6}{1/2 - \varepsilon/4} = \frac{1}{3} \frac{1 - \varepsilon}{1 - \varepsilon/2} \leq \frac{1}{3}, \\ 
\\
\mathbb{P}_{(2,\varepsilon)}\Big(\dualopenthree \Big| \dualopentwo\Big) = 0.
\end{cases}
\end{equation}
All these probabilities are smaller than $q_\varepsilon=1/2-\varepsilon/4$ for $\varepsilon\geq 0$ small enough. Actually, the knowledge of the past only acts through open already explored edges which decrease the probability for $Y_m$ to be open, and then allow a stochastic domination with a Bernoulli random variable with parameter $q_\varepsilon$. So for all $A \in \mathcal{F}_{m}$, $\mathbb{P}_{(2,\varepsilon)}( Z_m=1 \mid A ) \leq q_\varepsilon$ and thus almost surely
\begin{equation}
\label{eq:domnination}
\mathbb{P}_{(2,\varepsilon)} \big( Z_{m}=1 \mid \mathcal{F}_{m} \big) \leq q_\varepsilon.
\end{equation}
To conclude, it is enough to say that the event $\{X_1 = z^\ast\}$ belongs to $\mathcal{F}_{m}$ for any $m \geq 1$. Henceforth, for $i \in \{0,1\}$,
\[
\mathbb{P}^{z^\ast}_{(2,\varepsilon)} \big( Z_{m}=i \mid \mathcal{F}_m \big) = \mathbb{P}_{(2,\varepsilon)} \big( Z_{m}=i \mid \mathcal{F}_m \big),
\]
$\mathbb{P}^{z^\ast}_{(2,\varepsilon)}$-almost surely. This proves inequality \eqref{Zm<q_eps}.\\

\noindent
\textbf{Step 2: Domination of $\widetilde{\mathbb{P}}$ by $\mathcal{P}_{q_\varepsilon}$.} We first use the heredity property stated in Step~1 to get, for any $1\leq m\leq T-T_{k-1}$,
\begin{equation}\label{Step1-heredity}
\int f(\gamma^\ast) \, \widetilde{\mathbb{P}}_{m}(d\gamma^\ast) \leq \int f(\gamma^\ast) \, \widetilde{\mathbb{P}}_{0}(d\gamma^\ast) = \int f(\gamma^\ast) \, \mathcal{P}_{q_\varepsilon}(d\gamma^\ast)
\end{equation}
since by construction $\mathcal{R}_{0}(\omega,\gamma^\ast) = \gamma^\ast$, i.e., $\widetilde{\mathbb{P}}_{0} = \mathcal{P}_{q_\varepsilon}$.

Let us first assume that $T = T(\omega)$ is finite, for some configuration $\omega$. That is, the $k$-th visited cluster $\kedf{x^\ast}{k}$ contains a finite number of edges, precisely $T-T_{k-1}$. In this case, the final configuration $\mathcal{R}(\omega,\gamma^\ast)$ equals to $\mathcal{R}_{T-T_{k-1}}(\omega,\gamma^\ast)$. Thus, taking $m = T-T_{k-1}$ in the previous inequality, gives the desired result
\begin{equation}
\label{Step2-Goal}
\int f(\gamma^\ast) \, \widetilde{\mathbb{P}}(d\gamma^\ast) \leq \int f(\gamma^\ast) \, \mathcal{P}_{q_\varepsilon}(d\gamma^\ast).
\end{equation}
It then remains to prove that~\eqref{Step2-Goal} still holds when $T=\infty$. Passing to the limit $m \to \infty$ is required. Recall that the function $f\colon \Omega' \to \mathbb{R}$ is local since it only depends on edges of $\E'$ included in some bounded box $z^\ast+\Lambda$. The visited cluster $\kedf{x^\ast}{k}$ being infinite (since $T=\infty$), its exploration only concerns edges outside $z^\ast+\Lambda$ from a certain integer $m_0 = m_0(\omega)$. Hence, for any $m \geq m_0$, the random variables $\mathcal{R}_m(\omega,\gamma^\ast)$ and $\mathcal{R}_{m_0}(\omega,\gamma^\ast)$ coincide on edges including in the box $z^\ast+\Lambda$. So, the same holds for $\mathcal{R}(\omega,\gamma^\ast)$ and $\mathcal{R}_{m_0}(\omega,\gamma^\ast)$ and in particular
\[
f\big(\mathcal{R}(\omega,\gamma^\ast)\big) = f\big(\mathcal{R}_{m_0}(\omega,\gamma^\ast)\big).
\]
This means that the sequence $(f(\mathcal{R}_{m}(\omega,\gamma^\ast)))_m$ converges to $f(\mathcal{R}(\omega,\gamma^\ast))$ for almost every $(\omega,\gamma^\ast)$. Since $f$ is non-negative, the Fatou's lemma applies and we get
\begin{equation*}
\begin{split}
\int f(\gamma^\ast) \, \widetilde{\mathbb{P}}(d\gamma^\ast) = \int f(\mathcal{R}(\omega,\gamma^\ast)) \, \mathbb{P}^{z^\ast}_{(2,\varepsilon)}(d\omega) \mathcal{P}_{q_\varepsilon}(d\gamma^\ast) & \leq  \liminf_{m \to \infty} \int f(\mathcal{R}_{m}(\omega,\gamma^\ast)) \, \mathbb{P}^{z^\ast}_{(2,\varepsilon)}(d\omega) \mathcal{P}_{q_\varepsilon}(d\gamma^\ast) \\
& =  \liminf_{m \to \infty} \int f(\gamma^\ast) \, \widetilde{\mathbb{P}}_{m}(d\gamma^\ast) \\
& \leq  \int f(\gamma^\ast) \, \mathcal{P}_{q_\varepsilon}(d\gamma^\ast),
\end{split}
\end{equation*}
where the last inequality is due to \eqref{Step1-heredity}.

\medskip

\noindent
\textbf{Step 3: Conclusion.} Recall the following event defined in the previous section:
\[
\A{z^\ast} := \big\{ \gamma^\ast \in \Omega' : \, z^\ast \stackrel{\gamma^\ast}{\rightsquigarrow} z^\ast + \partial \Lambda_{n} \; \mbox{without \scalebox{0.7}{\pattern}} \big\}.
\]
Besides, we work conditioned on $X_1 = z^\ast$ and we consider the new event
\[
\big\{ \omega \in \Omega\colon  z^{\ast} \stackrel{\omega^\ast}{\rightsquigarrow} z^{\ast}+\partial \Lambda_{n} \mbox{ in $\kedf{x^\ast}{k}$} \big\}
\]
meaning that the exploration of the $k$-th visited cluster $\kedf{x^\ast}{k}$ starting at $z^\ast$ reaches the boundary $z^{\ast}+\partial \Lambda_{n}$. Let us emphasize that $\A{z^\ast}$ depends on (undirected) configurations of $\Omega'$ while the event stated just above depends on (directed) configurations of $\Omega$. This is why we add the upperscripts $\gamma^\ast$ and $\omega^\ast$ over $\rightsquigarrow$. For the moment let us assume that the following inclusion holds for any $\gamma^\ast \in \Omega'$:
\begin{equation}
\label{Inclusion}
\big\{ \omega \in \Omega \colon z^{\ast} \stackrel{\omega^\ast}{\rightsquigarrow} z^{\ast}+\partial \Lambda_{n} \mbox{ in $\kedf{x^\ast}{k}$} \big\} \subset \big\{ \omega \in \Omega\colon \mathcal{R}(\omega,\gamma^\ast) \in \A{z^{\ast}} \big\}.
\end{equation}
It is not difficult to conclude from \eqref{Inclusion}. We first write
\[
\PP^{z^\ast}_{(2,\varepsilon)} \big( z^{\ast} \stackrel{\omega^\ast}{\rightsquigarrow} z^{\ast}+\partial \Lambda_{n} \mbox{ in $\kedf{x^\ast}{k}$} \big) \leq \int \mathds{1}_{ \mathcal{R}(\omega,\gamma^\ast) \in  \A{z^{\ast}}} \, \PP^{z^\ast}_{(2,\varepsilon)}(d\omega)
\]
for any $\gamma^\ast \in \Omega'$. Thus, integrating with respect to $\mathcal{P}_{q_\varepsilon}$ and using Step~2 with $f=\mathds{1}_{\A{z^{\ast}}}$ (which is a local, non-negative and increasing function), we obtain:
\begin{eqnarray*}
\PP^{z^\ast}_{(2,\varepsilon)} \big( z^{\ast} \stackrel{\omega^\ast}{\rightsquigarrow} z^{\ast}+\partial \Lambda_{n} \mbox{ in $\kedf{x^\ast}{k}$} \big) & \leq & \int \mathds{1}_{ \mathcal{R}(\omega,\gamma^\ast) \in  \A{z^{\ast}}} \, \PP^{z^\ast}_{(2,\varepsilon)}(d\omega) \mathcal{P}_{q_\varepsilon}(d\gamma^\ast) \\
& = & \widetilde{\mathbb{P}}(\A{z^{\ast}}) \leq \mathcal{P}_{q_\varepsilon}(\A{z^{\ast}}) =\mathcal{P}_{q_\varepsilon}(\A{0^{\ast}}),
\end{eqnarray*}
where the last equality is due to translation-invariance in the independent bond percolation model on $\Za$. Finally,
\begin{equation*}
\begin{split}
\PP_{(2,\varepsilon)} \big( &| \kedf{x^\ast}{k} | \geq n \big)  \leq  \PP_{(2,\varepsilon)} \big( X_1 \stackrel{\omega^\ast}{\rightsquigarrow} X_1 + \partial \Lambda_{n} \mbox{ in $\kedf{x^\ast}{k}$} \big) \\
& =  \sum_{z^{\ast} \in \Za} \PP^{z^\ast}_{(2,\varepsilon)} \big( z^\ast \stackrel{\omega^\ast}{\rightsquigarrow} z^\ast + \partial \Lambda_{n} \mbox{ in $\kedf{x^\ast}{k}$} \big) \PP_{(2,\varepsilon)} (X_1 = z^\ast) \mathds{1}_{\PP_{(2,\varepsilon)} (X_1 = z^\ast) > 0} \\
& \leq \mathcal{P}_{q_\varepsilon}(\A{0^{\ast}}) \sum_{z^{\ast} \in \Za} \PP_{(2,\varepsilon)} (X_1 = z^\ast) \mathds{1}_{\PP_{(2,\varepsilon)} (X_1 = z^\ast) > 0} \\
& =  \mathcal{P}_{q_\varepsilon}(\A{0^{\ast}}),
\end{split}
\end{equation*}
as desired.\\

The last statement to prove is the inclusion~\eqref{Inclusion}. To do it, let us consider a configuration $\omega \in \Omega$ satisfying
\[
\big\{ z^{\ast} \stackrel{\omega^\ast}{\rightsquigarrow} z^{\ast}+\partial \Lambda_{n} \mbox{ in $\kedf{x^\ast}{k}$} \big\}
\]
and a configuration $\gamma^\ast$ in $\Omega'$. Hence, the exploration of the $k$-th visited cluster $\kedf{x^\ast}{k}(\omega)$ contains a path $\pi^\ast$ of directed edges from $z^{\ast}$ to $z^{\ast}+\partial \Lambda_{n}$. By definition of the exploration process, this path $\pi^\ast$ cannot contain the pattern \scalebox{0.3}{\begin{tikzpicture}
    \draw[step=1cm, gray,very thin] (-1,0) grid (2,1);
    \draw[blue, line width=1mm, -latex] (0,1) -- (0,0);
    \draw[blue, line width=1mm, -latex] (0,0) -- (1,0);
    \draw[blue, line width=1mm, -latex] (1,0) -- (1,1);
\end{tikzpicture}}, called a {\em left winding}, (nor its rotated variants) because this pattern contains three dual directed edges which all refer to the same primal vertex say $w$ meaning that such $w$ admits at most one (primal) outgoing edge, which is impossible in our 'rigid' $(2,\varepsilon)$-model. However the path $\pi^\ast$ may contain the pattern \scalebox{0.3}{\begin{tikzpicture}
    \draw[step=1cm, gray,very thin] (-1,0) grid (2,1);
    \draw[blue, line width=1mm, -latex] (0,0) -- (0,1);
    \draw[blue, line width=1mm, -latex] (1,0) -- (0,0);
    \draw[blue, line width=1mm, -latex] (1,1) -- (1,0);
\end{tikzpicture}}, called a {\em right winding}, (or its rotated variants) because the three edges that it contains refer to three different (primal) vertices. Now, let us denote by $\widetilde{\pi}^\ast$ the path $\pi^\ast$ in which orientation of edges are neglecting: $\widetilde{\pi}^\ast$ is a path of undirected edges joining $z^{\ast}$ to $z^{\ast}+\partial \Lambda_{n}$. Since, by construction, the configuration $\mathcal{R}(\omega,\gamma)$ contains all the open edges revealed during the exploration of $\kedf{x^\ast}{k}(\omega)$ (but without orientation), it also contains the path $\widetilde{\pi}^\ast$. Thanks to the previous discussion, the path $\widetilde{\pi}^\ast$ may contain the pattern \scalebox{0.3}{\begin{tikzpicture}
    \draw[step=1cm, gray,very thin] (-1,0) grid (2,1);
    \draw[blue, line width=1mm] (0,0) -- (0,1);
    \draw[blue, line width=1mm] (1,0) -- (0,0);
    \draw[blue, line width=1mm] (1,1) -- (1,0);
\end{tikzpicture}} (or its rotated variants) as being the undirected version of a right winding. However, $\widetilde{\pi}^\ast$ cannot contain \scalebox{0.7}{\pattern}. Indeed, any of the two directed versions of that pattern  necessarily contains a left winding which is forbidden for $\pi^\ast$.

In conclusion, the configuration $\mathcal{R}(\omega,\gamma^\ast)$ contains a path from $z^{\ast}$ to $z^{\ast}+\partial \Lambda_{n}$ which does not use \scalebox{0.7}{\pattern}, i.e., $\mathcal{R}(\omega,\gamma^\ast) \in \A{z^\ast}$ and the inclusion \eqref{Inclusion} is proved.

\section{Enhancement}
\label{sect:enhancement}

This section concerns the standard (undirected) Bernoulli bond percolation model and can be read independently of the rest of the paper. The main result is Theorem~\ref{prop:enhancement} which roughly says that one can trade the occurrence of the forbidden patterns for a slightly smaller parameter (for an edge to be open).

\begin{theorem}
\label{prop:enhancement}
There exists $\varepsilon'>0$ such that for all integer $n$ large enough, we have
\begin{equation}
\label{prop:Goal-Enhancement}
\mathcal{P}_{1/2}(o \to \partial \Lambda_n  \text{ without } \scalebox{0.7}{\pattern}) \leq \mathcal{P}_{1/2-\varepsilon'} (o \to \partial\Lambda_n ).
\end{equation}
\end{theorem}

To get such a result, the strategy consists in using an enhancement technique. 
For this, we introduce a new probabilistic model with two parameters $p$ and $q$, enhancing a standard, independent bond percolation model with parameter $p$ on $\Z^2$. 
The basic idea is that any forbidden pattern can be bypassed by using an extra diagonal edge which is open with probability $q$. The set of extra diagonal edges will be denoted by $\mathcal{D}$ and is defined below, see Figure~\ref{fig:addEdges} for an illustration. Thus, working on partial derivatives with respect to the parameters $p$ and $q$, we aim to compensate a small decrease of $p$ by a small increase of $q$, leading to~\eqref{prop:Goal-Enhancement}. The strategy is similar to the classical works on enhancement~\cite{aizenman1991strict} and~\cite{balister2014essential}, but note that we cannot simply apply their results because the enhancement we consider is {\em not essential}. This is mainly due to the fact that the constraint we impose only forbids the use of the whole pattern, while individual parts of the pattern may still appear in paths that connect the origin to the boundary of large boxes. This will become more clear in the following sections. 

\medskip 
But let us first show how one can use Theorem~\ref{prop:enhancement} to prove Theorem~\ref{thm:clusterSize} in the remaining case where $\varepsilon = 0$.

\begin{proof}[Proof of Theorem~\ref{thm:clusterSize} when $\varepsilon = 0$] 
It suffices to combine both Proposition~\ref{prop:stochasticDomination} and Theorem~\ref{prop:enhancement}:
\begin{eqnarray*}
\PP_{(2,0)} \big( | \kedf{x^\ast}{k} | \geq n \big) & \leq &  \mathcal{P}_{1/2} \big( o^\ast \rightsquigarrow o^\ast + \partial \Lambda_{n} \; \mbox{without \scalebox{0.7}{\pattern}} \big) \\
& \leq &  \mathcal{P}_{1/2-\varepsilon'} \big( o \rightsquigarrow o + \partial \Lambda_{n} \big)
\end{eqnarray*}
for any $n$ large enough, where $\varepsilon' > 0$ is given by Theorem~\ref{prop:enhancement}. We conclude using that the parameter $1/2-\varepsilon'$ corresponds to the subcritical regime for the independent bond percolation model in $\Z^2$ meaning that $\mathcal{P}_{1/2-\varepsilon'}( o \rightsquigarrow o + \partial \Lambda_{n})$ decreases exponentially fast to $0$ as $n$ tends to infinity.
\end{proof}

\subsection{The enhanced model}
Let $\mathcal{E}_0$ denote the set of (undirected) edges of $\mathbb{Z}^2$. Let us now consider the set of additional (undirected) diagonal edges
\[
\mathcal{D} := \bigcup_{(a,b)\in \mathbb{Z}^{2}} \Big\{ \{(a,b),(a-1,b+2)\},\{(a,b),(a+2,b+1)\} \Big\}.
\]
The new set of edges is the union $\mathcal{E}_0 \cup \mathcal{D}$ and so the enhanced configuration set is
\[
\Omega^+ := \{0,1\}^{\mathcal{E}_0 \cup \mathcal{D}}.
\]
To distinguish edges from $\mathcal{E}_0$ and $\mathcal{D}$, the edges of $\mathcal{E}_0$ are called $p$-edges while those of $\mathcal{D}$ are called $q$-edges or diagonal edges. Let us first remark that each occurrence of a forbidden pattern (created by $p$-edges) can be bypassed by exactly one diagonal edge; this is the role of the elements of $\mathcal{D}$. We will say that a forbidden pattern and its diagonal edge are {\em associated}.

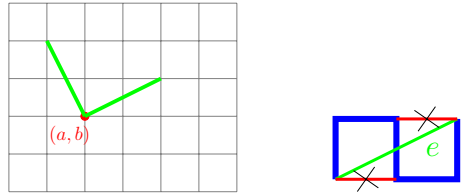
\begin{figure}[!ht]
\begin{center}
\scalebox{0.5}{
\begin{tikzpicture}
\draw[step=1cm,gray,thin] (1,3) grid (7,8);    
\filldraw[red] (3,5) circle (3pt);
\node at (2.6,4.5){\textcolor{red}{\Large $(a,b)$}};
\draw[green,line width=1mm] (3,5) -- (2,7);
\draw[green,line width=1mm] (3,5) -- (5,6);
\end{tikzpicture}} \hspace*{1cm}
\scalebox{0.8}{\begin{tikzpicture}
    \draw[step=1cm,gray,very thin] (0,0) grid (2,1);
    \draw[blue,line width=1mm] (0,0) -- (0,1) -- (1,1) -- (1,0) -- (2,0) -- (2,1);
    \draw[green,line width=0.5mm] (0,0) -- (2,1);
    \draw[red,line width=0.5mm] (0,0) -- (1,0);
    \draw[red,line width=0.5mm] (1,1) -- (2,1);
    \node at (0.5,0) [color=red,cross,rotate=10,inner sep=0pt, minimum size=10pt] {};
    \node at (1.5,1) [color=red,cross,rotate=10,inner sep=0pt, minimum size=10pt] {};
    \node at (1.6,0.5) [color=green] {\Large $e$};
\end{tikzpicture}}
\caption{\label{fig:addEdges}Left: Two diagonal edges (or $q$-edges) incident to the vertex $(a,b)$. Right: A forbidden pattern and its associated diagonal edge (in green).}
\end{center}
\end{figure}

Let us now define the family of probability measures on $\Omega^+$ that we want to study. There are two layers of randomness. First, edges of $\mathcal{E}_0$ are open or closed according to a standard independent bond percolation model with parameter $p$. Thus, conditionally to the states of the edges of $\mathcal{E}_0$, consider the subset $\mathcal{D}' \subset \mathcal{D}$ of diagonal edges whose associated forbidden pattern occurs: edges of $\mathcal{D}'$ are open with probability $q$, independently from each other. In other words, each diagonal edge is open with probability $q$ provided its associated forbidden pattern occurs. Otherwise, i.e., with probability $1-q$ or if its associated forbidden pattern does not occur, it is closed. We will denote by $\mathbb{P}_{p,q}^+$ the probability measure on $\Omega^+$ defined by this process. Note that the diagonal edges can only appear, if the associated forbidden pattern occurs, so the enhancement is {\em not essential} in the sense of~\cite[Chapter~3.3.]{grimmett1999percolation}. 
Let us begin with the following result.

\begin{lemma}
\label{lem:enhancedEqBer}
For any $p,q \in [0,1]$, the following holds.
\begin{itemize}
\item[(i)] $\mathbb{P}_{p,0}^+ \big( o \to \partial \Lambda_n  \text{ without } \scalebox{0.7}{\pattern} \big) = \mathcal{P}_{p} \big( o \to \partial \Lambda_n \text{ without } \scalebox{0.7}{\pattern} \big)$
\item[(ii)] $\mathbb{P}_{p,q}^+ \big( o \to \partial \Lambda_n  \text{ without } \scalebox{0.7}{\pattern} \big) \leq \mathcal{P}_{p} \big( o \to \partial \Lambda_n  \big)$
\end{itemize}
\end{lemma}

\begin{proof}
When the parameter $q$ of diagonal edges is null, the probability distribution $\mathbb{P}_{p,q}^+$ is nothing but the standard independent bond percolation model with parameter $p$, i.e., the probability distribution $\mathcal{P}_{p}$. This gives Item (i).

Let us consider $\Pi\colon \Omega^+ \mapsto \{0,1\}^{\mathcal{E}_0}$ the canonical projection of $\Omega^+$ onto $\{0,1\}^{\mathcal{E}_0}$. Since under $\mathbb{P}_{p,q}^+$ each edge of $\mathcal{E}_0$ is open independently with probability $p$, we have
\[
\mathcal{P}_{p} = \mathbb{P}_{p,q}^+\circ \Pi^{-1},
\]
for any $q \in [0,1]$. Besides, let $\omega^+ \in \Omega^+$ satisfying $\{o \to \partial \Lambda_n  \text{ without } \scalebox{0.7}{\pattern}\}$ and consider $\omega := \Pi(\omega^+) \in \{0,1\}^{\mathcal{E}_0}$. Whether $\omega^+$ uses diagonal edges or not, its projection $\omega := \Pi(\omega^+)$ belongs to $\{o \to \partial \Lambda_n\}$. In other words,
\[
\big\{\omega^+ \colon o \to \partial \Lambda_n  \text{ without } \scalebox{0.7}{\pattern} \big\} \subset \Pi^{-1} \big( \{ w \colon  o \to \partial \Lambda_n\} \big)
\]
which leads to Item (ii),
\begin{align*}
\mathbb{P}_{p,q}^+ \big( o \to \partial \Lambda_n  \text{ without } \scalebox{0.7}{\pattern} \big) 
\leq 
\mathbb{P}_{p,q}^+\circ \Pi^{-1} \big( o \to \partial \Lambda_n \big) = \mathcal{P}_{p} \big( o \to \partial \Lambda_n \big),
\end{align*}
as desired.
\end{proof}

In the sequel, it will be convenient to use the shorter notation
\[
\Theta_n(p,q) := \mathbb{P}_{p,q}^+ \big( o \to \partial \Lambda_n  \text{ without } \scalebox{0.7}{\pattern} \big).
\]

The main ingredient for the enhancement technique is an estimate that allows to control the partial derivative $\partial_p \Theta_n(p,q)$ in terms of the partial derivative with respect to the other parameter. This will be done in the following sections through several steps.

\begin{prop}
\label{prop:derivativeComp}
The partial derivatives of $\Theta_n$ w.r.t.~$p$ and $q$ exist and are non-negative for any $n\geq 1$. In addition, there exists a constant $C > 0$ such that $\forall n \geq 1$, $\forall q \in [0,1]$ and $\forall p \in [10^{-5},1-10^{-5}]$, we have
\[
\partial_p \Theta_n(p,q) \leq C \partial_q \Theta_n(p,q).
\]
\end{prop}

Theorem~\ref{prop:enhancement} is a direct consequence of Lemma~\ref{lem:enhancedEqBer} and Proposition~\ref{prop:derivativeComp}.

\begin{proof}[Proof of Theorem~\ref{prop:enhancement}]
According to the finite increment formula, we can assert that, for any $\varepsilon' \in [0,1/4]$ and any $q_0 \in [0,1]$, there exists $\xi = \xi(\varepsilon',q_0) \in [1/4,1/2] \times [0,1]$ such that
\begin{equation}
\label{eq:enchanceproof1}
\Theta_n(1/2,0) - \Theta_n(1/2-\varepsilon',q_0) = \nabla \Theta_n(\xi)\cdot(\varepsilon',-q_0).
\end{equation}
Then Proposition~\ref{prop:derivativeComp} provides the upper bound
\[
\nabla \Theta_n(\xi)\cdot(\varepsilon',-q_0) = \varepsilon' \partial_p \Theta_n(\xi) - q_0 \partial_q \Theta_n(\xi) \leq (C \varepsilon' - q_0) \partial_q \Theta_n(\xi),
\]
which can be made non-positive by choosing $\varepsilon'>0$ sufficiently small and using that $\partial_q \Theta_n$ is non-negative by Proposition~\ref{prop:derivativeComp}. We finally get  
\begin{equation}
\label{Ineg-Enhancement}
\Theta_n(1/2,0) \leq \Theta_n(1/2-\varepsilon',q_0).
\end{equation}
We conclude by combining~\eqref{Ineg-Enhancement} and Lemma~\ref{lem:enhancedEqBer}, Items~(i) and (ii): 
\begin{equation*}
\mathcal{P}_{1/2} \big( o \to \partial \Lambda_n  \text{ without } \scalebox{0.7}{\pattern} \big)=  \Theta_n(1/2,0)\leq \Theta_n(1/2-\varepsilon',q_0)\leq \mathcal{P}_{1/2-\varepsilon'} \big( o \to \partial \Lambda_n \big),
\end{equation*}
where $\varepsilon'>0$ has been chosen small enough.
\end{proof}

\subsection{Partial derivatives and Russo's formula}

The aim of this section is to show that for fixed $n \in \N$ the partial derivatives of $\Theta_n$ with respect to the parameters $p$ and $q$ exist and then to expand them using Russo's formula (Lemma~\ref{lem:Russo}). As a first step, we state that $\Theta_n$ is a monotone function with respect to each parameter $p$ and $q$, see Lemma~\ref{lem:monotonicity}.

Let us consider a family $\{U_e\colon e \in \mathcal{E}_0 \cup \mathcal{D}\}$ of independent random variables uniformly distributed on $[0,1]$. Let us denote by $\mathbb{P}^{+}$ the product measure that the $U_e$'s generate on the measurable space $[0,1]^{\mathcal{E}_0 \cup \mathcal{D}}$ (equipped with the product $\sigma$-algebra). Let us define the random subgraph $\mathcal{G}_{p,q}$ of $(\mathbb{Z}^2,\mathcal{E}_0 \cup \mathcal{D})$ as follows:
\begin{itemize}
\item Any $e$ in $\mathcal{E}_0$ is open in $\mathcal{G}_{p,q}$ if and only if $U_e \leq p$;
\item Any $e$ in $\mathcal{D}$ is open in $\mathcal{G}_{p,q}$ if and only if its associated forbidden pattern occurs (see the right side of Figure~\ref{fig:addEdges}) and $U_e \leq q$.
\end{itemize}
Under $\mathbb{P}^{+}$, the random graph $\mathcal{G}_{p,q}$ is distributed according to $\mathbb{P}^{+}_{p,q}$. This construction provides a coupling between several versions of the probability measure $\mathbb{P}^{+}_{p,q}$ which will be useful for proving Lemma~\ref{lem:Russo}.

Let us set
\[
\mathcal{A}_{n,p,q} := \big\{ o \to \partial\Lambda_n \; \mbox{in $\mathcal{G}_{p,q}$ without \scalebox{0.7}{\pattern}} \big\},
\]
which allows us to write the probability $\Theta_n(p,q)$ as
\[
\Theta_n(p,q) = \mathbb{P}^{+} \big( \mathcal{A}_{n,p,q} \big).
\]
To lighten notations, we will from now on omit the subscript $n$ in $\mathcal{A}_{n,p,q}$ and just write $\mathcal{A}_{p,q}$. The monotonicity of $\Theta_n$ with respect to each of the parameters $p$ and $q$ is an immediate consequence of the next result.

\begin{lemma}
\label{lem:monotonicity}
Let $p,p',q,q'\in [0,1]$ such that $p\leq p'$ and $q\leq q'$, then, 
\[
\mathcal{A}_{p,q} \subset \mathcal{A}_{p',q} \, \text{ and } \, \mathcal{A}_{p,q} \subset \mathcal{A}_{p,q'}.
\]
\end{lemma}

\begin{proof}
Let us consider a configuration $\omega \in \mathcal{A}_{p,q}$, i.e., there exists an open path $\gamma$ in $\mathcal{G}_{p,q}(\omega)$ from $o$ to $\partial \Lambda_n$ that does not use the forbidden patterns. Since increasing $q$ only has the effect of opening new $q$-edges (whose associated forbidden patterns occur), the random graph $\mathcal{G}_{p,q'}(\omega)$ contains $\mathcal{G}_{p,q}(\omega)$ and also the path $\gamma$. So $\omega \in \mathcal{A}_{p,q'}$ and the second inclusion of Lemma~\ref{lem:monotonicity} is proven. 
However, for increasing $p$ to $p'$ one needs to be more careful, since the opening of new $p$-edges may (create or) destroy some forbidden patterns and then may delete the associated $q$-edges that could be used by $\gamma$. Let $\mathcal{D}'$ be the set of $q$-edges which are open in $\mathcal{G}_{p,q}(\omega)$ but closed in $\mathcal{G}_{p',q}(\omega)$ because their associated forbidden patterns have been altered by the adding of new $p$-edges: see Figure~\ref{fig:diff} for an illustration. Either $\gamma$ does not use any element of $\mathcal{D}'$ and in this case, any open edge of $\gamma$ (in $\mathcal{G}_{p,q}(\omega)$) is still open in $\mathcal{G}_{p',q}(\omega)$. Or $\gamma$ uses such a diagonal edge, say $e = \{a,b\}$. This diagonal edge is now closed in $\mathcal{G}_{p',q}(\omega)$ but, in the same time, an open path of $p$-edges is appeared, included in the rectangle delimited by vertices $a$ and $b$, and joining $a$ and $b$ without using a forbidden pattern. This means that around each element of $\mathcal{D}'$ used by $\gamma$, local modifications can be performed so as to change $\gamma$ into a new open path, included in $\mathcal{G}_{p',q}(\omega)$ and linking $o$ to $\partial \Lambda_n$ without using a forbidden pattern. In both cases, the configuration $\omega$ still belongs to $\mathcal{A}_{p',q}$. This proves the first inclusion of Lemma~\ref{lem:monotonicity}.
\end{proof}

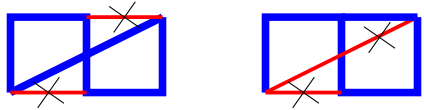
\begin{figure}[ht!]
\begin{center}
\begin{tikzpicture}
    \draw[step=1cm,gray,very thin] (0,0) grid (2,1);
    \draw[blue,line width=1mm] (0,0) -- (0,1) -- (1,1) -- (1,0) -- (2,0) -- (2,1);
    \draw[blue,line width=1mm] (0,0) -- (2,1);
    \draw[red,line width=0.5mm] (0,0) -- (1,0);
    \draw[red,line width=0.5mm] (1,1) -- (2,1);
    \node at (0.5,0) [color=red,cross,rotate=10,inner sep=0pt, minimum size=10pt] {};
    \node at (1.5,1) [color=red,cross,rotate=10,inner sep=0pt, minimum size=10pt] {};
\end{tikzpicture} \hspace*{1cm}
\begin{tikzpicture}
    \draw[step=1cm,gray,very thin] (0,0) grid (2,1);
    \draw[blue,line width=1mm] (0,0) -- (0,1) -- (1,1) -- (1,0) -- (2,0) -- (2,1);
    \draw[red,line width=0.5mm] (0,0) -- (2,1);
    \draw[red,line width=0.5mm] (0,0) -- (1,0);
    \draw[blue,line width=1mm] (1,1) -- (2,1);
    \node at (0.5,0) [color=red,cross,rotate=10,inner sep=0pt, minimum size=10pt] {};
    \node at (1.5,0.73) [color=red,cross,rotate=10,inner sep=0pt, minimum size=10pt] {};
\end{tikzpicture}
\caption{\label{fig:diff}The opening of a closed edge in the left pattern makes useless the corresponding diagonal edge. However, a new admissible open path of $p$-edges (i.e. without using a forbidden pattern) appears in the right pattern.}
\end{center}
\end{figure}

We are now ready to show Russo-type formula for partial derivatives of $\Theta_n$. To do so, let us first introduce some more notations. For any edge $e \in \mathcal{E}_0 \cup \mathcal{D}$, the random graphs $\mathcal{G}_{p,q}^e$ and $\mathcal{G}_{p,q}^{\neg e}$ are defined as $\mathcal{G}_{p,q}$ but with $U_e=0$ and $U_e=1$ respectively. Notice that a $p$-edge $e$ is always open in $\mathcal{G}_{p,q}^e$ (and always closed in $\mathcal{G}_{p,q}^{\neg e}$ resp.) whereas a diagonal edge $e$ is open in $\mathcal{G}_{p,q}^e$ only if its associated forbidden pattern occurs. Let us then define the event
\[
\mathcal{A}^{\square}_{p,q} := \big\{ o \to \partial\Lambda_n \text{ in $\mathcal{G}_{p,q}^\square$ without \scalebox{0.7}{\pattern}} \big\},
\]
where $\square \in \{e,\neg e\}$. An edge $e \in \mathcal{E}_0 \cup \mathcal{D}$ is called {\em pivotal} for $\mathcal{A}_{p,q}$ when the event $\mathcal{A}^{e}_{p,q}\!\setminus\!\mathcal{A}^{\neg e}_{p,q}$ occurs. As usual, the fact that $e$ is pivotal does not depend on its own state, i.e., on the random variable $U_e$.

Recall that $\Lambda_{n} = \{ z \in \Z^2 \colon \|z\|_{\ell^\infty} \leq n \}$ and $\partial \Lambda_{n} = \{ z \in \Z^2 \colon \|z\|_{\ell^\infty} = n \}$. Let us denote by $\mathcal{E}^{\tiny{\mbox{in}}}_n$ the set of $p$-edges having {\em at least one} endpoint in $\Lambda_n\!\setminus\!\partial\Lambda_n$ and $\mathcal{D}^{\tiny{\mbox{in}}}_n$ the set of diagonal edges having {\em both} endpoints in $\Lambda_n\!\setminus\!\partial\Lambda_n$.

\begin{lemma}
\label{lem:Russo}
The partial derivatives of $\Theta_n$ w.r.t.~$p$ and $q$ exist and are non-negative. In addition,
\begin{equation}
\label{eq:formuleRusso1}
\begin{split}
\partial_p\Theta_n(p,q) &= \sum_{e \in \mathcal{E}^{\tiny{\mbox{in}}}_n} \mathbb{P}^{+}\!\big( \text{$e$ is pivotal for $\mathcal{A}_{p,q}$}\big)\quad\text{ and}\\
%\label{eq:formuleRusso2}
\partial_q\Theta_n(p,q) &= \sum_{e \in \mathcal{D}^{\tiny{\mbox{in}}}_n} \mathbb{P}^{+} \big( \text{$e$ is pivotal for $\mathcal{A}_{p,q}$} \big).
\end{split}
\end{equation}
\end{lemma}

\begin{proof}
An important point consists in remarking that the event 
\[
\mathcal{A}_{p,q} = \big\{ o \to \partial\Lambda_n \; \mbox{in $\mathcal{G}_{p,q}$ without \scalebox{0.7}{\pattern}} \big\}
\]
only depends on edges of $\mathcal{E}^{\tiny{\mbox{in}}}_n$ and $\mathcal{D}^{\tiny{\mbox{in}}}_n$. Consider a $q$-edge $e=\{x,y\}$ whose endpoint $y$ does not belong to $\Lambda_n\!\setminus\!\partial\Lambda_n$. Even if its associated forbidden pattern will occur, there would exist an open path of $p$-edges joining $x$ to $\partial\Lambda_n$ (without making a forbidden pattern) bypassing the edge $e$ and making it useless for the event $\mathcal{A}_{p,q}$. In the same way, a $p$-edge whose both endpoints belong to $\partial\Lambda_n$ cannot help an open path starting at $o$ to join the boundary $\partial\Lambda_n$. Indeed, it could contribute to the occurrence of a forbidden pattern making the associated $q$-edge active, but such $q$-edge would not be in $\mathcal{D}^{\tiny{\mbox{in}}}_n$, and hence useless for $\mathcal{A}_{p,q}$.

The rest of the proof is standard, see~\cite[Chapter~2]{grimmett1999percolation}. We only focus on the case of the derivative with respect to $p$ since that with respect to $q$ is completely similar. Let us replace the parameter $p$ with the vector
\[
\textbf{p} = (p_e)_{e \in \mathcal{E}^{\tiny{\mbox{in}}}_n} \in [0,1]^{\mathcal{E}^{\tiny{\mbox{in}}}_n}.
\]
Let $f \in \mathcal{E}^{\tiny{\mbox{in}}}_n$ and $\textbf{p}' = (p'_e)_{e \in \mathcal{E}^{\tiny{\mbox{in}}}_n}$ where $p'_e=p_e$, for any edge $e \not= f$, and $p_f' \geq p_f$. Lemma~\ref{lem:monotonicity} allows to write
\begin{eqnarray*}
\Theta_n(\textbf{p}',q) - \Theta_n(\textbf{p},q) & = & \mathbb{P}^{+} \big( \mathcal{A}_{\textbf{p}',q} \setminus \mathcal{A}_{\textbf{p},q} \big) \\
& = & \mathbb{P}^{+} \big( U_f \in (p_f,p'_f] , \mathcal{A}_{\textbf{p},q}^e \setminus \mathcal{A}_{\textbf{p},q}^{\neg e} \big) \\
& = & (p'_f - p_f) \mathbb{P}^{+} \big( \text{$f$ is pivotal for $\mathcal{A}_{\textbf{p},q}$} \big)
\end{eqnarray*}
using the independence between the random variable $U_f$ and the fact that $f$ is pivotal. Henceforth, dividing by $p'_f - p_f$ and taking $p'_f - p_f \to 0$, we get
\[
\partial_{p_f} \Theta_n(\textbf{p},q) = \mathbb{P}^{+} \big( \text{$f$ is pivotal for $\mathcal{A}_{\textbf{p},q}$} \big).
\]
Since $\Theta_n(\textbf{p},q)$ only depends on a finite number of variables, we can write
\[
\partial_{p} \Theta_n(p,q) = \sum_{f \in \mathcal{E}^{\tiny{\mbox{in}}}_n} \partial_{p_f} \Theta_n(\textbf{p},q) \big|_{\textbf{p}=(p,\ldots,p)} = \sum_{f \in \mathcal{E}^{\tiny{\mbox{in}}}_n} \mathbb{P}^{+} \big( \text{$f$ is pivotal for $\mathcal{A}_{p,q}$} \big),
\]
which concludes the proof. 
\end{proof}

\subsection{Comparison of pivotal probabilities}

Recall that our goal is to prove Proposition~\ref{prop:derivativeComp}, i.e., to compare the partial derivatives $\partial_p \Theta_n(p,q)$ and $\partial_q \Theta_n(p,q)$. Thanks to Lemma~\ref{lem:Russo}, this can be reduced to a comparison of the probabilities for a $p$-edge and a $q$-edge to be pivotal. In this section, we shorten '$e$ is pivotal for $\mathcal{A}_{p,q}$' to '$e$ is pivotal'.

Recall that $\mathcal{E}^{\tiny{\mbox{in}}}_n$ denotes the set of $p$-edges having at least one endpoint in $\Lambda_n\!\setminus\!\partial\Lambda_n$ and $\mathcal{D}^{\tiny{\mbox{in}}}_n$ the set of $q$-edges having both endpoints in $\Lambda_n\!\setminus\!\partial\Lambda_n$. First of all, note that in the enhanced lattice there is a one-to-one correspondence between $p$-edges and $q$-edges, i.e., between the sets $\mathcal{E}_0$ and $\mathcal{D}$, as depicted in Figure~\ref{fig:associatedEdges}. From now on, for any $p$-edge $e$, we denote by $Q(e)$ its corresponding $q$-edge.

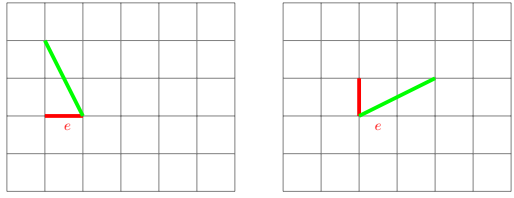
\begin{figure}[!ht]
\centering
\scalebox{0.5}{
\begin{tikzpicture}
\draw[step=1cm,gray,thin] (1,3) grid (7,8);    
\node at (2.6,4.7){\textcolor{red}{$e$}};
\draw[red,line width=1mm] (2,5) -- (3,5);
\draw[green,line width=1mm] (3,5) -- (2,7);
\end{tikzpicture}
\hspace*{1cm}
\begin{tikzpicture}
\draw[step=1cm,gray,thin] (1,3) grid (7,8);    
\node at (3.5,4.7){\textcolor{red}{$e$}};
\draw[red,line width=1mm] (3,5) -- (3,6);
\draw[green,line width=1mm] (3,5) -- (5,6);
\end{tikzpicture}
}
\caption{Two $p$-edges (in red) and their associated $q$-edges (in green).}
\label{fig:associatedEdges}
\end{figure}

The basic idea is the following: given a $p$-edge $e \in \mathcal{E}^{\tiny{\mbox{in}}}_n$, we will perform a local modification around $e$ so that the event $\{\text{$e$ is pivotal}\}$ becomes $\{\text{$Q(e)$ is pivotal}\}$ while comparing their probabilities. This strategy will apply successfully to most $p$-edges in $\mathcal{E}^{\tiny{\mbox{in}}}_n$, but not to all of them. Indeed, some $p$-edges satisfy $\mathbb{P}^{+}(\text{$e$ is pivotal}) > 0$ while $\mathbb{P}^{+}(\text{$Q(e)$ is pivotal}) = 0$. In that case, our strategy fails. Such pathological $p$-edges $e \in \mathcal{E}^{\tiny{\mbox{in}}}_n$ are of two types:
\begin{itemize}
\item When $Q(e) \notin \mathcal{D}^{\tiny{\mbox{in}}}_n$ (i.e., the associated $q$-edge $Q(e)$ touches the boundary $\partial\Lambda_n$). We have already noticed in Lemma~\ref{lem:Russo} that the event $\mathcal{A}_{p,q}$ only depends on $q$-edges of $\mathcal{D}^{\tiny{\mbox{in}}}_n$.
\item When $e$ belongs to the set
\begin{multline*}\mathcal{I}_{\text{bad}}=\Big\{\{(0,-1),(0,0)\},\{(-1,-1),(-1,0)\},\{(-1,0),(-1,1)\},\{(-2,0),(-2,1)\},\\ \{(0,0),(1,0)\},\{(0,-1),(1,-1)\},\{(-1,-1),(-1,0)\},\{(-2,-1),(-2,0)\}\Big\}.
\end{multline*}
In that case, the origin $o$ is one of the vertices located on the forbidden pattern associated to $e$ (except both extremities) and the associated $q$-edge $Q(e)$ cannot be pivotal for $\mathcal{A}_{p,q}$.
\end{itemize}
To overcome this difficulty, we make a distinction between the so-called {\em good} and {\em bad} edges: a $p$-edge $e \in  \mathcal{E}^{\tiny{\mbox{in}}}_n$ is called {\em good} if $Q(e)\in \mathcal{D}^{\tiny{\mbox{in}}}_n$ and $e\notin \mathcal{I}_{\text{ bad}}$. When $e$ is good, we will manage to compare $\mathbb{P}^{+}(\text{$e$ is pivotal})$ and $\mathbb{P}^{+}(\text{$Q(e)$ is pivotal})$ in~\eqref{PivotGood} below. An edge which is not good is called {\em bad}. In other words, bad edges are those edges $e$ such that the associated $q$-edges cannot be pivotal while $e$ itself is pivotal with positive probability. When $e$ is bad, we will only be able to compare its probability to be pivotal to that of another edge which will be good. See \eqref{PivotBad} below. Let us denote by $\mathcal{E}^{\tiny{\mbox{good}}}_n$ the set of good $p$-edges, and by $\mathcal{E}^{\tiny{\mbox{bad}}}_n$ its complement set in $\mathcal{E}^{\tiny{\mbox{in}}}_n$, i.e., $\mathcal{E}^{\tiny{\mbox{bad}}}_n := \mathcal{E}^{\tiny{\mbox{in}}}_n\!\setminus\!\mathcal{E}^{\tiny{\mbox{good}}}_n$. The next result allows to prove Proposition~\ref{prop:derivativeComp}. Its proof is postponed to the two following subsections.

\begin{lemma}
\label{lem:PivotGood-Bad}
The following inequalities hold:
\begin{equation}
\label{PivotGood}
\forall e \in \mathcal{E}^{\tiny{\mbox{good}}}_n , \; \mathbb{P}^+\!\big( \text{$e$ is pivotal} \big) \leq [p(1-p)]^{-71} \mathbb{P}^+\!\big( \text{$Q(e)$ is pivotal} \big), and
\end{equation}
\begin{equation}
\label{PivotBad}
\sum_{e \in \mathcal{E}^{\tiny{\mbox{bad}}}_n} \mathbb{P}^+\!\big( \text{$e$ is pivotal} \big) \leq 5 [p(1-p)]^{-71} \sum_{e \in \mathcal{E}^{\tiny{\mbox{good}}}_n} \mathbb{P}^+\!\big( \text{$e$ is pivotal} \big).
\end{equation}
\end{lemma}

\begin{proof}[Proof of Proposition~\ref{prop:derivativeComp}] Using Lemma~\ref{lem:PivotGood-Bad}, we first write
\begin{eqnarray*}
\sum_{e \in \mathcal{E}^{\tiny{\mbox{in}}}_n} \mathbb{P}^{+}\!\big( \text{$e$ is pivotal} \big) & = & \sum_{e \in \mathcal{E}^{\tiny{\mbox{good}}}_n} \mathbb{P}^{+}\!\big( \text{$e$ is pivotal} \big) + \sum_{e \in \mathcal{E}^{\tiny{\mbox{bad}}}_n} \mathbb{P}^{+}\!\big( \text{$e$ is pivotal} \big) \\
& \leq & \big( 1 + 5 [p(1-p)]^{-71} \big) \sum_{e \in \mathcal{E}^{\tiny{\mbox{good}}}_n} \mathbb{P}^{+}\!\big( \text{$e$ is pivotal} \big) \\
& \leq & \big( 1 + 5 [p(1-p)]^{-71} \big) [p(1-p)]^{-71} \sum_{e \in \mathcal{E}^{\tiny{\mbox{good}}}_n} \mathbb{P}^{+}\!\big( \text{$Q(e)$ is pivotal} \big) \\
& \leq & \big( 1 + 5 [p(1-p)]^{-71} \big) [p(1-p)]^{-71} \sum_{e \in \mathcal{D}^{\tiny{\mbox{in}}}_n} \mathbb{P}^{+}\!\big( \text{$e$ is pivotal} \big).
\end{eqnarray*}
Lemma~\ref{lem:Russo} then provides
\[
\partial_p \Theta_n(p,q) \leq \left( 5 [p(1-p)]^{-142} + [p(1-p)]^{-71} \right) \partial_q \Theta_n(p,q)
\]
from which we get Proposition~\ref{prop:derivativeComp} with $C := 5 \times 10^{710}+10^{355}$ since by hypothesis $p\in [10^{-5},1-10^{-5}]$.
\end{proof}

\subsubsection{From good $p$-edges to $q$-edges.}
\label{sect:pToq}
This section is devoted to the proof of~\eqref{PivotGood} in Lemma~\ref{lem:PivotGood-Bad}. Let us consider a good $p$-edge $e \in \mathcal{E}^{\tiny{\mbox{good}}}_n$. By symmetry of the lattice $\Z^2$, we can assume without loss of generality that $e$ is vertical, precisely $e := \{(a,b),(a,b+1)\}$. Let us consider the box
\[
B := \big( \llbracket0,6\rrbracket \times \llbracket0,5\rrbracket \big) \setminus \big\{ (0,0),(0,5),(6,0),(6,5) \big\} \subset \Z^2
\]
(without corners) and thus set $B_e := B + (a-2,b-2)$. See Figure~\ref{fig:box} for an illustration.

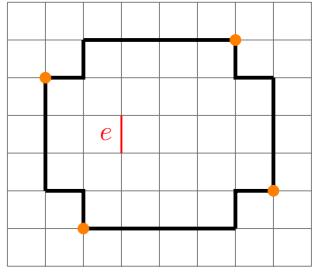
\begin{figure}[!ht]
\centering
\scalebox{0.5}{
\begin{tikzpicture}
\draw[step=1cm,gray,thin] (0,1) grid (8,8);    
\draw[red,line width=0.5mm] (3,4) -- (3,5);
\draw[line width=1mm] (1,3) -- (1,6);
\draw[line width=1mm] (2,7) -- (6,7);
\draw[line width=1mm] (7,6) -- (7,3);
\draw[line width=1mm] (6,2) -- (2,2);
\draw[line width=1mm] (2,2) -- (2,3) -- (1,3);
\draw[line width=1mm] (6,2) -- (6,3) -- (7,3);
\draw[line width=1mm] (1,6) -- (2,6) -- (2,7);
\draw[line width=1mm] (7,6) -- (6,6) -- (6,7);
\node at (2.6,4.5){\textcolor{red}{\huge$e$}};
\filldraw[orange] (2,2) circle (4pt);
\filldraw[orange] (1,6) circle (4pt);
\filldraw[orange] (7,3) circle (4pt);
\filldraw[orange] (6,7) circle (4pt);
%\node at (2.6,1.5){\huge$B_e$};
\end{tikzpicture}}
\caption{The vertical edge $e := \{(a,b),(a,b+1)\}$ (in red) and the box $B_e$ whose border is delimited by a black line. Moreover, the orange dots represent the points to avoid (if it is possible) according to rule ($\ast$). See below.}
\label{fig:box}
\end{figure}

Note that, although $e$ is a good $p$-edge, the box $B_e$ could exceed $\Lambda_n$. Let us denote by $\text{In}(B_e)$ the set of edges in $\mathcal{E}_0 \cup \mathcal{D}$ having both endpoints in $B_e \cap \Lambda_n$ and by $\text{Out}(B_e)$ the set of edges in $\mathcal{E}_0 \cup \mathcal{D}$ having both endpoints in $\Lambda_n$ and at most one in $B_e$. Now, let us consider the set $\text{Piv}^\text{Out}_{B_e}$ of enhanced configurations which are compatible outside the box $B_e$ with the fact that $e$ is pivotal:
\[
\text{Piv}^\text{Out}_{B_e} := \big\{ \omega \in \Omega^+ \colon \exists \eta \in \{ \text{$e$ is pivotal} \} \, s.t.~\omega(e') = \eta(e') , \; \forall e' \in \text{Out}(B_e) \big\}.
\]
The event $\text{Piv}^\text{Out}_{B_e}$ belongs to $\mathcal{G}_{\text{Out}(B_e)}$ the $\sigma$-algebra generated by the states of the edges in $\text{Out}(B_e)$.

Here is our local modification argument. Any configuration $\omega$ in $\text{Piv}^\text{Out}_{B_e}$ can be modified inside the box $B_e$ into a new configuration $T_e(\omega)$ for which $Q(e)$ is now pivotal.

\begin{lemma}
\label{lem:mapToQedge}
Let $e \in \mathcal{E}^{\tiny{\mbox{good}}}_n$, where the set of good edges is defined as above. Then, there exists a measurable map $T_e\colon \Omega^+\to\Omega^+ $ such that, for any $\omega \in \Omega^+$,
\begin{itemize}
\item[(i)] The configurations $\omega$ and $T_e(\omega)$ may only differ on $\text{In}(B_e)$.
\item[(ii)] $T_e(\omega)$ is defined on $\text{In}(B_e)$ in a deterministic way according to $\omega$ on $\text{Out}(B_e)$.
\item[(iii)] If $\omega \in \text{Piv}^\text{Out}_{B_e}$ and all its $p$-edges in $\text{In}(B_e)$ coincide with those of $T_e(\omega)$ then $\omega \in \{ \text{$Q(e)$ is pivotal} \}$.
\end{itemize}
\end{lemma}

The main idea behind the construction is summarized in Figure~\ref{fig:enhancement-good-schematic} and a concrete visualization of the construction in all possible cases can be found online\footnote{\url{https://bennhenry.github.io/NeighPerc/}}.  We postpone the proof of this result to the end of this section and first see how it is used to prove~\eqref{PivotGood} in Lemma~\ref{lem:PivotGood-Bad}. 

\begin{proof}[Proof of \eqref{PivotGood}]
Let us consider the set of configurations $\omega$ coinciding with their modifications $T_e(\omega)$ inside the box $B_e$ but only for $p$-edges, more precisely we define
\[
\mathcal{R}_{T_e} := \big\{ \omega \in \Omega^+ \colon \omega(e') = T_e(\omega)(e') \text{ for any $p$-edge $e'$ in $\text{In}(B_e)$} \big\}.
\]
Then Item (iii) of Lemma~\ref{lem:mapToQedge} tells us that $\text{Piv}^\text{Out}_{B_e} \cap \mathcal{R}_{T_e} \subset \{\text{$Q(e)$ is pivotal}\}$.

Let us work conditionally on $\mathcal{G}_{\text{Out}(B_e)}$. The configuration $T_e$ then becomes deterministic on $\text{In}(B_e)$ by Lemma~\ref{lem:mapToQedge}, Item (ii). Hence the fact that a configuration $\omega$ belongs to the event $\mathcal{R}_{T_e}$ only depends on the states of its $p$-edges in $\text{In}(B_e)$-- that are $71$ in number and even less if $B_e$ exceeds $\Lambda_n$ --which have to be open or closed according to $T_e(\omega)$. Consequently,
\begin{equation}
\label{Minor-R}
\mathbb{P}^+\!\big( \mathcal{R}_{T_e} \,|\, \mathcal{G}_{\text{Out}(B_e)} \big) \geq [p(1-p)]^{71}.
\end{equation}

Using the inclusions $\{\text{$e$ is pivotal}\} \subset \text{Piv}^\text{Out}_{B_e}$ and $\text{Piv}^\text{Out}_{B_e} \cap \mathcal{R}_{T_e} \subset \{\text{$Q(e)$ is pivotal}\}$ given by Lemma~\ref{lem:mapToQedge}, we get
\begin{eqnarray*}
\mathbb{P}^+\!\big( \text{$Q(e)$ is pivotal} \big) & = & \mathbb{E} \big[ \mathbb{P}^+\!\big( \text{$Q(e)$ is pivotal} \,|\, \mathcal{G}_{\text{Out}(B_e)} \big) \big] \\
& \geq & \mathbb{E} \big[ \mathds{1}_{\text{Piv}^\text{Out}_{B_e}} \, \mathbb{P}^+\!\big( \mathcal{R}_{T_e} \,|\, \mathcal{G}_{\text{Out}(B_e)} \big) \big] \\
& \geq & [p(1-p)]^{71} \, \mathbb{P}^+\!\big( \text{Piv}^\text{Out}_{B_e} \big) \\
& \geq & [p(1-p)]^{71} \, \mathbb{P}^+\!\big( \text{$e$ is pivotal} \big),
\end{eqnarray*}
thanks to $\text{Piv}^\text{Out}_{B_e} \in \mathcal{G}_{\text{Out}(B_e)}$ and~\eqref{Minor-R}. Statement~\eqref{PivotGood} of Lemma~\ref{lem:PivotGood-Bad} is then proved.
\end{proof}

\begin{figure}[!ht]
\centering
\includegraphics[width=.65\textwidth]{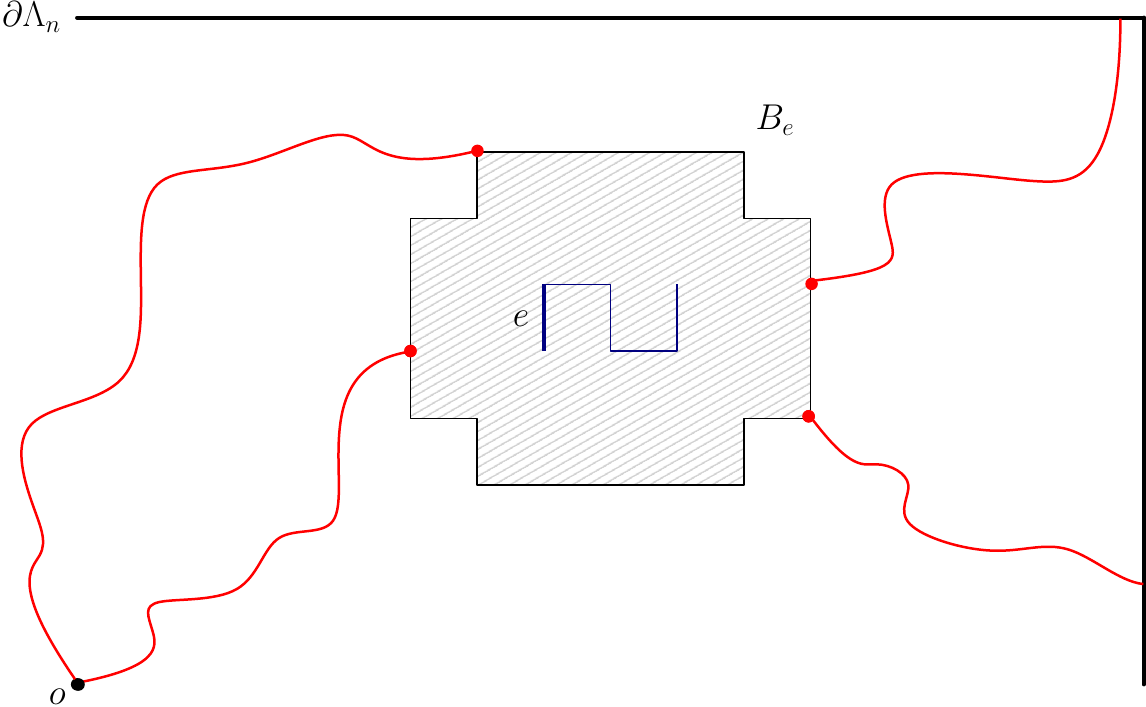}
\caption{\label{fig:enhancement-good-schematic} For good edges $e$ and their associated forbidden pattern we only need to change the status of edges in the shaded region $B_e$ around $e$. If $e$ is a pivotal p-edge, then there is at least one admissible path from the origin to the boundary of $B_e$ and at least one path connecting the boundary of $B_e$ to the boundary of $\Lambda_n$ (in red). By taking some care in choosing a suitable pair out of possibly multiple options one can then open and close edges of $\text{In}(B_e)$ only to create a configuration in which $Q(e)$ is pivotal without creating the forbidden pattern elsewhere.}
\end{figure}

Let us now provide the construction of the measurable mapping $T_e$ in Lemma~\ref{lem:mapToQedge}.

\begin{proof}[Proof of Lemma~\ref{lem:mapToQedge}]
Let $\omega \in \Omega^+$. First, for any edge $e'$ not included in $\text{In}(B_e)$, we set $T_e(\omega)(e') := \omega(e')$ and, for any edge $e'$ included in $\text{In}(B_e)$ but diagonal, we set $T_e(\omega)(e') := 0$ (they are all closed). In the case where $\omega \notin \text{Piv}^\text{Out}_{B_e}$, we also set $T_e(\omega)(e') := 0$ for any $p$-edges $e'$ of $\text{In}(B_e)$. Lemma~\ref{lem:mapToQedge} is then proved in this case.

Now, let us assume that $\omega \in \text{Piv}^\text{Out}_{B_e}$. Let $\partial B_e$ be the vertices of $B_e$ having a neighbor outside $B_e$. We also assume that $o \notin (B_e\!\setminus\!\partial B_e)$ and $B_e \cap \partial \Lambda_n = \emptyset$. These particular cases will be treated later. The proof relies on a semi-explicit construction of the configuration $T_e(\omega)$ on $\text{In}(B_e)$. Since $\omega \in \text{Piv}^\text{Out}_{B_e}$ the following holds. For the configuration $\omega$, there exist an admissible path (i.e., without forbidden patterns whose associated $q$-edge is closed) from $o$ to some vertex $x \in \partial B_e$ and an admissible path from some $y \in \partial B_e$ to $\partial\Lambda_n$ using no edges of $\text{In}(B_e)$. Moreover, there is no admissible path from $o$ to $\partial \Lambda_n$ which does not use any edge of $\text{In}(B_e)$, meaning that the previous vertices $x$ and $y$ must be different. Such vertices $x$ and $y$ are respectively called entry and exit points. Since a configuration $\omega$ may admit several pairs $(x,y)$ of entry-exit points, we pick one of them according to any given deterministic rule satisfying the constraint ($\ast$):
\[
\mbox{($\ast$)} \hspace*{1cm} \begin{array}{c} \mbox{Avoid choosing any point of $(a-2, b-2)+\{(1,0),\ (0,4),\ (5,5),\ (6,1)\}$} \\
\mbox{as entry or exit point if possible.}
\end{array}
\]
These points to avoid are represented by orange dots on Figure~\ref{fig:box}. Let us notice that the case $o \in \partial B_e$ being allowed here, the entry point $x$ could be the origin $o$.

There is a cycle $C$ in the box $B_e$ of $p$-edges at distance $1$ from $\partial B_e$ (surrounding the forbidden pattern associated to $e$). On this cycle, we identify two vertices $u := (a,b-1)$ and $v := (a+2,b+2)$ that split $C$ into two paths, a green one and a magenta one. See Figure~\ref{fig:constructionOfT-1}. Let $x'$ and $y'$ be the vertices on $C$ at distance $1$ respectively from the entry point $x$ and the exit point $y$. Although $x$ an $y$ are different, $x'$ and $y'$ could be equal. The construction of $T_e(\omega)$ on $\text{In}(B_e)$ will be different according to $x' \not= y'$ or not.

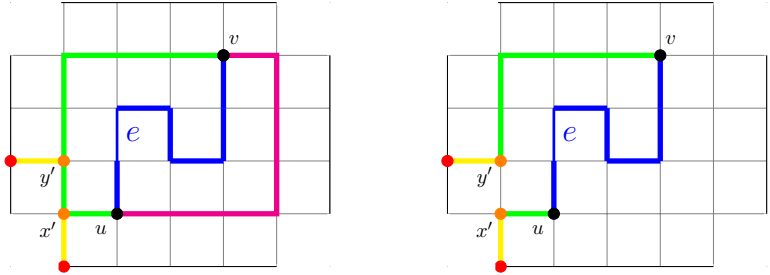
\begin{figure}[!ht]
\centering
\scalebox{0.7}{
\begin{tikzpicture}
\draw[step=1cm,gray,thin] (1,2) grid (7,7);    
\draw[blue,line width=0.5mm] (3,4) -- (3,5);
\draw (1,2) -- (1,7);
\draw (1,7) -- (7,7);
\draw (7,7) -- (7,2);
\draw (7,2) -- (1,2);
\draw[yellow,line width=1mm] (1,4) -- (2,4);
\draw[yellow,line width=1mm] (2,2) -- (2,3);
\filldraw[red] (1,4) circle (3pt);
\filldraw[red] (2,2) circle (3pt);
\draw[blue,line width=1mm] (3,5) -- (4,5);
\draw[blue,line width=1mm] (4,5) -- (4,4);
\draw[blue,line width=1mm] (4,4) -- (5,4);
\draw[blue,line width=1mm] (5,4) -- (5,5);
\draw[blue,line width=1mm] (5,5) -- (5,6);
\draw[blue,line width=1mm] (3,3) -- (3,4);
\draw[green,line width=1mm] (3,3) -- (2,3) -- (2,4) -- (2,5) -- (2,6) -- (3,6) -- (4,6) -- (5,6);
\draw[magenta,line width=1mm] (3,3) -- (4,3) -- (5,3) -- (6,3) -- (6,4) -- (6,5) -- (6,6) -- (5,6);
\node at (3.3,4.5){\textcolor{blue}{\LARGE$e$}};
\filldraw[orange] (2,4) circle (3pt);
\filldraw[orange] (2,3) circle (3pt);
\filldraw[black] (3,3) circle (3pt);
\filldraw[black] (5,6) circle (3pt);
\node at (1.7,2.7) {$x'$};
\node at (1.7,3.7) {$y'$};
\node at (2.7,2.7) {$u$};
\node at (5.2,6.3) {$v$};
\filldraw[black] (5,6) circle (3pt);
\draw[white,line width=1mm] (1,2) -- (1,3);
\draw[white,line width=1mm] (1,2) -- (1.9,2);

\draw[white,line width=1mm] (1,6) -- (1,7);
\draw[white,line width=1mm] (1,7) -- (1.95,7);

\draw[white,line width=1mm] (6,7) -- (7,7);
\draw[white,line width=1mm] (7,7) -- (7,6);

\draw[white,line width=1mm] (6,2) -- (7,2);
\draw[white,line width=1mm] (7,2) -- (7,3);
\end{tikzpicture}
}
\hspace*{1cm}
\scalebox{0.7}{
\begin{tikzpicture}
\draw[step=1cm,gray,thin] (1,2) grid (7,7);    
\draw[blue,line width=0.5mm] (3,4) -- (3,5);
\draw (1,2) -- (1,7);
\draw (1,7) -- (7,7);
\draw (7,7) -- (7,2);
\draw (7,2) -- (1,2);
\draw[yellow,line width=1mm] (1,4) -- (2,4);
\draw[yellow,line width=1mm] (2,2) -- (2,3);
\filldraw[red] (1,4) circle (3pt);
\filldraw[red] (2,2) circle (3pt);
\draw[blue,line width=1mm] (3,5) -- (4,5);
\draw[blue,line width=1mm] (4,5) -- (4,4);
\draw[blue,line width=1mm] (4,4) -- (5,4);
\draw[blue,line width=1mm] (5,4) -- (5,5);
\draw[blue,line width=1mm] (5,5) -- (5,6);
\draw[blue,line width=1mm] (3,3) -- (3,4);
\draw[green,line width=1mm] (3,3) -- (2,3);
\draw[green,line width=1mm] (2,4) -- (2,5) -- (2,6) -- (3,6) -- (4,6) -- (5,6);
\node at (3.3,4.5){\textcolor{blue}{\LARGE$e$}};
\filldraw[orange] (2,4) circle (3pt);
\filldraw[orange] (2,3) circle (3pt);
\filldraw[black] (3,3) circle (3pt);
\filldraw[black] (5,6) circle (3pt);
\node at (1.7,2.7) {$x'$};
\node at (1.7,3.7) {$y'$};
\node at (2.7,2.7) {$u$};
\node at (5.2,6.3) {$v$};

\draw[white,line width=1mm] (1,2) -- (1,3);
\draw[white,line width=1mm] (1,2) -- (1.9,2);

\draw[white,line width=1mm] (1,6) -- (1,7);
\draw[white,line width=1mm] (1,7) -- (1.95,7);

\draw[white,line width=1mm] (6,7) -- (7,7);
\draw[white,line width=1mm] (7,7) -- (7,6);

\draw[white,line width=1mm] (6,2) -- (7,2);
\draw[white,line width=1mm] (7,2) -- (7,3);
\end{tikzpicture}
}
\caption{\label{fig:constructionOfT-1} Construction of $T_e(\omega)$ when $x' \not= y'$ (in orange) and they are on the green path. Vertices $x$ and $y$ are represented by red dots. Edges $\{x,x'\}$ and $\{y,y'\}$ are in yellow. Blue edges are $\{u,(a,b)\}$, $\{(a+2,b+1),v\}$ and those corresponding to the forbidden pattern associated to $e$. Left: the green and magenta paths. Right: the yellow, green and blue edges form the multicolor path $\pi$.}
\end{figure}

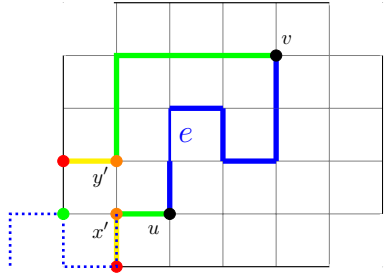
\begin{figure}[!ht]
\centering
\scalebox{0.7}{
\begin{tikzpicture}
\draw[step=1cm,gray,thin] (1,2) grid (7,7);    
\draw[blue,line width=0.5mm] (3,4) -- (3,5);
\draw (1,2) -- (1,7);
\draw (1,7) -- (7,7);
\draw (7,7) -- (7,2);
\draw (7,2) -- (1,2);
\draw[yellow,line width=1mm] (1,4) -- (2,4);
\draw[yellow,line width=1mm] (2,2) -- (2,3);
\filldraw[red] (1,4) circle (3pt);
\filldraw[red] (2,2) circle (3pt);
\draw[blue,line width=1mm] (3,5) -- (4,5);
\draw[blue,line width=1mm] (4,5) -- (4,4);
\draw[blue,line width=1mm] (4,4) -- (5,4);
\draw[blue,line width=1mm] (5,4) -- (5,5);
\draw[blue,line width=1mm] (5,5) -- (5,6);
\draw[blue,line width=1mm] (3,3) -- (3,4);
\draw[green,line width=1mm] (3,3) -- (2,3);
\draw[green,line width=1mm] (2,4) -- (2,5) -- (2,6) -- (3,6) -- (4,6) -- (5,6);
\node at (3.3,4.5){\textcolor{blue}{\LARGE$e$}};
\filldraw[orange] (2,4) circle (3pt);
\filldraw[orange] (2,3) circle (3pt);

\filldraw[black] (3,3) circle (3pt);
\filldraw[black] (5,6) circle (3pt);
\node at (1.7,2.7) {$x'$};
\node at (1.7,3.7) {$y'$};
\node at (2.7,2.7) {$u$};
\node at (5.2,6.3) {$v$};

\draw[white,line width=1mm] (1,2) -- (1,3);
\draw[white,line width=1mm] (1,2) -- (1.9,2);

\draw[white,line width=1mm] (1,6) -- (1,7);
\draw[white,line width=1mm] (1,7) -- (1.95,7);

\draw[white,line width=1mm] (6,7) -- (7,7);
\draw[white,line width=1mm] (7,7) -- (7,6);

\draw[white,line width=1mm] (6,2) -- (7,2);
\draw[white,line width=1mm] (7,2) -- (7,3);

\draw[blue, dotted, line width=0.5mm] (2,3) -- (2,2) -- (1,2) -- (1,3) -- (0,3) -- (0,2);
\filldraw[green] (1,3) circle (3pt);
\end{tikzpicture}
}
\caption{\label{fig:constructionOfT-undesired} Opening the yellow edge $\{x,x'\}$ may create a forbidden pattern with the exterior of $B_e$ (represented by the dotted blue edges). However, if so, it contradicts constraint $(\ast)$ as the green vertex could have been chosen as entry/exit point.}
\end{figure}

Let us first consider the case where $x' \not= y'$ and they belong to the same colored path (say the green one for instance). Thus, removing the green edges between $x'$ and $y'$, and also the magenta path, we finally get a multicolor path, say $\pi$, as depicted in Figure~\ref{fig:constructionOfT-1}, made up with yellow edges $\{x,x'\}$ and $\{y,y'\}$, retained green edges joining $x',y'$ to $u,v$ and blue edges from $u$ to $v$. We point out here that the constraint $(\ast)$ allows to prevent the creation of a forbidden pattern when opening the edges $\{x,x'\}$ or $\{y,y'\}$: see Figure~\ref{fig:constructionOfT-undesired}. When $x'$ and $y'$ belong to different colored paths, say $x'$ on the green one and $y'$ on the magenta one, we proceed in a slightly different way. As depicted on Figure~\ref{fig:constructionOfT-2}, we retain this time green edges between $x'$ and $u$ and magenta edges between $y'$ and $v$. This provides a new multicolor path still denoted by $\pi$. In both cases, we have built a (multicolor) path $\pi$ of $p$-edges joining $x$ and $y$ having exactly one forbidden pattern in $\text{In}(B_e)$, the one associated to the edge $e$. For the configuration $T_e(\omega)$ restricted to edges of $\text{In}(B_e)$, we only declare open the edges belonging to the multicolor path $\pi$, all others being declared closed.

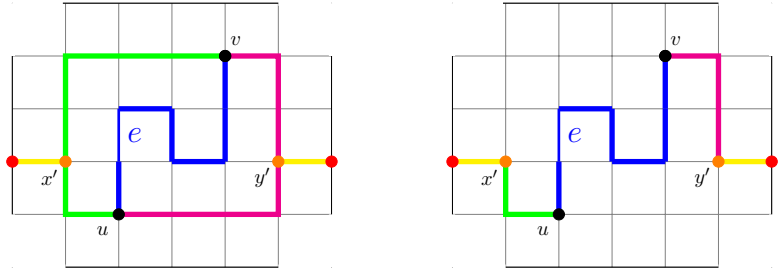
\begin{figure}[!ht]
\centering
\scalebox{0.7}{
\begin{tikzpicture}
\draw[step=1cm,gray,thin] (1,2) grid (7,7);    
\draw[blue,line width=0.5mm] (3,4) -- (3,5);
\draw (1,2) -- (1,7);
\draw (1,7) -- (7,7);
\draw (7,7) -- (7,2);
\draw (7,2) -- (1,2);
\draw[yellow,line width=1mm] (1,4) -- (2,4);
\draw[yellow,line width=1mm] (7,4) -- (6,4);
\filldraw[red] (1,4) circle (3pt);
\filldraw[red] (7,4) circle (3pt);
\draw[blue,line width=1mm] (3,5) -- (4,5);
\draw[blue,line width=1mm] (4,5) -- (4,4);
\draw[blue,line width=1mm] (4,4) -- (5,4);
\draw[blue,line width=1mm] (5,4) -- (5,5);
\draw[blue,line width=1mm] (5,5) -- (5,6);
\draw[blue,line width=1mm] (3,3) -- (3,4);
\draw[green,line width=1mm] (3,3) -- (2,3) -- (2,4) -- (2,5) -- (2,6) -- (3,6) -- (4,6) -- (5,6);
\draw[magenta,line width=1mm] (3,3) -- (4,3) -- (5,3) -- (6,3) -- (6,4) -- (6,5) -- (6,6) -- (5,6);
\node at (3.3,4.5){\textcolor{blue}{\LARGE$e$}};
%\node at (2.6,1.5){$B_e$};
\filldraw[orange] (2,4) circle (3pt);
\filldraw[orange] (6,4) circle (3pt);
\filldraw[black] (3,3) circle (3pt);
\filldraw[black] (5,6) circle (3pt);
\node at (1.7,3.7) {$x'$};
\node at (5.7,3.7) {$y'$};
\node at (2.7,2.7) {$u$};
\node at (5.2,6.3) {$v$};
\filldraw[black] (5,6) circle (3pt);
\draw[white,line width=1mm] (1,2) -- (1,3);
\draw[white,line width=1mm] (1,2) -- (2,2);

\draw[white,line width=1mm] (1,6) -- (1,7);
\draw[white,line width=1mm] (1,7) -- (1.95,7);

\draw[white,line width=1mm] (6,7) -- (7,7);
\draw[white,line width=1mm] (7,7) -- (7,6);

\draw[white,line width=1mm] (6,2) -- (7,2);
\draw[white,line width=1mm] (7,2) -- (7,3);
\end{tikzpicture}
}
\hspace*{1cm}
\scalebox{0.7}{
\begin{tikzpicture}
\draw[step=1cm,gray,thin] (1,2) grid (7,7);    
\draw[blue,line width=0.5mm] (3,4) -- (3,5);
\draw (1,2) -- (1,7);
\draw (1,7) -- (7,7);
\draw (7,7) -- (7,2);
\draw (7,2) -- (1,2);
\draw[yellow,line width=1mm] (1,4) -- (2,4);
\draw[yellow,line width=1mm] (7,4) -- (6,4);
\filldraw[red] (1,4) circle (3pt);
\filldraw[red] (7,4) circle (3pt);
\draw[blue,line width=1mm] (3,5) -- (4,5);
\draw[blue,line width=1mm] (4,5) -- (4,4);
\draw[blue,line width=1mm] (4,4) -- (5,4);
\draw[blue,line width=1mm] (5,4) -- (5,5);
\draw[blue,line width=1mm] (5,5) -- (5,6);
\draw[blue,line width=1mm] (3,3) -- (3,4);
\draw[magenta,line width=1mm] (6,4) -- (6,5) -- (6,6) -- (5,6);
\draw[green,line width=1mm] (3,3) -- (2,3) -- (2,4);
\node at (3.3,4.5){\textcolor{blue}{\LARGE$e$}};
%\node at (2.6,1.5){$B_e$};
\filldraw[orange] (2,4) circle (3pt);
\filldraw[orange] (6,4) circle (3pt);
\filldraw[black] (3,3) circle (3pt);
\filldraw[black] (5,6) circle (3pt);
\node at (1.7,3.7) {$x'$};
\node at (5.7,3.7) {$y'$};
\node at (2.7,2.7) {$u$};
\node at (5.2,6.3) {$v$};
\filldraw[black] (5,6) circle (3pt);
\draw[white,line width=1mm] (1,2) -- (1,3);
\draw[white,line width=1mm] (1,2) -- (2,2);

\draw[white,line width=1mm] (1,6) -- (1,7);
\draw[white,line width=1mm] (1,7) -- (1.95,7);

\draw[white,line width=1mm] (6,7) -- (7,7);
\draw[white,line width=1mm] (7,7) -- (7,6);

\draw[white,line width=1mm] (6,2) -- (7,2);
\draw[white,line width=1mm] (7,2) -- (7,3);
\end{tikzpicture}
}
\caption{\label{fig:constructionOfT-2}Construction of $T_e(\omega)$ when $x' \not= y'$ (in orange) and $x'$ belongs to the green path, $y'$ belongs to the magenta path. To the right: we build a multicolor path $\pi$ with the yellow edges $\{x,x'\}$ and $\{y,y'\}$, the green edges between $x'$ and $u$, the magenta edges between $y'$ and $v$, and the blue edges.}
\end{figure}

So, we have modified the configuration $\omega \in \text{Piv}^\text{Out}_{B_e}$ inside the box $B_e$ into a new configuration $T_e(\omega)$ which belongs to the event $\{ \text{$Q(e)$ is pivotal} \}$. First, $\omega$ and $T_e(\omega)$ may only differ on $\text{In}(B_e)$ (Item $(i)$). This local modification only depends on the couple of entry-exit points, i.e., on the configuration $\omega$ on $\text{Out}(B_e)$ (Item $(ii)$). Besides, the configuration $T_e(\omega)$ inside $\text{In}(B_e)$ makes active only one $q$-edge, namely $Q(e)$, since by construction $T_e(\omega)$ contains only one forbidden pattern inside $\text{In}(B_e)$, the one associated to $Q(e)$. As the event $\{ \text{$Q(e)$ is pivotal} \}$ does not depend on the state of $Q(e)$, we can assert that if the configuration $\omega \in \text{Piv}^\text{Out}_{B_e}$ admits all its $p$-edges in $\text{In}(B_e)$ equal to those of $T_e(\omega)$ then $\omega \in \{ \text{$Q(e)$ is pivotal} \}$.

\medskip

Let us now focus on the case where $x'=y'$ and we can assume without loss of generality that any possible couple $(x,y)$ of entry-exit point satisfies $x'=y'$. By symmetry, we can assume that $x = (a,b)-(1,2)$ and $y = (a,b)-(2,1)$, see Figure~\ref{fig:constructionOfT-3}. The natural strategy consists in opening (for $T_e(\omega)$) the edge $\{x,x+(1,0)\}$ and to use the couple $(x+(1,0),y)$ as entry-exit points to apply the previous construction. Opening the edge $\{x,x+(1,0)\}$ could create two problems. First, the new open path $\gamma' := \gamma \cup \{x,x+(1,0)\}$, where $\gamma$ is an admissible open path from $o$ to $x$, could not be admissible, the adding of the open edge $\{x,x+(1,0)\}$ possibly creating a forbidden pattern contained in $\gamma'$. But this cannot happen. Besides, opening $\{x,x+(1,0)\}$ could create an admissible path from $o$ to $\partial \Lambda_n$ avoiding the box $B_e$ (except $\{x,x+(1,0)\}$) and then preventing $Q(e)$ to be pivotal. But in that case, $x+(1,0)$ would be a possible exit point and then the couple of entry-exit points $(x,x+(1,0))$ would have been selected instead of $(x,y)$. By hypothesis, this cannot happen too.

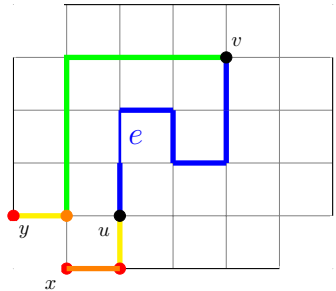
\begin{figure}[!ht]
\centering
\scalebox{0.7}{
\begin{tikzpicture}
\draw[step=1cm,gray,thin] (1,2) grid (7,7);    
\draw[blue,line width=0.5mm] (3,4) -- (3,5);
\draw (1,2) -- (1,7);
\draw (1,7) -- (7,7);
\draw (7,7) -- (7,2);
\draw (7,2) -- (1,2);
\draw[blue,line width=1mm] (3,5) -- (4,5);
\draw[blue,line width=1mm] (4,5) -- (4,4);
\draw[blue,line width=1mm] (4,4) -- (5,4);
\draw[blue,line width=1mm] (5,4) -- (5,5);
\draw[blue,line width=1mm] (5,5) -- (5,6);
\draw[blue,line width=1mm] (3,3) -- (3,4);

%\draw[blue, line width=0.5mm, dotted] (3,2) -- (3,1) -- (2,1) -- (2,0) -- (3,0);
\draw[yellow,line width=1mm] (1,3) -- (2,3);
\draw[yellow,line width=1mm] (3,2) -- (3,3);
\draw[green,line width=1mm] (2,4) -- (2,3);
\draw[green,line width=1mm] (2,4) -- (2,5) -- (2,6) -- (3,6) -- (4,6) -- (5,6);
\node at (3.3,4.5){\textcolor{blue}{\LARGE$e$}};
\filldraw[red] (2,2) circle (3pt);
\filldraw[red] (1,3) circle (3pt);
\filldraw[red] (3,2) circle (3pt);

\filldraw[orange] (2,3) circle (3pt);
\filldraw[black] (3,3) circle (3pt);
\filldraw[black] (5,6) circle (3pt);.
\draw[orange,, line width=1mm] (2,2) -- (3,2);
\node at (2.7,2.7) {$u$};
\node at (5.2,6.3) {$v$};
\node at (1.7,1.7) {$x$};
\node at (1.2,2.7) {$y$};
\draw[white,line width=1mm] (1,2) -- (1,2.9);
\draw[white,line width=1mm] (1,2) -- (1.9,2);

\draw[white,line width=1mm] (1,6) -- (1,7);
\draw[white,line width=1mm] (1,7) -- (1.95,7);

\draw[white,line width=1mm] (6,7) -- (7,7);
\draw[white,line width=1mm] (7,7) -- (7,6);

\draw[white,line width=1mm] (6,2) -- (7,2);
\draw[white,line width=1mm] (7,2) -- (7,3);
%\node at (2.6,1.5){$B_e$};
\end{tikzpicture}}
\caption{\label{fig:constructionOfT-3}Construction of $T_e(\omega)$ in the case $x'=y'$.}
\end{figure}

\medskip

Next, let us study the case where $B_e \cap \partial \Lambda_n \not= \emptyset$. The box $B_e$ exceeds $\Lambda_n\!\setminus\!\partial \Lambda_n$ but not too much since $e$ is a good $p$-edge: $B_e$ has actually to be included in $\Lambda_{n+1}$, so in particular parts of the cycle $C$ are contained in $\partial \Lambda_n$. Since $\omega \in \text{Piv}^\text{Out}_{B_e}$, there exists an admissible path from $o$ to some entry point $x \in \partial B_e \cap \Lambda_n$. Moreover, there is no admissible path from $o$ to $\partial \Lambda_n$ which does not use any edge of $\text{In}(B_e)$. Considering what has been done previously, we choose an entry point $x$ and connect it to $u = (a,b-1)$ or $v$, depending on the choice of $x$ and the location of the cycle $C$ that lies on the boundary of $\Lambda_n$, by opening $p$-edges. We also open up all of the edges of the forbidden pattern and the edges connecting $u$ and $v$ to it. Next, we connect $v = (a+2,b+2)$ (or $u$) to $\partial \Lambda_n$. Hence, we modify the configuration $\omega$ on edges of $\text{In}(B_e)$ into a new configuration $T_e(\omega)$ satisfying the three items of Lemma~\ref{lem:mapToQedge}. 

Finally, it remains to deal with the cases where $o \in (B_e\!\setminus\!\partial B_e)$. Since the $p$-edge $e$ is good, these possible cases are of two types. Either the origin $o$ belongs to the cycle $C$, then $o$ will play the role of $x'$ in the previously explained construction and one only needs to choose an exit point from $B_e$. Or $o$ is one of the two extremities of the forbidden pattern associated to $Q(e)$, i.e., $o = (a,b)$ or $o = (a+2,b+1)$. In this case we obviously do not need to choose an entry point and just choose an exit point $y$. We then connect the pattern to this exit point $y$ by connecting the opposite end of the pattern to $u$ or $v$ and then connect it to the exit point by only using edges on the cycle $C$ plus the edge from $y'$ to $y$. This local modification also leads to a configuration satisfying the required properties. 

For a visualization of the concrete edge-by-edge construction in all the possible cases we refer the interested reader to our online tool\footnote{\url{https://bennhenry.github.io/NeighPerc/}}. 
\end{proof}

\subsubsection{From bad $p$-edges to good $p$-edges}
\label{sect:BadToGood}

Let $e \in \mathcal{E}^{\tiny{\mbox{bad}}}_n$. Recall that by symmetry, we can assume that $e$ is vertical, precisely $e = \{(a,b),(a,b+1)\}$. The edge $e$ being bad means either that the rectangle $L(e) := (a,b)+[0,2] \times [0,1]$ containing its associated $q$-edge $Q(e)$ intersects $\partial \Lambda_n$ or that $$
e\in \Big\{\{(0,-1),(0,0)\},\{(-1,-1),(-1,0)\},\{(-1,0),(-1,1)\},\{(-2,0),(-2,1)\}\Big\},
$$ in which case $Q(e)$ cannot be pivotal as there are no possible path from $0$ to $\partial\Lambda_n$ through the associated forbidden pattern. 
In this section, we aim to pick out a good $p$-edge $J(e)$ and then compare the probabilities for $e$ and $J(e)$ to be pivotal, again through a local modification argument as in the previous section, with the ultimate goal of proving~\eqref{PivotBad}.

The different ways for $e$ to be bad can be split into six distinct categories as laid out below. We then define $J(e)$ accordingly. Let us specify that the {\em north side of $\partial \Lambda_n$} is the set $\llbracket -n,n \rrbracket \times \{n\}$. The {\em south, east and west sides of $\partial \Lambda_n$} are defined similarly. In particular, the corner point $(n,n)$ belongs to two different sides, the north and east ones.
\begin{itemize}
\item[(a)] If the rectangle $L(e)$ intersects only the north side of $\partial \Lambda_n$, meaning that $-(n-1)\leq a\leq n-3$ and $b=n-1$, then we set $J(e) := e-(0,1)$.
\item[(b)] If the rectangle $L(e)$ intersects only the south side of $\partial \Lambda_n$, meaning that $-(n-1)\leq a\leq n-3$ and $b=-n$, then we set $J(e) := e+(0,1)$.
\item[(c)] If the rectangle $L(e)$ intersects only the east side of $\partial \Lambda_n$, meaning that $a \in \{n-2,n-1\}$ and $-(n-1)\leq b \leq n-2$, then we set $J(e) := \{(n-3,b),(n-3,b+1)\}$.
\item[(d)] If the rectangle $L(e)$ intersects both north and east sides of $\partial \Lambda_n$, meaning that $a \in \{n-2,n-1\}$ and $b=n-1$, then we set $J(e) := \{(n-3,n-2),(n-3,n-1)\}$.
\item[(e)] If the rectangle $L(e)$ intersects both south and east sides of $\partial \Lambda_n$, meaning that $a \in \{n-2,n-1\}$ and $b=-n$, then we set $J(e) := \{(n-3,-(n-1)),(n-3,-(n-2)\}$.
\item[(f)] If $e\in  \Big\{\{(0,-1),(0,0)\},\{(-1,-1),(-1,0)\},\{(-1,0),(-1,1)\},\{(-2,0),(-2,1)\}\Big\}$, then we set $J(e) := e-(1,0)$ if $e-(1,0)\notin \mathcal{I}_{\text{bad}}$, $J(e) := e+(0,1)$ otherwise.
\end{itemize}

\begin{figure}[!ht]
\centering
\scalebox{0.5}{
\begin{tikzpicture}
\draw[step=1cm,gray,thin] (2,2) grid (7,7);  
\draw[ultra thick] (7,8)--(7,1);
\draw[green,line width=1mm] (4,4) -- (4,5);
\node at (3.3,4.5){\textcolor{black}{\LARGE$J(e)$}};
\draw[red,line width=1mm] (6,4) -- (6,5);
\node at (5.6,4.5){\textcolor{black}{\LARGE$e$}};
\node at (7.8,6){\huge$\partial \Lambda_n$};
\end{tikzpicture}}
\hspace*{1cm}
\scalebox{0.5}{
\centering
\begin{tikzpicture}
\draw[step=1cm,gray,thin] (1,4) grid (7,8);  
\draw[ultra thick] (7,9)--(7,4);
\draw[ultra thick] (7,4)--(0,4);
\draw[green,line width=1mm] (4,5) -- (4,6);
\node at (3.3,5.5){\textcolor{black}{\LARGE$J(e)$}};
\draw[red,line width=1mm] (5,5) -- (5,6);
\draw[red,line width=1mm] (6,5) -- (6,6);
\draw[orange,line width=1mm] (4,4) -- (4,5);
\draw[blue,line width=1mm] (5,4) -- (5,5);
\draw[blue,line width=1mm] (6,4) -- (6,5);
\node at (7.8,6){\huge$\partial \Lambda_n$};
\end{tikzpicture}}
\hspace*{1cm}
\scalebox{0.45}{
\begin{tikzpicture}
\draw[step=1cm,gray,thin] (0,0) grid (10,10);    
\node at (4.8,4.8){$0$};
\draw[red,line width=1mm] (5,5) -- (5,4);
\draw[red,line width=1mm] (4,5) -- (4,4);
\draw[red,line width=1mm] (4,5) -- (4,6);
\draw[red,line width=1mm] (3,5) -- (3,6);
\draw[red,line width=1mm] (5,4) -- (4,4);
\draw[red,line width=1mm] (5,3) -- (4,3);
\draw[red,line width=1mm] (6,4) -- (5,4);
\draw[red,line width=1mm] (6,5) -- (5,5);
\draw[red,line width=1mm] (9,0) -- (9,9) -- (0,9);
\draw[red,line width=1mm] (8,0) -- (8,8) -- (0,8);
\draw[red,line width=1mm] (8,8) -- (9,8);
\draw[red,line width=1mm] (8,8) -- (8,9);
\foreach \phi in {1,2,...,9}{
\draw[red,line width=1mm] (\phi,0) -- (\phi,1);
}
\foreach \phi in {1,2,...,9}{
\draw[red,line width=1mm] (9,\phi) -- (10,\phi);
}
\foreach \phi in {1,2,...,9}{
\draw[red,line width=1mm] (\phi,9) -- (\phi,10);
}
\foreach \phi in {1,2,...,9}{
\draw[red,line width=1mm] (0,\phi) -- (1,\phi);
}
\end{tikzpicture}
}
\caption{\label{fig:badtogood}{\em Left:} the bad edge $e$ (in red) corresponding to the category (c) and the good edge $J(e)$ (in green). {\em Center:} five bad edges are represented having the same image by the map $J$ (in green); the red bad edges correspond to the category (c), the blue ones to the category (e) and the orange one to the category (b). {\em Right:} all bad edges (vertical or horizontal) for the box $\Lambda_5$ are represented in red.}
\end{figure}
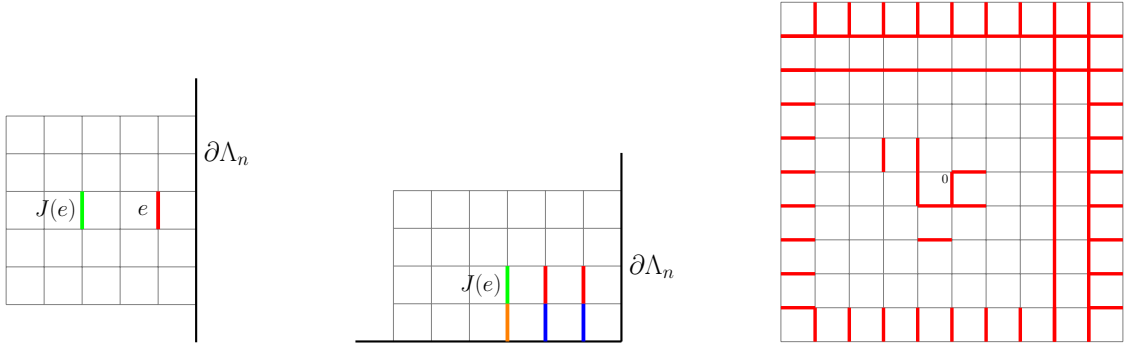

\noindent
See Figure~\ref{fig:badtogood} for examples. For a (finite) set $A$ we denote by $\# A$ the cardinality of $A$. With this notation at hand, we can now state the main step on our way to proving~\eqref{PivotBad}.

\begin{lemma}
\label{lem:badMap}
The map $J\colon  \mathcal{E}^{\tiny{\mbox{\emph{bad}}}}_n \to \mathcal{E}^{\tiny{\mbox{good}}}_n$ defined above satisfies the following properties: for any $e \in \mathcal{E}^{\tiny{\mbox{\emph{good}}}}_n$, $\# J^{-1}(\{e\}) \leq 5$ and for any $e \in \mathcal{E}^{\tiny{\mbox{bad}}}_n$,
\begin{equation}
\label{Bad-J}
\mathbb{P}^+ \left( \emph{$e$ is pivotal} \right) \leq \left[ p(1-p) \right]^{-71} \mathbb{P}^+ \left( \emph{$J(e)$ is pivotal}\right).
\end{equation}
\end{lemma}
We postpone the proof of Lemma~\ref{lem:badMap} to the end of this section and first show how to apply it to conclude the proof of inequality~\eqref{PivotBad} in Lemma~\ref{lem:PivotGood-Bad}. 
\begin{proof}[Proof of~\eqref{PivotBad}]
Lemma~\ref{lem:badMap} immediately gives the desired inequality
\begin{eqnarray*}
\sum_{e \in \mathcal{E}^{\tiny{\mbox{bad}}}_n} \mathbb{P}^+\!\big( \text{$e$ is pivotal} \big) & \leq & [p(1-p)]^{-71} \sum_{e \in \mathcal{E}^{\tiny{\mbox{bad}}}_n} \mathbb{P}^+\!\big( \text{$J(e)$ is pivotal} \big) \\
& \leq & 5 [p(1-p)]^{-71} \sum_{e \in \mathcal{E}^{\tiny{\mbox{good}}}_n} \mathbb{P}^+\!\big( \text{$e$ is pivotal} \big),
\end{eqnarray*}
and the proof is finished. 
\end{proof}

It then remains to prove Lemma~\ref{lem:badMap}. The inequality $\# J^{-1}(\{e\}) \leq 5$ is clear by construction, the worst case being illustrated by categories (d) or (e), and Figure~\ref{fig:badtogood}. So we focus on~\eqref{Bad-J}. The strategy to prove it is very similar to what has been done in Section~\ref{sect:pToq}. Roughly speaking, any configuration $\omega$ in $\text{Piv}^\text{Out}_{B_{J(e)}}$, where
\[
\text{Piv}^\text{Out}_{B_{J(e)}} := \big\{ \omega \in \Omega^+ \colon \exists \eta \in \{ \text{$e$ is pivotal} \} \, s.t.~\omega(e') = \eta(e') , \; \forall e' \in \text{Out}(B_{J(e)}) \big\}
\]
can be modified inside the box $B_{J(e)}$ into a new configuration $S_e(\omega)$ for which $J(e)$ is now pivotal. Notice that, in any case, the box $B_{J(e)}$ contains the bad edge $e$. 

\begin{lemma}
\label{lem:mapToJedge}
Recall that $e \in \mathcal{E}^{\tiny{\mbox{bad}}}_n$ is defined as above. There exists a measurable map $S_e\colon\Omega^+\to\Omega^+ $ such that, for any $\omega \in \Omega^+$,
\begin{itemize}
\item[(i)] The configurations $\omega$ and $S_e(\omega)$ may only differ on $\emph{In}(B_{J(e)})$.
\item[(ii)] $S_e(\omega)$ is defined on $\emph{In}(B_{J(e)})$ in a deterministic way according to $\omega$ on $\emph{Out}(B_{J(e)})$.
\item[(iii)] If $\omega \in \emph{Piv}^\emph{Out}_{B_{J(e)}}$ and all its $p$-edges in $\text{In}(B_e)$ coincide with those of $S_e(\omega)$ then $\omega \in \{ \emph{$J(e)$ is pivotal} \}$.
\end{itemize}
\end{lemma}

With this technical helper in place, we are ready to prove the last building-block towards our main result. 

\begin{proof}[Proof of Lemma~\ref{lem:badMap}] It actually suffices to prove~\eqref{Bad-J}. We consider the set of configurations coinciding with their modifications inside the box $B_{J(e)}$ but only for $p$-edges:
\[
\mathcal{R}_{S_e} := \big\{ \omega \in \Omega^+ \colon \omega(e') = S_e(\omega)(e') \text{ for any $p$-edge $e'$ in $\text{In}(B_{J(e)})$} \big\}.
\]
Item (iii) of Lemma~\ref{lem:mapToJedge} asserts that $\text{Piv}^\text{Out}_{B_{J(e)}} \cap \mathcal{R}_{S_e} \subset \{ \text{$J(e)$ is pivotal} \}$.

Thanks to Lemma~\ref{lem:mapToJedge}, conditionally to $\mathcal{G}_{\text{Out}(B_{J(e)})}$, the configuration $S_e$ becomes deterministic on $\text{In}(B_{J(e)})$. Hence the fact that a configuration belongs to $\mathcal{R}_{S_e}$ only depends on the states of its $p$-edges in $\text{In}(B_{J(e)})$ which have to be open or closed according to $S_e(\omega)$. We then have the lower bound
\begin{equation}
\label{Minor-R-S}
\mathbb{P}^+\!\big( \mathcal{R}_{S_e} \,|\, \mathcal{G}_{\text{Out}(B_{J(e)})} \big) \geq [p(1-p)]^{71}.
\end{equation}
The rest of the proof now works as in the proof of~\eqref{PivotGood} in Section~\ref{sect:pToq} via the bounds
\begin{eqnarray*}
\mathbb{P}^+\!\big( \text{$J(e)$ is pivotal} \big) & \geq & \mathbb{E} \big[ \mathds{1}_{\text{Piv}^\text{Out}_{B_{J(e)}}} \, \mathbb{P}^+\!\big( \mathcal{R}_{S_e} \,|\, \mathcal{G}_{\text{Out}(B_{J(e)})} \big) \big] \\
& \geq & [p(1-p)]^{71} \, \mathbb{P}^+\!\big( \text{Piv}^\text{Out}_{B_{J(e)}} \big) \\
& \geq & [p(1-p)]^{71} \, \mathbb{P}^+\!\big( \text{$e$ is pivotal} \big),
\end{eqnarray*}
and the proof is finished. 
\end{proof}

\begin{proof}[Proof of Lemma~\ref{lem:mapToJedge}]
Let $e \in \mathcal{E}^{\tiny{\mbox{bad}}}_n$. Let us first focus on the cases (a)-(e) in which, by construction of $J(e)$, the box $B_{J(e)}$ overlaps $\partial \Lambda_n$. Since $\omega \in \text{Piv}^\text{Out}_{B_{J(e)}}$, there exists an admissible path from $o$ to some entry point $x \in \partial B_{J(e)} \cap (\Lambda_n\!\setminus\!\partial\Lambda_n)$ on the one hand and, on the other hand, there is no admissible path from $o$ to $\partial \Lambda_n$ which does not use any edge of $\text{In}(B_{J(e)})$. Proceeding as in Section~\ref{sect:pToq}, we modify the configuration $\omega$ inside of $\text{In}(B_{J(e)})$ into a new configuration $S_e(\omega)$ satisfying the three items of Lemma~\ref{lem:mapToJedge} and in particular $S_e(\omega) \in \{ \text{$J(e)$ is pivotal} \}$. We leave the details to the reader, but the general recipe is again to choose a suitable entry point $x$, connect it to the closest point $x'$ on the cycle and then use the necessary parts of the cycle to create a path to the origin which uses the edge $J(e)$ in a pivotal way, see Figure~\ref{fig:enhancement-bad-schematic-boundary} for a schematic illustration\footnotemark. \footnotetext{Again, the interested reader will be able to find a visualization of the constructions in all possible cases here: \url{https://bennhenry.github.io/NeighPerc/}.}

\begin{figure}[!ht]
\centering
\includegraphics[width=.65\textwidth]{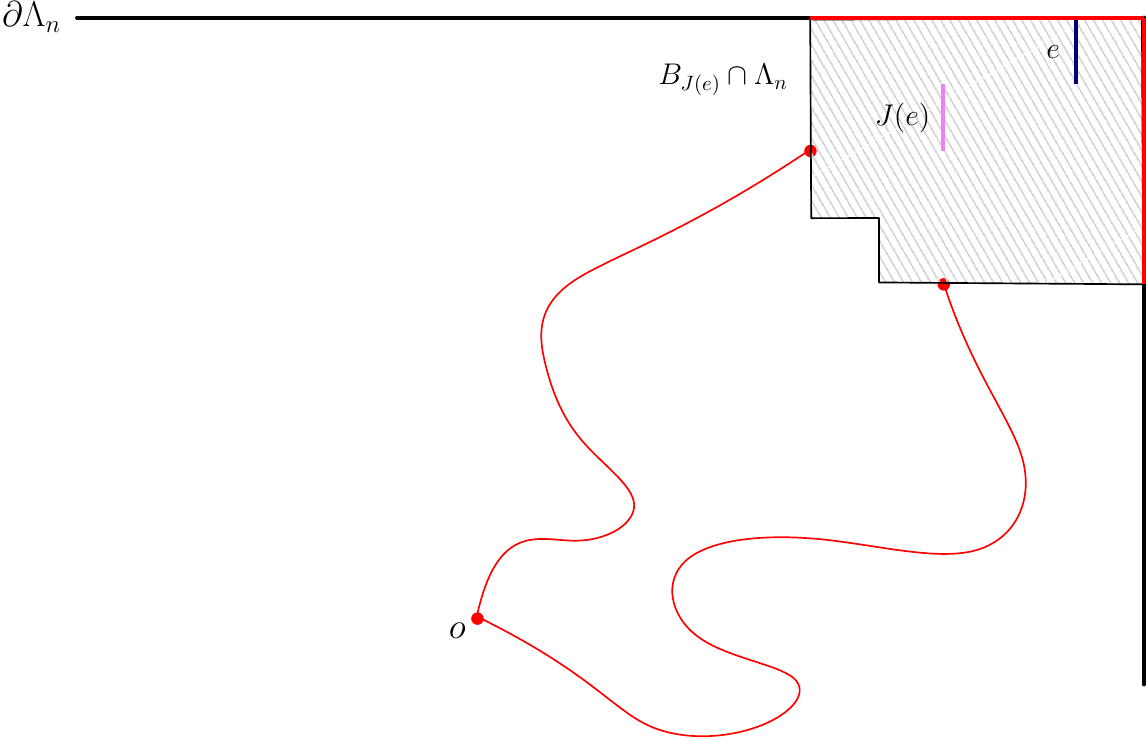}
\caption{\label{fig:enhancement-bad-schematic-boundary} In the case where the edge $e$ (blue) is bad because it is too close to the boundary $\partial \Lambda_n$, we map it to a good edge $J(e)$ (pink) which is slightly further inside of $\Lambda_n$ but still such that $e \in B_{J(e)}$. If $e$ is pivotal for a configuration $\omega$, then there must be at least one admissible path (in red) from the origin to the boundary of $B_{J(e)}$, so we can again choose an entry point and then connect it to the part of boundary $\partial\Lambda_n$ which intersects $B_{J(e)}$ via a path that uses $J(e)$ in a pivotal way. For this, we only need to change the status of edges inside of the shaded area, i.e., in $B_{J(e)}\cap \Lambda_n$.}
\end{figure}

In the last case (f), the box $B_{J(e)}$ contains the origin $o$ and we therefore do only need to choose an exit point but not an entry point. There are three different subcases regarding the precise position of $o$ inside of $B_{J(e)}$. Either $J(e)$ is directly connected to $o$, or $o$ lies on the cycle $C$, or it lies inside of the cycle but is not touched by $J(e)$. In all of these one can construct a path from the origin to the specified exit point which uses the edge $J(e)$ in a pivotal way by only opening parts of the cycle $C$, the edge $J(e)$ plus one to three additional edges, depending on the position of $o$ in $B_{J(e)}$. We again leave the details to the reader and refer to Figure~\ref{fig:enhancement-schematic-bad-origin} for an instructive example\footnotemark. 

\begin{figure}[!ht]
\centering
\includegraphics[width=.75\textwidth]{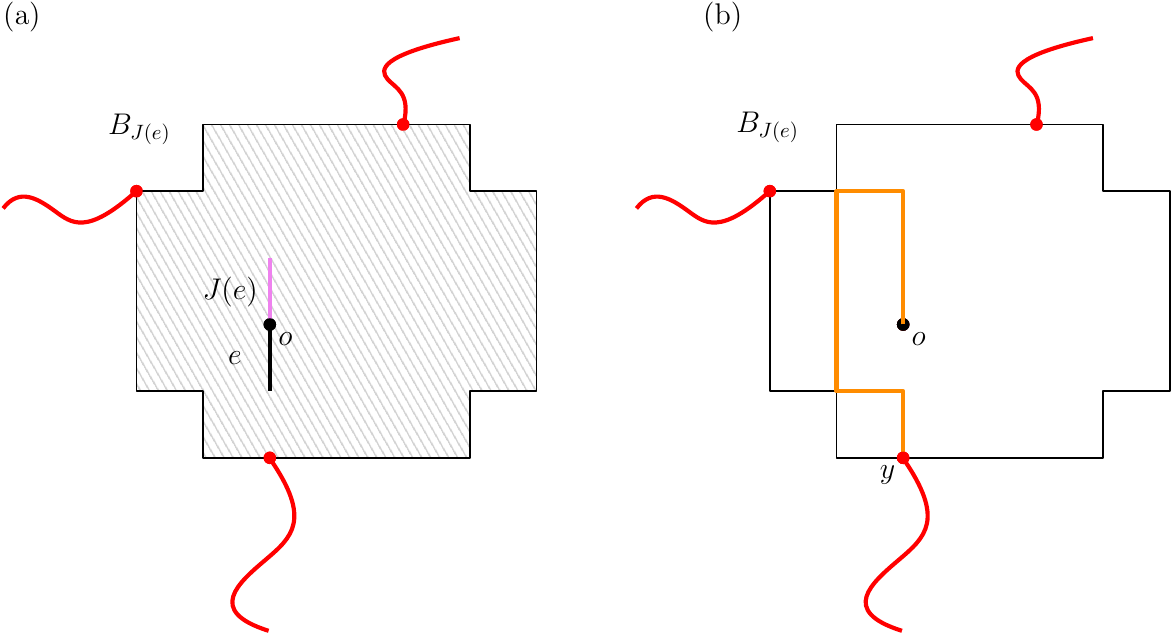}
\caption{\label{fig:enhancement-schematic-bad-origin} (a) In the case where the edge $e$ (black) is bad because the associated pattern contains the origin $o$ in its interior, we map it to a good edge $J(e)$ (pink) nearby such that this is not the case. If $e$ is pivotal for a configuration $\omega$, then there must be at least one admissible path (red) from $\partial B_{J(e)}$ to $\partial \Lambda_n$. (b) Therefore,  we can again choose an exit point $y$ and then connect it to the origin  via an admissible path (orange) in $B_{J(e)}$ that uses the edge $J(e)$ in a pivotal way. All other edges in $B_{J(e)}$ are closed.}
\end{figure}
\end{proof}

\section{Proof of Theorem~\ref{thm_upperbound_cor}}
\label{sec_upperbound_cor}

Let us denote by $\mathbb{P}^{\text{corn}}_p$ the probability measure on $\Omega=\{0,1\}^\mathcal{E}$ corresponding to the directed-corner model with parameter $p$. When $p=1/2$, each vertex chooses independently exactly one of the four combinations of open edges: north-east, north-west, south-east and south-west, each of them with probability $1/4$, see Figure~\ref{fig:corners} for an illustration. 

\begin{figure}[ht]
\centering
\includegraphics[width=.6\textwidth]{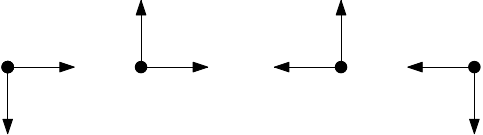}
\caption{\label{fig:corners} The four different configurations of outgoing edges a vertex can choose in the directed corner model with parameter $p=1/2$.}
\end{figure}

The proof follows step by step the one of Theorem~\ref{thm:upperBound} and the only major difference arises when one wants to get good probabilistic bounds for the number of pivotal edges that one discovers in the exploration algorithm, see Lemma~\ref{lem:NbPivotCorner} below. But let us provide some more details. 

Our strategy is again to prove that the size of the dual forward set $\df{o^\ast}{}$ admits subexponential decay under $\mathbb{P}^{\text{corn}}_{1/2}$, as in Theorem~\ref{thm:clusterSize}. Indeed, this prevents the occurrence of long paths in the dual model and allows us to conclude via the same block-renormalization argument used in Section~\ref{sect:proofTheoUB}, that the directed-corner model still percolates for some parameter $p<1/2$, thus proving that $p_c^{\text{corn}}<1/2$. To do this, we will again explore the forward set $\df{o^\ast}{}$ with the same exploration algorithm as in Section~\ref{sect:ExploAlgo}. During the exploration process, some pivotal edges (defined as in Section~\ref{sect:PivotalAlgo}, see Figure~\ref{fig:PivotalEdge}) will appear, so that the explored cluster $\edf{x^\ast}$ can be split, as before, into disjoint clusters of vertices visited during the algorithm linked by pivotal edges.
For this, note that Lemma~\ref{lem:pivot} still holds here, i.e., when the algorithm reveals a pivotal edge, it is necessarily the last edge to be explored at this stage, i.e., if it is closed, the algorithm stops. So until this point the proof works exactly like for Theorem~\ref{thm:upperBound} and it  remains to show that 
\begin{enumerate}[\bfseries (S1)]
    \item the random number of pivotal edges has sufficiently light tails (as in Lemma~\ref{lem:SubGeom}) and 
    \item the size of each visited cluster $\kedf{o^\ast}{k}$ separating two consecutive pivotal edges is also sufficiently small (as in Theorem~\ref{theo:subGeomClusterSize}).
\end{enumerate}
Then, putting together the steps $\mathbf{(S1)}$ and $\mathbf{(S2)}$, as in the proof of Theorem~\ref{thm:clusterSize} at the end of Section~\ref{sect:VisitedCluster}, we get our conclusion. 

\medskip

Let us first focus on the size of the explored clusters $\kedf{o^\ast}{k}$, $k\geq 1$, i.e., on step $\mathbf{(S2)}$. The important remark is that each (non-pivotal) revealed edge $e^\ast$ during the exploration of $\kedf{o^\ast}{k}$ has probability at most $1/2$ to be open conditionally to what has been already explored. By symmetry, let us assume that the dual directed edge $e^\ast=(x^\ast,y^\ast)$ is the east side of a unit square centered at the primal vertex $z$ and let $e^\ast_N$, $e^\ast_W$ and $e^\ast_S$ be the north, west and south dual directed edges associated to $z$. The states of these four edges are then dependent. Recall that Rule~1 of the exploration algorithm implies that $e^\ast_N$ has not been explored yet. So only $e^\ast_W$ and $e^\ast_S$ can have been  explored before and, in that case, are necessarily open (otherwise $e^\ast$ would be pivotal but there is no pivotal edge in $\kedf{o^\ast}{k}$ by construction). If neither of the two edges $e^\ast_W$ and $e^\ast_S$ have been explored then
\[
\mathbb{P}^{\text{corn}}_{1/2} \Big( \scalebox{1.1}{\dualopen} \Big| \scalebox{1.1}{\dualempty} \Big) = \mathbb{P}^{\text{corn}}_{1/2} \Big( \scalebox{1.1}{\dualopen} \Big) = 1/2. 
\]
If only $e^\ast_W$ has been explored (and is open) then the two possible combinations for open outgoing edges starting at the primal vertex $z$ are north-east and south-east. In both cases, the dual edge $e^\ast$ will be closed:
\[
\mathbb{P}^{\text{corn}}_{1/2} \Big(\begin{tikzpicture}[baseline=0.6em,scale=0.65]
			\node[inner sep=0, outer sep=0] at (-0.2,0) {\hspace{2pt}};
			\node[inner sep=0, outer sep=0] at (1.2,0) {\hspace{2pt}};
			\draw (0,0) -- (0,1) -- (1,1) -- (1,0) -- (0,0);
			\draw[fill=red, color=red] (0.5,0.5) circle  (0.075);
			\draw[color=red, very thick ,opacity= 0.4] (0.5,0.5) -- (1,0.5);
			\draw[color=red, very thick ,opacity= 0.4] (0.5,0.5) -- (0,0.5);
			\draw[->,blue,very thick] (1,0) -- (1,1); 
			\path[pattern=north west lines, pattern color=blue] (0.9,0) rectangle (1,1);
			\begin{scope}[rotate around={180:(0.5,0.5)}]
				\draw[->,blue,very thick] (1,0) -- (1,1); 
				\path[pattern=north west lines, pattern color=blue] (0.9,0) rectangle (1,1);
			\end{scope}
		\end{tikzpicture}
  \Big| 
\begin{tikzpicture}[baseline=0.6em,scale=0.65]
			\node[inner sep=0, outer sep=0] at (-0.2,0) {\hspace{2pt}};
			\node[inner sep=0, outer sep=0] at (1.2,0) {\hspace{2pt}};
			\draw (0,0) -- (0,1) -- (1,1) -- (1,0) -- (0,0);
			\draw[fill=red, color=red] (0.5,0.5) circle  (0.075);
			\draw[color=red, very thick ,opacity= 0.4] (0.5,0.5) -- (0,0.5);
			\begin{scope}[rotate around={180:(0.5,0.5)}]
				\draw[->,blue,very thick] (1,0) -- (1,1); 
				\path[pattern=north west lines, pattern color=blue] (0.9,0) rectangle (1,1);
			\end{scope}
		\end{tikzpicture}
\Big) = 0.
\]
Other cases are treated similarly:
\[
\mathbb{P}^{\text{corn}}_{1/2} \Big(\begin{tikzpicture}[baseline=0.6em,scale=0.65]
			\node[inner sep=0, outer sep=0] at (-0.2,0) {\hspace{2pt}};
			\node[inner sep=0, outer sep=0] at (1.2,0) {\hspace{2pt}};
			\draw (0,0) -- (0,1) -- (1,1) -- (1,0) -- (0,0);
			\draw[fill=red, color=red] (0.5,0.5) circle  (0.075);fill=red
			\draw[color=red, very thick ,opacity= 0.4] (0.5,0.5) -- (1,0.5);
			\draw[color=red, very thick ,opacity= 0.4] (0.5,0.5) -- (0.5,0);
			\draw[->,blue,very thick] (1,0) -- (1,1); 
			\path[pattern=north west lines, pattern color=blue] (0.9,0) rectangle (1,1);
			\begin{scope}[rotate around={270:(0.5,0.5)}]
				\draw[->,blue,very thick] (1,0) -- (1,1); 
				\path[pattern=north west lines, pattern color=blue] (0.9,0) rectangle (1,1);
			\end{scope}
		\end{tikzpicture}
  \Big| 
  \begin{tikzpicture}[baseline=0.6em,scale=0.65]
			\node[inner sep=0, outer sep=0] at (-0.2,0) {\hspace{2pt}};
			\node[inner sep=0, outer sep=0] at (1.2,0) {\hspace{2pt}};
			\draw (0,0) -- (0,1) -- (1,1) -- (1,0) -- (0,0);
			\draw[fill=red, color=red] (0.5,0.5) circle  (0.075);fill=red
			\draw[color=red, very thick ,opacity= 0.4] (0.5,0.5) -- (0.5,0);
			\begin{scope}[rotate around={270:(0.5,0.5)}]
				\draw[->,blue,very thick] (1,0) -- (1,1); 
				\path[pattern=north west lines, pattern color=blue] (0.9,0) rectangle (1,1);
			\end{scope}
		\end{tikzpicture}
  \Big) = 1/2 \; \mbox{ and } \; \mathbb{P}^{\text{corn}}_{1/2} \Big( \scalebox{1.1}{\dualopenthree} \Big| \scalebox{1.1}{\dualopentwo} \Big) = 0.
\]
Moreover, the left winding pattern \scalebox{0.3}{\begin{tikzpicture}
    \draw[step=1cm, gray,very thin] (-1,0) grid (2,1);
    \draw[blue, line width=1mm, -latex] (0,1) -- (0,0);
    \draw[blue, line width=1mm, -latex] (0,0) -- (1,0);
    \draw[blue, line width=1mm, -latex] (1,0) -- (1,1);
\end{tikzpicture}} cannot be visited during the exploration process. Henceforth, proceeding as in Section~\ref{sec:UnderConstraints}, we can again stochastically dominate $\kedf{o^\ast}{k}$ by the cluster of the origin in a (undirected and independent) Bernoulli percolation model with parameter $1/2$ under the constraint that the pattern \scalebox{0.7}{\pattern} is forbidden to be used by a percolating path. Thus applying Theorem~\ref{prop:enhancement}, we then obtain an analogue of Theorem~\ref{theo:subGeomClusterSize}: under the measure $\mathbb{P}^{\text{corn}}_{1/2}$ the probability of the event that $|\kedf{o^\ast}{k}|$ is larger than $n$ decreases exponentially fast with $n$. So we are done with step $\mathbf{(S2)}$.

\medskip
Step $\mathbf{(S1)}$ follows from the following lemma. 
\begin{lemma}
\label{lem:NbPivotCorner}
The number $\mathcal{T}_{\text{piv}}$ of open pivotal edges revealed by the exploration algorithm satisfies the following inequality: for any $n\geq 1$,
\[
\mathbb{P}^{\text{corn}}_{1/2} \big( \mathcal{T}_{\text{piv}} \geq n \big) \leq C e^{-c n},
\]
where $c,C>0$ are constants that do not depend on $n$. 
\end{lemma}

For the rest of the proof let us denote by $e^\ast=(x^\ast,y^\ast)$ the edge revealed at the $n$-th step of the exploration process of $\df{o^\ast}{}$ and assume $e^\ast$ is pivotal. By symmetry, we can assume that $e^\ast$ is the east side of a unit square centered at a primal vertex $z$ (we also use the notations $e^\ast_N$, $e^\ast_W$ and $e^\ast_S$ previously introduced). We denote by $\mathcal{F}_n$ the $\sigma$-algebra gathering all the information generated by the exploration process until step $n$. Unlike in the proof of Theorem~\ref{thm:upperBound}, the revealed pivotal edge $e^\ast$ may be open with probability $1$ conditionally to $\mathcal{F}_n$. We will say that such a pivotal edge is {\em automatically open}, or use the shorter expression {\em auto-open}. The property of an edge $e^\ast$ to be automatically open or not is measurable with respect to $\mathcal{F}_n$: 

\begin{lemma}
\label{lem:RkAutoOpen}
In the previously used notation, the pivotal edge $e^\ast$ is automatically open if and only if $e^\ast_W$ has been already explored and is closed. Besides,
\begin{equation}
\label{NotAutoOpen}
\mathbb{P}^{\text{corn}}_{1/2} \big( \mbox{$e^\ast$ is open} \, |\, \mathcal{F}_n \big) \mathds{1}\{\mbox{$e^\ast$ is not auto-open}\} \leq 1/2.
\end{equation}
\end{lemma}

\begin{proof}[Proof of Lemma~\ref{lem:RkAutoOpen}] If $e^\ast_W$ has already been explored and is closed, then the two possible combinations for open outgoing edges starting at the primal vertex $z$ are north-west and south-west. In both cases, the dual edge $e^\ast$ is open with probability one, i.e., it is automatically open. Otherwise, either $e^\ast_W$ has already been explored and is open, or $e^\ast_W$ has not been explored yet. In the first case, the two possible combinations for open outgoing edges starting at $z$ are north-east and south-east: $e^\ast$ is open with zero probability. In the second case, the edge $e^\ast_S$ has already been explored and is closed, since $e^\ast$ is pivotal. So the two possible combinations for open outgoing edges starting at $z$ are south-west and south-east. In both cases, $e^\ast$ is open with probability $1/2$. Hence, these two later cases are the only two ways for the pivotal edge $e^\ast$ to not be automatically open and we get~\eqref{NotAutoOpen}.
\end{proof}

\begin{proof}[Proof of Lemma~\ref{lem:NbPivotCorner}] Let us denote by $\mathcal{T}_{\text{piv}}'$ (respectively $\mathcal{T}_{\text{piv}}''$) the number of open pivotal edges revealed by the exploration process of $\df{o^\ast}{}$ that are automatically open (respectively that are not), so that $\mathcal{T}_{\text{piv}} = \mathcal{T}_{\text{piv}}'+\mathcal{T}_{\text{piv}}''$. Thanks to \eqref{NotAutoOpen}, we prove that the probability $\mathbb{P}^{\text{corn}}(\mathcal{T}_{\text{piv}}'' \geq n)$ is at most $2^{-n}$ (proceeding as in the proof of Lemma~\ref{lem:SubGeom}).

It then remains to get the same kind of inequality but for $\mathcal{T}_{\text{piv}}'$ to obtain Lemma~\ref{lem:NbPivotCorner}. To do this, we use the fact that, since any automatically open pivotal edge is in particular pivotal, it is still the last way for the exploration process to continue. Hence, even though it cannot immediately stop our exploration process because it is always open, we will show that the exploration process has a decent chance to stop a few steps after meeting an automatically open pivotal edge.

So we focus on the case where $e^\ast=(x^\ast,y^\ast)$ is automatically open, which means that $e^\ast_W$ has already been explored and is closed by Lemma~\ref{lem:RkAutoOpen}. The vertex $u^\ast$, defined as the starting point of $e^\ast_W := (u^\ast,\cdot)$, has then been visited before step $n$. So does the edge $(u^\ast,y^\ast)$ (just after $e^\ast_W$), since the algorithm proceeds in the depth-first fashion and in the counter-clockwise sense: it is necessarily closed by Rule~1. Let us make the following {\em tricky remark.} It is crucial to point out that the edge $(y^\ast,v^\ast)$, where $v^\ast := y^\ast+(0,1)$, in the case where it will be revealed, will be a pivotal edge but not an automatically open one. Indeed, $(y^\ast,v^\ast)$ is pivotal since $(u^\ast,y^\ast)$ has been already explored and is closed. The fact that $(y^\ast,v^\ast)$ is automatically open would mean, by Lemma~\ref{lem:RkAutoOpen}, that $(w^\ast,u^\ast)$, where $w^\ast := u^\ast+(0,1)$, has been already explored and is closed, leading to a situation which is forbidden by planarity (refer to Case~2 in the proof of Lemma~\ref{lem:pivot}). In conclusion, the edge $(y^\ast,v^\ast)$, in the case where it will be revealed, has a probability at least $1/2$ to be closed. See Figure~\ref{fig:MostFavorable} for an illustration.

At the $(n+1)$-th step of the algorithm, the edge $(y^\ast,a^\ast)$, with $a^\ast := y^\ast+(1,0)$, is revealed and several situations may occur according to $\mathcal{F}_n$. In the sequel, we only investigate two cases: the most favorable one, i.e., in order to quickly stop the exploration process, in which $(y^\ast, a^\ast)$ is not an automatically open pivotal edge and in the opposite situation, the less favorable case, see Figure~\ref{fig:LessFavorable}. We leave the verification of the intermediate cases to the reader.

\medskip

\textbf{The most favorable case:} The edge $(y^\ast,a^\ast)$ is not an automatically open pivotal edge. Either it is pivotal (but not automatically open) and has probability at least $1/2$ to be closed by Lemma~\ref{lem:RkAutoOpen}. Or it is not pivotal and has also probability at least $1/2$ to be closed. In both cases, the edge $(y^\ast,a^\ast)$ revealed at the $(n+1)$-th step is closed with probability at least $1/2$ (refer to the previous computations). Assume such situation occurs. Then the last chance for the algorithm to continue its exploration is the edge $(y^\ast,v^\ast)$. Either $v^\ast$ has been already explored and the algorithm stops by Rule~1. Or $(y^\ast,v^\ast)$ is revealed and is closed with probability at least $1/2$ by the previous analysis. In conclusion, in any case, the algorithm process does not visit any extra vertex after $y^\ast$ with probability at least $(1/2)^2$ by independence.

\begin{figure}[!ht]
\begin{center}
\includegraphics[width=10cm,height=4cm]{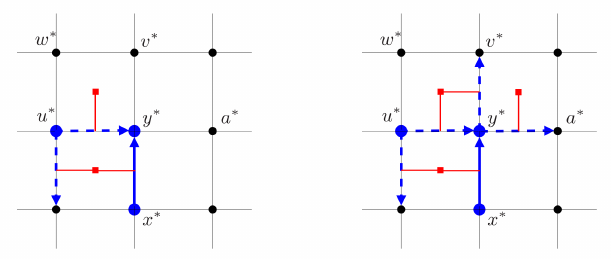}
\caption{\label{fig:MostFavorable}Dual vertices visited by the exploration process are blue points. Open and revealed edges are depicted with blue lines while closed and revealed edges are depicted with blue dotted lines. By red lines, we indicate the primal vertices at which the states of dual edges depend. {\em Left:} the situation at the end of the $n$-th step, the edge $e^\ast=(x^\ast,y^\ast)$ has just been revealed. It is an automatically open pivotal edge. {\em Right:} the most favorable case. Both edges $(y^\ast, a^\ast)$ thus $(y^\ast, v^\ast)$ are closed with probability at least $1/4$. The exploration algorithm then stops.}
\end{center}
\end{figure}

\medskip

\textbf{The less favorable case:} The edges $(y^\ast,a^\ast)$ and $(a^\ast,c^\ast)$ are automatically open pivotal edges, where $c^\ast := a^\ast - (0,1)$. See Figure~\ref{fig:LessFavorable}. This hypothesis admits several consequences. First, the edge $(b^\ast,v^\ast)$ has been already explored and is closed (this justifies that $(y^\ast,a^\ast)$ is an automatically open pivotal edge). Hence, the vertex $b^\ast$ has been already visited, like the edge $(b^\ast,a^\ast)$ which is closed by Rule~1. The same holds for the automatically open pivotal edge $(a^\ast,c^\ast)$: $j^\ast$ has been already visited, both edges $(j^\ast,i^\ast)$ and $(j^\ast,c^\ast)$ are explored too, and closed. The left part of Figure~\ref{fig:LessFavorable} represents the corresponding situation. Now, the important remark is that, at the end of step $n+2$, all the edges that are to be explored (i.e., in the list $L^{(n+3)}$) can be each closed with positive probability. By Rule 1, $(c^\ast,x^\ast)$, $(c^\ast,j^\ast)$ and $(a^\ast,b^\ast)$ will not be explored. So, the list $L^{(n+3)}$ only contains the edges $(c^\ast,d^\ast)$, $(a^\ast,i^\ast)$ and $(y^\ast,v^\ast)$ (in this particular order) which are all pivotal but not automatically open, thanks to the tricky remark (see above). By Lemma~\ref{lem:RkAutoOpen}, each of them has probability at least $1/2$ to be closed and, by independence, the algorithm stops after the visit of $c^\ast$ with probability at least $(1/2)^3$.

\begin{figure}[!ht]
\begin{center}
\includegraphics[width=12cm,height=5.5cm]{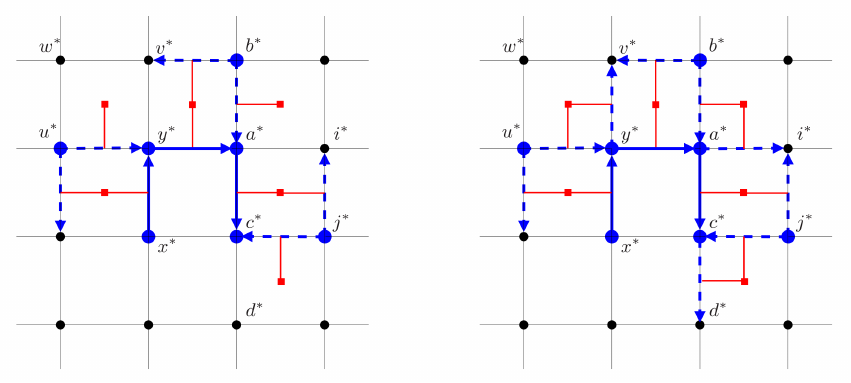}
\caption{\label{fig:LessFavorable}The color code is the same as in Figure~\ref{fig:MostFavorable}. {\em Left:} the situation after the revealment of the three automatically open pivotal edges that are $(x^\ast,y^\ast)$, $(y^\ast,a^\ast)$ and $(a^\ast,c^\ast)$. {\em Right:} the three edges remaining to be explored, namely $(c^\ast,d^\ast)$, $(a^\ast,i^\ast)$ and $(y^\ast,v^\ast)$, are all closed with probability at least $1/8$. The algorithm then stops.}
\end{center}
\end{figure}

Let us now conclude. Proceeding as in the proof of Lemma~\ref{lem:SubGeom}, we can state that the probability $\mathbb{P}^{\text{corn}}_{1/2}(\mathcal{T}_{\text{piv}}' \geq n)$ is at most $(1 - 1/8)^{n/3}$. Indeed, for each automatically open pivotal edge $e^\ast$, with probability at least $1/8$, the exploration process will visit at most two extra automatically open pivotal edges after $e^\ast$. The power $n/3$ means that we apply our argument to one automatically open pivotal edge in three.
\end{proof}

\section*{Appendix}
\label{sec_prelim}

This appendix is devoted to the proofs of Lemmas~\ref{lem:monotoneP_kappa}, \ref{lem_low1st}, \ref{lem:Critical-iid} and~\ref{lem:Critical-0-4}.

\begin{proof}[Proof of Lemma~\ref{lem:monotoneP_kappa}]
Let $p \in [0,1]$ and recall that $k=\lfloor 2dp\rfloor$ and $\varepsilon=2dp-k$. Recall also that the probability measure $\PP_p$ is defined on the configuration set $\Omega$ as follows:
\begin{itemize}
\item Pick $k$ directed edges among the $2d$ outgoing edges from $x$ and,
\item with probability $\varepsilon$, pick a $(k+1)$-th directed edge among the $2d - k$ remaining ones.
\end{itemize}

For any vertex $x$, we denote by $(e_i(x))_{1\leq i\leq 2d}$ the $2d$ outgoing edges from $x$. Consider a random permutation of these edges, denoted by $(e_{\sigma_i}(x))_{1\leq i\leq 2d}$, where $\sigma = \sigma^x = (\sigma_i)_{1\leq i\leq 2d}$ is chosen uniformly among the permutations of $(1,\dots, 2d)$. Moreover, the random permutations $\{\sigma^x\}_{x\in\Z^{d}}$ are chosen independently from each other. We also consider a sequence $\{U_x\}_{x\in\Z^{d}}$ of i.i.d.~random variables whose common distribution is the uniform law on $[0,1]$, independent from the $\{\sigma^x\}_{x\in\Z^{d}}$ too. The collection $\{\sigma^x,U_x\}_{x\in\Z^{d}}$ constitutes the random input of our construction. 

Let us build a random variable $X_p$ valued in $\Omega = \{0,1\}^{\E}$ as follows. For any vertex $x$, we declare open the edges $e_{\sigma_1}(x),\ldots,e_{\sigma_k}(x)$ and also the edge $e_{\sigma_{k+1}}(x)$ only if $U_x < \varepsilon$. Now, for any $e \in \E$, we set $X_p(e) = 1$ if and only if $e$ has been declared open. Otherwise $X_p(e) = 0$. Hence, $X_p$ is distributed according to $\PP_p$. Indeed, for any vertex $x$, the edges $e_{\sigma_1}(x),\ldots,e_{\sigma_{k+1}}(x)$ are chosen uniformly among $(e_i(x))_{1\leq i\leq 2d}$ and $e_{\sigma_{k+1}}(x)$ is selected uniformly among the remaining edges and it is open with probability $\varepsilon$.

Henceforth, the previous construction provides a monotone coupling as $X_p \prec X_{p'}$ whenever $p \leq p'$ where $\prec$ denotes the partial order on $\Omega$. If $k < k'$ then by construction all the edges that are open for $X_p$ are also open for $X_{p'}$. If $k = k'$ then, for any vertex $x$, $U_x < \varepsilon$ implies $U_x < \varepsilon'$ since $\varepsilon = p-k$ and $\varepsilon' = p'-k'$, meaning that, if the edge $e_{\sigma_{k+1}}(x)$ is open for $X_p$, it is also open for $X_{p'}$. Consequently
\[
\PP_p(o \rightsquigarrow \infty) = P(X_p \in \{o\rightsquigarrow\infty\}) \leq P(X_{p'} \in \{o\rightsquigarrow\infty\}) = \PP_{p'}(o \rightsquigarrow \infty),
\]
as desired. 
\end{proof}

\begin{proof}[Proof of Lemma~\ref{lem_low1st}]
For all $n\geq 0$, le us write
\begin{align*}
\theta_p = \PP_p(o \rightsquigarrow \infty) \leq \mathbb{P}_p \big( \exists \text{ self-avoiding open path of length $n$ starting at the origin} \big)
\end{align*}
and hence, using the union bound, 
\begin{equation}
\label{UnionBound:c_n}
\theta_p \leq \limsup_{n \uparrow \infty} c_n(d) p^n,
\end{equation}
where $c_n(d)$ denotes the number of self-avoiding paths of length $n$ in the $d$-dimensional hypercubic lattice that start at the origin. Above, it is also used that a self-avoiding path visits a vertex only once and hence independence applies. So, the right-hand side in \eqref{UnionBound:c_n} equals zero whenever $c(d) p < 1$ where $c(d) = \lim_{n\uparrow\infty} c_n(d)^{1/n}$ is the connective constant. This implies the desired result.
\end{proof}

\begin{proof}[Proof of Lemma~\ref{lem:Critical-iid}]
We only consider the case $d=2$ since the other cases can be proved similarly. Recall that $\mathbb{P}^{\text{iid}}_p$ is the probability distribution on the configuration set $\Omega = \{0,1\}^\E$, where $\E = \{(x,y)\colon \|x-y\|_{\ell^{1}} = 1\}$ is the set of {\em directed} edges of $\mathbb Z^2$, in which each directed edge is independently open with probability $p$. Also, $\mathcal{P}_p$ denotes the probability distribution on $\Omega' = \{0,1\}^{\E'}$, where $\E' = \{\{x,y\}\colon \|x-y\|_{\ell^{1}} = 1\}$ is the set of {\em undirected} edges of $\mathbb Z^2$, in which each undirected edge is independently open with probability $p$. On a common probability space, we aim to build two variables $X_p$ and $Y_p$ respectively valued in $\Omega$ and $\Omega'$ such that $X_p \sim \mathbb{P}^{\text{iid}}_p$, $Y_p \sim \mathcal{P}_p$ and
\[
\mathbb{P}^{\text{iid}}_p(o \rightsquigarrow \infty) = P(X_p \in \{o \rightsquigarrow \infty\}) = P(Y_p \in \{o \rightsquigarrow \infty\}) = \mathcal{P}_p(o \rightsquigarrow \infty).
\]
With a slight abuse of notations, we still denote by $o \rightsquigarrow \infty$ for the existence of an infinite self-avoiding open path of undirected edges starting at $o$. Then, using well-known percolation results for $\mathcal{P}_p$ in $\Z^2$,the statement of Lemma~\ref{lem:Critical-iid} follows directly.

In order to define the coupling, let us consider a family $\{U_e\colon e \in \E\}$ of i.i.d.~random variables uniformly distributed on $[0,1]$. For any directed edge $e$, we set $X_p(e) = 1$ if and only if $U_e < p$. The random variable $X_p$, valued in $\Omega$, is distributed according to $\mathbb{P}^{\text{iid}}_p$.

Now, let us explore the forward set $\df{o}{}$ associated to $X_p$ while defining step by step a variable $Y_p$ valued in $\Omega'$. This exploration process is performed in breadth-first fashion and in the counter-clockwise sense. At the beginning, we set $V=\{o\}$ and $L$ as the list made up with the four directed edges starting at $o$. The list $L$ will play the role of directed edges to be (possibly) revealed. While $L$ is non empty, we proceed as follows. Take $(x,y)$ be the first element of $L$:
\begin{itemize}
\item If $(y,x)$ has been already revealed then we do not explore $(x,y)$. The vertex $y$ has been already visited: returning to $y$ will not show anything new about the exploration of $\df{o}{}$. We delete $(x,y)$ from the list $L$.
\item If $(y,x)$ has not been revealed yet, we reveal $(x,y)$:
\begin{itemize}
\item If $X_p((x,y)) = 0$ (closed) then we set $Y_p(\{x,y\}) = 0$ and delete $(x,y)$ from the list $L$.
\item If $X_p((x,y)) = 1$ (open) then we set $Y_p(\{x,y\}) = 1$, add the vertex $y$ to the set $V$ and add the four directed edges starting at $y$ to the list $L$ (at the beginning of the list). We delete $(x,y)$ from the list $L$.
\end{itemize}
\end{itemize}

Note that, if the edge $(x,y)$ has been revealed during the exploration process, then $(y,x)$ has not been, so that the value $Y_p(\{x,y\})$ is defined without ambiguity. If neither of the edges $(x,y)$ and $(y,x)$ have been revealed during the exploration process, we merely set $Y_p(\{x,y\})$ as being equal to $X_p((x,y))$ or $X_p((y,x))$ according to the lexicographic order between $x$ and $y$. Henceforth the random variable $Y_p$ is defined on the whole set $\E'$. Note also that, for any edge $\{x,y\}$, exactly one of the uniform random variables $U_{(x,y)}$ and $U_{(y,x)}$ has been used to define $Y_p(\{x,y\})$ and this choice may depend on the other $U_e$'s (through the exploration process) but not on $U_{(x,y)}$ and $U_{(y,x)}$ themselves. For this reason $Y_p$ is distributed according to $\mathcal{P}_p$.

Since the exploration process is performed in breadth-first fashion, the whole set $\df{o}{}$ is explored by the exploration process (whether it is finite or not). Hence our coupling means that the forward set $\df{o}{}$ (for $X_p$) is equal to the cluster of the origin (for $Y_p$)-- made up with vertices that can be reached from $o$ by an open undirected path. As a consequence, $P(X_p \in \{o \rightsquigarrow \infty\}) = P(Y_p \in \{o \rightsquigarrow \infty\})$ which finishes the proof.
\end{proof}

\begin{proof}[Proof of Lemma~\ref{lem:Critical-0-4}]
We trivially couple the all-or-none directed percolation model with i.i.d.\ Bernoulli site percolation by saying that a vertex is open if and only if the vertex connects (via directed edges) to all its four neighbors. In particular, if $p>p_c^{\text{site}}$, then (by ergodicity of the model) the origin is connected to infinity via a self-avoiding path of open vertices with positive probability. But this implies that also the origin is connected to infinity via a directed path in the all-or-none model with positive probability. Hence, $p_c^{\text{aon}}\le p_c^{\text{site}}$. But the same also holds the other way around, which concludes the proof. 
\end{proof}

\section*{Acknowledgements}

BJ and JK acknowledge the financial support of the Leibniz Association within the Leibniz Junior Research Group on {\em Probabilistic Methods for Dynamic Communication Networks} as part of the Leibniz Competition. DC and BH are supported in part by the CNRS RT 2179 MAIAGES, DC is supported by the ANR project GrHyDy (ANR-20-CE40-0002). DC and BH thank J.-B. Gou\'er\'e and V. Tassion for stimulating discussions and valuable advice.

\bibliographystyle{amsalpha}
\bibliography{bib}

\newcommand{\etalchar}[1]{$^{#1}$}
\providecommand{\bysame}{\leavevmode\hbox to3em{\hrulefill}\thinspace}
\providecommand{\MR}{\relax\ifhmode\unskip\space\fi MR }
\providecommand{\MRhref}[2]{%
  \href{http://www.ams.org/mathscinet-getitem?mr=#1}{#2}
}
\providecommand{\href}[2]{#2}
\begin{thebibliography}{dLSdS{\etalchar{+}}20}

\bibitem[AG91]{aizenman1991strict}
M.~Aizenman and G.~Grimmett, \emph{Strict monotonicity for critical points in
  percolation and ferromagnetic models}, J. Stat. Phys. \textbf{63} (1991),
  817--835.

\bibitem[BB13]{balister2013percolation}
P.~Balister and B.~Bollob{\'a}s, \emph{Percolation in the $k$-nearest neighbor
  graph}, Recent results in designs and graphs: A tribute to Lucia Gionfriddo,
  Quaderni di Matematica \textbf{28} (2013), 83--100.

\bibitem[BBR14]{balister2014essential}
P.~Balister, B.~Bollob{\'a}s, and O.~Riordan, \emph{Essential enhancements
  revisited}, arXiv preprint arXiv:1402.0834 (2014).

\bibitem[Bee21]{beekenkamp2021sharpness}
Th. Beekenkamp, \emph{Sharpness of the phase transition for the orthant model},
  Math. Phys. Anal. Geom. \textbf{24} (2021), no.~4, 36.

\bibitem[BHH24]{beaton2024chemical}
N.~Beaton, M.~Holmes, and X.~Huang, \emph{Chemical distance for the
  half-orthant model}, arXiv preprint arXiv:2401.03647 (2024).

\bibitem[BR06]{bollobas2006percolation}
B.~Bollob{\'a}s and O.~Riordan, \emph{Percolation}, Cambridge University Press,
  2006.

\bibitem[dLMSV22]{de2022approximation}
B.~de~Lima, S.~Martineau, H.~Sanna, and D.~Valesin, \emph{Approximation on
  slabs and uniqueness for {B}ernoulli percolation with a sublattice of
  defects}, ALEA Lat. Am. J. Probab. Math. Stat. \textbf{19} (2022), no.~2,
  1767--1797.

\bibitem[dLSdS{\etalchar{+}}20]{de2020constrained}
B.~de~Lima, R.~Sanchis, D.~dos Santos, V.~Sidoravicius, and R.~Teodoro,
  \emph{The constrained-degree percolation model}, Stoch. Process. Their Appl.
  \textbf{130} (2020), no.~9, 5492--5509.

\bibitem[FIV13]{10.1214/11-AOP720}
S.~Friedli, D.~Ioffe, and Y.~Velenik, \emph{Subcritical percolation with a line
  of defects}, Ann. Probab. \textbf{41} (2013), no.~3B, 2013 -- 2046.

\bibitem[GJ10]{grimmett2010random}
G.~Grimmett and S.~Janson, \emph{Random graphs with forbidden vertex degrees},
  Random Struct. Algorithms. \textbf{37} (2010), no.~2, 137--175.

\bibitem[GL17]{grimmett20171}
G.~Grimmett and Z.~Li, \emph{The 1-2 model}, Contemp. Math \textbf{969} (2017),
  139--152.

\bibitem[GM16]{gouere2016nonoptimality}
J.-B. Gou{\'e}r{\'e} and R.~Marchand, \emph{Nonoptimality of constant radii in
  high dimensional continuum percolation}, Ann. Appl. Probab. \textbf{44}
  (2016), no.~1, 307--323.

\bibitem[Gou14]{gouere2014percolation}
J.-B. Gou{\'e}r{\'e}, \emph{Percolation in a multiscale {B}oolean model}, ALEA
  Lat. Am. J. Probab. Math. Stat. (2014), 11--1.

\bibitem[GP17]{GhoshPeres}
S.~Ghosh and Y.~Peres, \emph{Rigidity and tolerance in point processes:
  {Gaussian} zeros and {Ginibre} eigenvalues}, Duke Math. J. \textbf{166}
  (2017), no.~10, 1789--1858.

\bibitem[Gri99]{grimmett1999percolation}
G.~Grimmett, \emph{Percolation}, Springer, 1999.

\bibitem[Gri01]{grimmett2001infinite}
\bysame, \emph{Infinite paths in randomly oriented lattices}, Random Struct.
  Algorithms. \textbf{18} (2001), no.~3, 257--266.

\bibitem[Gri06]{grimmett2006random}
\bysame, \emph{The random-cluster model}, vol. 333, Springer, 2006.

\bibitem[HL21]{holroyd2021constrained}
A.~Holroyd and Z.~Li, \emph{Constrained percolation in two dimensions}, Ann.
  Inst. Henri Poincar{\'e} D \textbf{8} (2021), no.~3, 323--375.

\bibitem[HM96]{HaggMee96}
O.~H{\"a}ggstr{\"o}m and R.~Meester, \emph{Nearest neighbor and hard sphere
  models in continuum percolation}, Random Struct. Algorithms \textbf{9}
  (1996), no.~3, 295--315.

\bibitem[HS13]{HolroydSoo}
A.~Holroyd and T.~Soo, \emph{Insertion and deletion tolerance of point
  processes}, Electron. J. Probab. \textbf{18} (2013), 24, Id/No 74.

\bibitem[HS14]{holmes2014degenerate}
M.~Holmes and Th. Salisbury, \emph{Degenerate random environments}, Random
  Struct. Algorithms. \textbf{45} (2014), no.~1, 111--137.

\bibitem[HS21a]{Holmes}
\bysame, \emph{Phase transitions for degenerate random environments}, ALEA Lat.
  Am. J. Probab. Math. Stat. \textbf{18} (2021), 707 -- 725.

\bibitem[HS21b]{10.1214/20-AOP1476}
\bysame, \emph{A shape theorem for the orthant model}, J. Stat. Phys.
  \textbf{49} (2021), no.~3, 1237 -- 1256.

\bibitem[IJvRM15]{iliev2015phase}
G.~Iliev, E.~Janse~van Rensburg, and N.~Madras, \emph{Phase diagram of
  inhomogeneous percolation with a defect plane}, J. Stat. Phys. \textbf{158}
  (2015), 255--299.

\bibitem[Jac15]{jacobsen2015critical}
J.~Jacobsen, \emph{Critical points of {P}otts and {O(N)} models from eigenvalue
  identities in periodic temperley--lieb algebras}, J. Phys. A: Math. Theor.
  \textbf{48} (2015), no.~45, 454003.

\bibitem[JKLT23]{KJLT}
B.~Jahnel, J.~K{\"o}ppl, B.~Lodewijks, and A.~T{\'o}bi{\'a}s, \emph{Percolation
  in lattice $ k $-neighbor graphs}, arXiv preprint arXiv:2306.14888 (2023).

\bibitem[KOS06]{kenyon2006dimers}
R.~Kenyon, A.~Okounkov, and S.~Sheffield, \emph{Dimers and amoebae}, Ann. Math.
  (2006), 1019--1056.

\bibitem[LSS97]{liggett1997domination}
Thomas~M Liggett, Roberto~H Schonmann, and Alan~M Stacey, \emph{Domination by
  product measures}, The Annals of Probability \textbf{25} (1997), no.~1,
  71--95.

\bibitem[NW97]{newman1997percolation}
Ch. Newman and C.~Wu, \emph{Percolation and contact processes with
  low-dimensional inhomogeneity}, Ann. Probab. \textbf{25} (1997), no.~4,
  1832--1845.

\bibitem[PS14]{PeresSly}
Y.~Peres and A.~Sly, \emph{Rigidity and tolerance for perturbed lattices},
  arXiv preprint arXiv:1409.4490 (2014).

\bibitem[PT00]{PonTitt00}
A.~P{\"o}nitz and P.~Tittmann, \emph{Improved upper bounds for self-avoiding
  walks in $\mathbb{Z}^d$}, Electron. J. Comb. \textbf{7} (2000), R21--R21.

\bibitem[QZ07]{quintanilla2007asymmetry}
J.~Quintanilla and R.~Ziff, \emph{Asymmetry in the percolation thresholds of
  fully penetrable disks with two different radii}, Phys. Rev. E: Stat. Nonlin.
  Soft Matter Phys. \textbf{76} (2007), no.~5, 051115.

\bibitem[vdBE96]{van1996new}
J.~van~den Berg and A.~Ermakov, \emph{A new lower bound for the critical
  probability of site percolation on the square lattice}, Random Struct.
  Algorithms. \textbf{8} (1996), no.~3, 199--212.

\bibitem[Wie89]{wierman1989ab}
J.~Wierman, \emph{{AB} percolation: a brief survey}, Banach Cent. Publ.
  \textbf{25} (1989), no.~1, 241--251.

\bibitem[Zha94]{zhang1994note}
Y.~Zhang, \emph{A note on inhomogeneous percolation}, Ann. Probab. (1994),
  803--819.

\end{thebibliography}
\appendix

\end{document}